\documentclass[10pt, english]{amsart}

\usepackage{amsmath,amssymb,enumerate}

\usepackage[T1]{fontenc}
\usepackage[utf8]{inputenc}
\usepackage[all]{xy}

\usepackage{babel}
\usepackage{amstext}
\usepackage{amsmath}
\usepackage{amsfonts}
\usepackage{latexsym}
\usepackage{ifthen}
\usepackage{multirow}
\usepackage{graphicx}
\usepackage{tikz-cd}
\usepackage{tikz}
\usetikzlibrary{calc}
\usepackage{float}

\xyoption{all}
\newcommand\defect{{\rm def}}

\newcommand\sZ{{\mathcal Z}}

\newcommand{\rk}{{\rm rk}}
\newcommand{\codim}{{\rm codim}}

\newcommand{\sK}{{\mathcal K}}

\newtheorem{lemma1}{}[section]

\newenvironment{lemma}{\begin{lemma1}{\bf Lemma.}}{\end{lemma1}}

\newenvironment{theorem}{\begin{lemma1}{\bf Theorem.}}{\end{lemma1}}
\newenvironment{proposition}{\begin{lemma1}{\bf Proposition.}}{\end{lemma1}}
\newenvironment{corollary}{\begin{lemma1}{\bf Corollary.}}{\end{lemma1}}
\newenvironment{remark}{\begin{lemma1}{\bf Remark.}\rm}{\end{lemma1}}
\newenvironment{remarks}{\begin{lemma1}{\bf Remarks.}\rm}{\end{lemma1}}
\newenvironment{definition}{\begin{lemma1}{\bf Definition.}}{\end{lemma1}}
\newenvironment{notation}{\begin{lemma1}{\bf Notation.}}{\end{lemma1}}

\newenvironment{setup}{\begin{lemma1}{\bf Setup.}}{\end{lemma1}}

\newenvironment{assumption}{\begin{lemma1}{\bf Assumption.}}{\end{lemma1}}

\newenvironment{remark*}{{\bf Remark.}}{}
\newenvironment{example*}{{\bf Example.}}{}
\newenvironment{assumption*}{{\bf Assumption.}}{}
\newenvironment{claim*}{{\bf Claim.}}{}

\newcommand{\R}{\ensuremath{\mathbb{R}}}
\newcommand{\Q}{\ensuremath{\mathbb{Q}}}

\newcommand{\C}{\ensuremath{\mathbb{C}}}

\newcommand{\PP}{\ensuremath{\mathbb{P}}}

\usepackage{eurosym}

\newcommand{\holom}[3]{\ensuremath{#1:#2  \rightarrow #3}}
\newcommand{\fibre}[2]{\ensuremath{#1^{-1} (#2)}}

\makeatletter
\ifnum\@ptsize=0  \fi \ifnum\@ptsize=2  \fi 

\newcommand\sF{{\mathcal F}}

\newcommand\sH{{\mathcal H}}
\newcommand\sI{{\mathcal I}}

\newcommand\sO{{\mathcal O}}

\newcommand\sS{{\mathcal S}}
\newcommand\sU{{\mathcal U}}

\newcommand\sC{{\mathcal C}}

\newcommand\sD{{\mathcal D}}

\DeclareMathOperator*{\pic}{Pic}

\DeclareMathOperator*{\red}{red}

\DeclareMathOperator*{\Bir}{Bir}
\DeclareMathOperator*{\Mult}{Mult}
\DeclareMathOperator*{\Inf}{Inf}
\DeclareMathOperator*{\Fin}{Fin}

\newcommand{\Chow}[1]{\ensuremath{\mbox{\rm Chow}(#1)}}

\DeclareMathOperator*{\supp}{Supp}

\newcommand{\Univ}{\ensuremath{\mathcal{U}}}

\usepackage{enumitem}
\setlist[itemize]{leftmargin=*}
\setlist[enumerate]{leftmargin=*}
\title{Fano manifolds with big tangent bundle: a characterisation of $V_5$} 
\date{\today}

\author{Andreas H\"oring}
\author{Jie Liu}

\address{Andreas H\"oring, Universit\'e C\^ote d'Azur, CNRS, LJAD, France, Institut Universitaire de France}
\email{Andreas.Hoering@univ-cotedazur.fr}

\address{Jie Liu, Institute of Mathematics, Academy of Mathematics and Systems Science, Chinese Academy of Sciences, Beijing, 100190, China}
\email{jliu@amss.ac.cn}


\begin{document}

\begin{abstract}  
Let $X$ be a Fano manifold with Picard number one such that the tangent bundle $T_X$ is big. If $X$ admits a rational curve with trivial normal bundle, we show that $X$ is isomorphic to the del Pezzo threefold of degree five.
\end{abstract}

\maketitle

\section{Introduction}

\subsection{Main results}

Fano manifold are naturally classified in terms of their anticanonical class $-K_X$, but making an assumption about the positivity of the tangent bundle allows to identify Fanos with a particularly rich geometry. This idea already appears in Mori's characterisation of projective manifolds with ample tangent bundle \cite{Mor79} and has been pursued by many authors over the last fourty years (\cite{Mok88, CP91, DPS94, MOSW15, Mat18, Iwa18, HIM19} to cite just a few). The common denominator of these papers is that the tangent bundle is assumed to be ``semipositive'' over a Zariski open subset $\emptyset \neq X_0 \subset X$. In this paper we go in a direction that is closer to the idea of higher Fano manifolds \cite{DJS07, AC12}: we assume that the tangent bundle is big, i.e., the canonical class $\zeta:= \sO_{\PP(X)}(1)$ on the projectivised tangent bundle $\PP(T_X)$ is a big divisor class. While this implies that $\zeta$ is ``semipositive'' over some Zariski open subset $\emptyset \neq U_0 \subset \PP(T_X)$, its negative locus 
typically maps surjectively onto $X$.
Nevertheless the examples considered in our earlier work with F. Shao \cite{HLS20} indicate that assuming bigness should lead to strong restrictions on the manifold.
We prove the first classification result for Fano manifolds with big tangent bundles:

\begin{theorem}
	\label{thm:big-finite-VMRT}
	Let $X$ be a Fano manifold with Picard number one.
	Assume that the tangent bundle $T_X$ is big. Then $X$ contains 
	 a rational curve with trivial normal bundle if and only if $X$ is isomorphic to the quintic del Pezzo threefold~$V_5$, i.e., a smooth codimension $3$ linear section of the Grassmannian ${\rm Gr}(2,5)\subset \PP^9$ in its Pl\"ucker embedding.
\end{theorem}

Note that the assumption on the Picard number excludes trivial examples like $\PP^1 \times Y$ and allows us to focus on the most interesting case.
We had shown in \cite[Thm.1.5]{HLS20} that $V_5$ is the only prime del Pezzo threefold
having big tangent bundle. The key result of this paper is that the presence of a family of minimal rational curves of anticanonical degree two forces the dimension of $X$ to be small. Combining Theorem \ref{thm:big-finite-VMRT} with the classification of prime Fano threefolds due to Iskovskikh, one easily deduces:

\begin{corollary}
	\label{c.threefolds}
	Let $X$ be a smooth Fano threefold with Picard number one. Then $T_X$ is big if and only if $X$ is isomorphic to one of the following varieties:
	\begin{enumerate}
		\item the projective space $\PP^3$;
		
		\item the smooth quadric threefold $Q^3\subset \PP^4$;
		
		\item the quintic del Pezzo threefold $V_5$.
	\end{enumerate}
\end{corollary}

In view of Theorem \ref{thm:big-finite-VMRT} and the examples considered in \cite{HLS20},
we expect that prime Fano manifolds with big tangent bundle should have a rather high index. Our computations yield a first step in this direction:

\begin{corollary}
	\label{c.indexone}
	Let $X$ be a Fano manifold with Picard number one such that $T_X$ is big. Assume that there exists a covering family $\sK$ of minimal rational curves on $X$ such that the VMRT $\sC_x\subset \PP(\Omega_{X,x})$ at a general point $x\in X$ is not dual defective. Then $X$ has index at least two.
\end{corollary}

\subsection{Strategy of proof of Theorem \ref{thm:big-finite-VMRT}}

Let $X$ be a Fano manifold with Picard number one such that $T_X$ is big, and let $\sK$ be a family
of minimal rational curves such that $-K_X \cdot l=2$ for a curve $[l] \in \sK$.
The starting point of our proof is the total dual VMRT $\check{\mathcal C} \subset \PP(T_X)$ associated to the family $\sK$ (cf. Definition \ref{Def:Dual-VMRT}): over a general point $x \in X$ the fibre $\check{\mathcal C}_x$
is a finite union of hyperplanes, one for each curve in $\sK$ passing through the point $x$. We observed in \cite{HLS20} that the bigness assumption leads to restrictions
on the cohomology class of $\check{\mathcal C} \subset \PP(T_X)$.
In fact Proposition \ref{Prop:Degree-VMRT} shows
that $X$ has index two, i.e., we have $-K_X = 2 A$ with $A$ an ample Cartier divisor. Thus $A \cdot l=1$ (the rational curve $l$ is a ``line''), and the family $\sK$ is unsplit. 

Eventually we want to consider the saturated subsheaf $\sF \subset T_X$ defined by the linear span of the tangent directions
of the minimal rational curves in the family $\sK$ (see \cite[Sect.2]{Hwa01} for more details).
Our goal is to show that this distribution 
has rank three and $c_1(\sF)=c_1(-K_X)$. If $X$ has dimension at least four, a theorem of Peternell \cite{Pet12} on the generic ampleness of $T_X$ then yields a contradiction. 

Proving that $\rk \sF=3$ requires a new tool that should be of independent interest: smoothing two curves
of $\sK$ passing through a general point, we obtain a family of rational curves $\sH$
on $X$ such that $A \cdot C=2$ for a general member $[C]\in \sH$ (the family $\sH$ is a family of ``conics''), and we consider its total variety of tangents $\sD \subset \PP(\Omega_X)$ (cf. Subsection \ref{subsectionVOT}). The main technical result of this paper (Theorem \ref{thm:irreduciblity-tangent-H}) is that in our case the general fibre $\sD_x \subset \PP(\Omega_{X,x})$ is a linear plane! The proof of this surprising fact comes in two steps: in the first step we observe that the bigness assumption and the projective duality imply that the
irreducible components of $\sD_x \subset \PP(\Omega_{X,x})$ are two dimensional cones
with vertices the points of the VMRT $\sC_x \subset \PP(\Omega_{X,x})$.
The second, much more difficult, step is to show that these cones are actual linear planes. 
For this second step we consider the tangent map
$$
\tau_x: \sH_x \dashrightarrow \PP(\Omega_{X,x})
$$
defined on the (normalised) family of curves parametrised by $\sH$ passing through a general point $x$. Contrary to the case of minimal rational curves,
the tangent map $\tau_x$ always has indeterminacies (cf. Lemma \ref{lemmaindeterminacy}) and is not always birational onto its image
\cite[Prop.3.2]{IN03}. Our basic idea is to turn these pathologies to our advantage, using that we already know that the image, the tangent variety $\sD_x$, is a union of cones. 

Unfortunately the implementation of this idea is not straightforward, for two  reasons:
\begin{itemize}
    \item While a general member $l$ of $\sK$ is smooth in a general point of $X$ \cite{Keb02}, it might have singularities in some special point. 
    \item For a general member $l$ of $\sK$, denote by $B_l \subset \sK$ the members of $\sK$ that intersect $l$. Then $B_l$ might have several irreducible components.
\end{itemize}
A posteriori these problems disappear: on the del Pezzo threefold $V_5$ the family of minimal rational curves is a family of lines, and the curve $B_l$ is a smooth conic in $\sK \simeq \PP^2$. A priori, things are less obvious, and we have to be very careful at every step of the construction, building on ideas of Hwang \cite{Hwa17a} for the analysis
of rational curves with trivial normal bundle. Indeed the description of the family of minimal rational curves in Subsection \ref{subsectionbig} is the first big step of our proof. For the convenience of the reader we have summarised the results in Theorem \ref{Thm:Summary-VMRT} and Corollary \ref{c.lines-B_l^1-properties} so that it is possible to read the later sections independently.

{\bf Acknowledgements.}
We would like to thank Baohua Fu, Jun-Muk Hwang and Qifeng Li for very helpful communications on various technical problems related to this subject. The first-named author thanks the Institut Universitaire de France and the A.N.R. project Foliage (ANR-16-CE40-0008) for providing excellent working conditions.
The second-named author is supported by the NSFC grants No. 11688101 and No. 12001521.

\section{Notation and basic facts}
\label{Section:basic-facts}

We work over $\C$, for general definitions we refer to \cite{Har77}. 
We use the terminology of \cite{KM98} for birational geometry and that of \cite{Ko96} for rational curves.
We refer to Lazarsfeld's books for notions of positivity of $\R$-divisors, in particular
\cite[Sect.2.2]{Laz04a} for pseudoeffectivity and the associated cones.
We will use the terminology of $\Q$-twisted vector bundles as explained in
\cite[Sect.6.2]{Laz04b}

Manifolds and varieties will always be supposed to be irreducible. A fibration is a proper surjective map with connected fibres \holom{\varphi}{X}{Y} between normal varieties.

Let $X$ be a projective variety, and let $\Chow{X}_1$ be the Chow variety of $1$-cycles on $X$. Let $\mathcal{H}$ be the normalisation of a subvariety of $\Chow{X}_1$,
and denote by $\Univ_{\sH}$ the normalisation of the pull-back of the universal family. Then we have a natural diagram
\begin{equation} \label{standarddiagram}
    \begin{tikzcd}[column sep=large, row sep=large]
	\Univ_{\sH}  \arrow[r,"e_{\sH}"] \arrow[d,"q_{\sH}"]    &   X  \\
	\sH                                                     &
\end{tikzcd}
\end{equation}

\begin{definition}
	Let $X$ be a projective variety.
	\begin{enumerate}
		\item The normalisation $\sH$ of an irreducible projective subvariety of $\Chow{X}_1$ is called an irreducible covering family of curves on $X$ if the following conditions holds:
		\begin{enumerate}
			\item a general fiber of $q_{\sH}$ is irreducible and reduced, and
			\item the evalution morphism $e_{\sH}$ is surjective.
		\end{enumerate}
		
		\item A covering family of curves on $X$ is a finite union of irreducible covering family of curves.
	\end{enumerate}
\end{definition}

In this paper we will frequently study the behaviour of covering families under certain morphisms.
We follow the definitions of \cite{Hwa17a}, which differ from the classical definitions of pull-back and push-forward:

\begin{definition}
	\label{Defn:pull-back-forward-web}
	Let $X$ be a projective variety and let $\sH$ be a covering family of curves over $X$.
	\begin{enumerate}
		\item Let $g\colon X\rightarrow X'$ be a surjective morphism to a projective variety $X'$ such that $g$ does not contract general members of irreducible components of $\sH$. The images under $g$ of the general members of irreducible components of $\sH$ determine a covering family of curves on $X'$ which we denote by  $g_*\sH \rightarrow \Chow{X'}_1$. We have a natural commutative diagram of dominant rational maps
		\[
		\begin{tikzcd}[column sep=large, row sep=large]
			\Univ_{\sH} \arrow[r,dashed] \arrow[d,"q_{\sH}" left]  &  \Univ_{g_*\sH} \arrow[d,"q_{g_*\sH}"] \\
			\sH         \arrow[r,dashed]                      &  g_*\sH
		\end{tikzcd}
		\]
		
		\item Let $g\colon X\rightarrow X'$ be a surjective morphism to a projective variety $X'$. We define a decomposition 
		$$
		\sH={\rm Fib}_g(\sH)\cup {\rm Hor}_g(\sH)
		$$
		with no common components as follows. An irreducible component $\sH_i$ of $\sH$ belongs to ${\rm Fib}_g(\sH)$ if and only if the general members in $\sH_i$ are contracted by $g$. Moreover, we can also define a decomposition 
		$$
		{\rm Hor}_g(\sH)={\Inf}_g(\sH)\cup {\Fin}_g(\sH)
		$$
		with no common components as follows. An irreducible component $\sH_i$ of ${\rm Hor}_g(\sH)$ belongs to $\Fin_g(\sH)$ (resp. $\Inf_g(\sH)$) if and only if the induced morphism $\sH_i\dashrightarrow g_*\sH_i$ is generically finite (resp. not generically finite). 
		
		\item Let $f\colon X'\rightarrow X$ be a surjective generically finite morphism from a projective variety $X'$. The inverse images under $f$ of the general members of irreducible components of $\sH$ determines a covering family of curves on $X'$ which we will denote by $f^*\sH\rightarrow \Chow{X'}_1$. We have a natural commutative diagram whose horizontal arrows are generically finite rational maps
		\[
		\begin{tikzcd}[column sep=large, row sep=large]
			\Univ_{f^*\sH} \arrow[r,dashed] \arrow[d,"q_{f^*\sH}" left]  &  \Univ_{\sH} \arrow[d,"q_{\sH}"] \\
			f^*\sH          \arrow[r,dashed]                              &  \sH
		\end{tikzcd}
		\]
		We shall denote by $\Bir(f^*\sH)$ (resp. $\Mult(f^*\sH)$) the union of irreducible components of $f^*\sH$ whose general members are sent to members of $\sH$ birationally (resp. not birationally) by $f$, so that 
		$$
		f^*\sH=\Bir(f^*\sH)\cup \Mult(f^*\sH).
		$$
	\end{enumerate}
\end{definition}

\begin{remark*}
In order to avoid confusion, let us stress that in the definitions of $g_*\sH$ and $f^*\sH$, the images and inverse images are taken in the set-theoretical sense, not in the cycle-theoretic sense. This choice is coherent with the definition in our main reference \cite{Hwa17a}.
\end{remark*}

\section{Webs of algebraic curves}
\label{Section:web-curves}

\begin{definition}
	Let $X$ be a projective variety. 
	\begin{enumerate}
		\item An irreducible covering family of curves $\sH$ on $X$ is called an irreducible web of curves if the evaluation morphism $e_{\sH}:\Univ_{\sH}\rightarrow X$ is generically finite.
		
		\item A web of curves on $X$ is a finite union of irreducible webs of curves. In particular, a web of curves on $X$ has pure dimension equal to $\dim(X)-1$.
		
		\item The degree of an irreducible web $\sH$ of curves on $X$ is defined to be the degree of the evaluation morphism $e_{\sH}$.
	\end{enumerate}
\end{definition}

The term \emph{web} originates from the notion of a web-structure in local differential geometry, defined as follows.

\begin{definition}
	\label{defn:webstructure}
	Let $U$ be a complex manifold. A web-structure (of rank $1$) on $U$ is a finite collection of line subbundles
	\[
	W_i\subset T_U,\quad 1\leq i\leq d,
	\]
	for some integer $d\geq 1$ such that for any $1\leq i\not=j\leq d$, the intersection $W_i\cap W_j\subset T_U$ is the zero section.
\end{definition}

An irreducible web $\sH$ of curves with degree $d$ on a projective variety $X$ will define a web-structure at a general point of $X$ in the following sense \cite[Proposition 3.5]{Hwa17a}: there exists a dense Zariski-open subset $X_o \subset X$ such that each point $x\in X_o$ has an Euclidean open neighborhood $U_x \subset X_o$ satisfying the following conditions:
\begin{enumerate}
	\item the preimage $e_{\sH}^{-1}(U_x)$ consists of $d$ disjoint connected open subsets
	\[
	V_x^1,\dots,V_x^d\subset \Univ_{\sH}^{\rm sm},
	\]
	each of which is biholomorphic to $U_x$ by $e_{\sH}$;
	
	\item for each $1\leq i\leq d$, each fibre of $q_{\sH}\vert_{V_x^i}\colon V_x^i\rightarrow q_{\sH}(V_x^i)$ is smooth and connected;
	
	\item let $W_i\subset T_{U_x}$, $1\leq i\leq d$, be the image of $T_{V^i_x/q_{\sH}(V^i_x)}$ ; then $W_1,\dots,W_d$ give a web-structure on $U_x$.
\end{enumerate}
In particular, for any member $C$ of $\sH$, the intersection $C\cap U_x$ is smooth.

Let $X$ be a projective manifold of dimension $n$. An irreducible and reduced curve $C\subset X$ has a \emph{trivial normal bundle} if under the normalisation $f\colon \widetilde{C}\rightarrow C$, we get the following exact sequence
\[
0\rightarrow T_{\widetilde{C}}\rightarrow f^*T_X\rightarrow \sO_{\widetilde{C}}^{\oplus (n-1)}\rightarrow 0.
\]
In particular the morphism $\widetilde{C}  \rightarrow X$ is immersive.
Suppose that a curve $C$ in an $n$-dimensional projective manifold $X$ has trivial normal bundle and deformations of $C$ with constant geometric genus cover an open subset of $X$. The germ of the space of deformations of $C$ with constant geometric genus must have dimension at least $n-1$. The Zariski tangent space to this space at the point corresponding to $C$ is $$
H^0(\widetilde{C},f^*T_X/T_{\widetilde{C}}) \simeq H^0(\widetilde{C}, \sO_{\widetilde{C}}^{\oplus (n-1)}) \simeq \C^{n-1}.
$$
Thus the germ is smooth. The normalisation $\sH$ of the closure of this germ in the Chow variety of $X$ will define an irreducible web of curves.

\begin{definition}
	\label{defn:etale-web-curves}
	Let $X$ be a projective manifold. An \'etale web of curves on $X$ is an irreducible web  of curves $\sH$ on $X$ such that the following holds: in terms of the normalised universal family 
	\[
	\begin{tikzcd}[column sep=large, row sep=large]
		\Univ_{\sH} \arrow[r,"e_{\sH}"] \arrow[d,"q_{\sH}" left] & X  \\
		\sH
	\end{tikzcd}
	\]
	there exists a Zariski open subset $\sH_o$ of the smooth locus of $\sH$ such that for each point $t\in \sH_o$,
	\begin{enumerate}
		\item the fibre $(q_{\sH})^{-1}(t)$ is a smooth curve, and
		
		\item the morphism $e_{\sH}$ is unramified over the Zariski open subset $(q_{\sH})^{-1}(\sH_o)$.
	\end{enumerate}
\end{definition}

Let $\sH$ be an \'etale web of curves on a projective manifold $X$, and let $Z\subset X$ be a closed subset with codimension at least two. Since $e_{\sH}$ is unramified in a neighbourhood of a general fibre of $q_{\sH}$, then an easy dimension count shows that a general member of $\sH$ is disjoint from $Z$.

\begin{remark}
	Our definition of \'etale webs of curves is \emph{a priori} slightly weaker than the definition of \'etale webs of \emph{smooth} curves given in \cite[Definition 6.1]{Hwa17a} since we do not require 
	that $e_\sH(\fibre{q_\sH}{t}) \subset X$ is smooth for general $t \in \sH$.
However, many results proved for \'etale webs of smooth curves in \cite{Hwa17a} still hold in this more general situation.
\end{remark}

One can easily see that the irreducible web of curves defined by a curve $C$ with trivial normal bundle in a projective manifold $X$ is an \'etale web of curves. Conversely, if $\sH$ is an \'etale web of curves, then for a general point $t\in \sH$, the curve $C_t\subset X$ corresponding to $t$ has trivial normal bundle since the natural induced morphism
\[
(q_{\sH})^{-1}(t)\longrightarrow C_t
\]
is birational and finite, and hence it coincides with the normalisation of $C_t$. 

\begin{definition} \label{definition-transversally}
Let $X$ be a complex manifold, and let $C \subset X$ be a connected curve.
Let $B \subset X$ be a hypersurface. We say that the intersection 
$C\cap B$ is transversal if for every point $p \in C \cap B$ and every local irreducible germ $C'$ of $C$ at $p$, the germ $C'$ is smooth and not tangent to $B$.
\end{definition}

 The following lemma is stated for \'etale web of smooth curves in \cite{Hwa17a}, but the proof still works for \'etale web of curves (see also \cite[Proposition 1 and Proposition 6]{HM03}).

\begin{lemma}
	\cite[Lemma 6.2]{Hwa17a}
	\label{lemma:transversal-locus}
	Let $\sH$ be an irreducible \'etale web of curves on a projective manifold $X$. Let $B\subset X$ be an irreducible hypersurface which has positive intersection number with members of $\sH$. Let $B_o\subset B$ be a dense Zariski open subset. Then there exist dense Zariski open subsets $B'\subset B_o$ and $\sH_B\subset \sH_o$ such that 
	\begin{enumerate}
		\item for any member $[C]\in \sH_B$, the curve $C$ intersects $B$ transversally and $C\cap B$ is contained in $B'$; and
		
		\item for any point $x\in B'$, there exists $t\in \sH_B$ such that $x\in C_t$; and
		
		\item if a curve $C\subset X$ is a member in $\sH$ such that $C\cap B'\not=\emptyset$, then $C=C_t$ for some point $t\in \sH_B$.
	\end{enumerate}
\end{lemma}

\begin{proof} We use the notation of Definition \ref{defn:etale-web-curves}.
Since the intersection number is positive, the induced morphism 
	\[
	(e_{\sH})^{-1}(B)\longrightarrow \sH
	\]
	is surjective and generically finite. Thus we may choose a dense Zariski open subset $\sH_1\subset \sH_o$ such that the restricted morphism
	\[
	\bar{B}_1:=(e_{\sH})^{-1}(B)\cap (q_{\sH})^{-1}(\sH_1)\longrightarrow \sH_1
	\]
	is surjective and \'etale. Then we can choose a Zariski open subset $B'\subset B_o$ such that $(e_{\sH})^{-1}(B')\subset \bar{B}_1$. Set $\sH_B:=q_{\sH}((e_{\sH})^{-1}(B'))$. Then they satisfy the required conditions.
\end{proof}

\begin{proposition}
	\label{Prop:Pull-Back-Trivial-Normal-Bundle}
	\cite[Proposition 6]{HM03}
	Let $f\colon Y\rightarrow X$ be a generically finite morphism between projective manifolds. Suppose that there exists an \'etale web  of curves $\sH$ on $X$ and let $C\subset X$ be a general member in $\sH$. Then every irreducible component of $f^{-1}(C)$ has trivial normal bundle. In particular, the covering family of curves $f^*\sH$ is a finite union of \'etale webs of curves on $Y$.
\end{proposition}

Let $f\colon Y\rightarrow X$ be a surjective generically finite morphism from an irreducible normal projective variety $Y$ to a projective manifold $X$. We consider the set
\begin{center}
	$\{y\in Y\,|\,Y$ is singular at $y$ or $f$ is not a local biholomorphism near $y\}$.
\end{center}
The \emph{ramification divisor} $R$ of $f$ is the union of codimension one components of the above set which are not contracted to lower dimensional subvarieties of $X$ by $f$. The \emph{branch divisor} of $f$ is the image of the ramification divisor $R$ of $f$, denoted by $B\subset X$. Then $f$ is \'etale over $X\setminus(B\cup Z)$ for some subvariety $Z$ of codimension at least two in $X$.

The following proposition is proved in \cite[Proposition 6.7]{Hwa17a} for \'etale webs of smooth curves and we follow the same argument (see also \cite[Lemma 4.2]{HM01}).

\begin{proposition}
	\label{prop:existence-multiple-component}
	\cite[Proposition 6.7]{Hwa17a}
	Let $X$ be a simply connected projective manifold equipped with an \'etale web $\sH$ of curves. Let $Y$ be an irreducible normal projective variety, and let $f\colon Y\rightarrow X$ be a generically finite morphism of degree at least two. Let $R$ be the ramification divisor of $f$. Let $C\subset X$ be a general member of $\sH$ and let $C'$ be an irreducible member of $f^{-1}(C)$. If the intersection $C'\cap R$ is not empty, then the induced morphism $C'\rightarrow C$ is not birational.  In particular, if $C'$ is an irreducible component of $f^{-1}(C)$ such that $C'\rightarrow C$ is birational, then $f$ is unramified in a neighborhood of $C'$.
\end{proposition}

\begin{proof}
	As $X$ is simply connected, the branch divisor $B$ is not empty. On the other hand, we can choose a subvariety $Z\subset X$ with codimension at least two such that $f$ is unramified over $f^{-1}(X\setminus(B\cup Z))$. Moreover, since $C$ is supposed to be general, we may also assume that $C$ is disjoint from $Z$. As $C'\cap R\not=\emptyset$, we have $C\cap B\not=\emptyset$. Thus, by Lemma \ref{lemma:transversal-locus} and the generic choice of $C$, we may assume that $C$ intersects $B$ transversally (see Definition \ref{definition-transversally}). Let $y\in C'\cap R$ be a point and denote $f(y)$ by $x$. Then $C$ is immersed at $x$ and $C'$ is immersed at $y$ (see \cite[Lemma 1]{HM03}). Since $C'$ is not disjoint from $R$, it follows that the induced morphism
	\[
	\widetilde{f}\colon \widetilde{C}'\longrightarrow \widetilde{C}
	\]
	is ramified at a point $z\in \widetilde{C}'$ such that $n'(z)=y$, where $n'\colon \widetilde{C}'\rightarrow C'$ and $n\colon \widetilde{C}\rightarrow C$ are normalisations. As a consequence, the morphism $\widetilde{f}$ is not birational and hence $C'\rightarrow C$ is not birational.
\end{proof}

As an immediate application, we obtain:

\begin{corollary}
	\label{cor:existence-mult-components}
	\cite[Proposition 6.7 and Proposition 6.8]{Hwa17a}
	Let $\sH$ be an \'etale web of curves on a simply connected projective manifold $X$ of Picard number one. Let $f: Y\rightarrow X$ be a generically finite morphism of degree at least two from an irreducible normal projective variety $Y$. Then there exists at least one irreducible component of $f^{-1}(C)$ which is not birational onto $C$, i.e., $\Mult({f^*\sH})$ is not empty.
\end{corollary}

\begin{proof}
	Since $X$ is simply connected and the degree of $f$ is at least two, the branch locus $B$ and the ramification divisor $R$ are not empty. Moreover, as $X$ has Picard number one, a general curve $C$ intersects the branch locus $B$ and is not contained in it.
Thus there exists an irreducible component $C'$ of $f^{-1}(C)$ such that $C'\cap R\not=\emptyset$. Then we conclude by Proposition \ref{prop:existence-multiple-component}.
\end{proof}

	Following the argument in \cite[Proposition 6.6]{Hwa17a} word-by-word and applying \cite[Corollary 3.12]{Hwa17a}, one obtains:

\begin{proposition}
	\label{prop:finiteness-multiple-components}
	\cite[Proposition 6.6]{Hwa17a}
	Let $X$ be a projective manifold and let $\sH$ be an \'etale web of curves on $X$. Using the terminology of Definition \ref{Defn:pull-back-forward-web}, the set $\Mult(e_{\sH}^*\sH)$ is contained in $\Fin_{q_{\sH}}(e_{\sH}^*\sH)$.
\end{proposition}

\section{Varieties of tangents of \'etale webs of curves}
\label{Section:VOT}

\subsection{Variety of tangents}
\label{subsectionVOT}

Let $X$ be a projective manifold. Let $\sH$ be an irreducible covering family of curves on $X$. Denote by $\holom{e_\sH}{\Univ_{\sH}}{X}$ and $\holom{q_\sH}{\Univ_{\sH}}{\sH}$
the natural maps on the normalised universal family.
We have the following Cartesian diagram
\[
\begin{tikzcd}[column sep=large, row sep=large]
	\PP(e_{\sH}^*\Omega_X) \arrow[r,"\widetilde{e}_{\sH}"] \arrow[d] & \PP(\Omega_X) \arrow[d]  \\
	\Univ_{\sH} \arrow[r,"e_{\sH}"]                                    & X
\end{tikzcd}
\]
Let $\sS$ be the image of the composition
\[
e_{\sH}^*\Omega_X \longrightarrow \Omega_{\Univ_{\sH}} \longrightarrow \Omega_{\Univ_{\sH}/\sH}.
\]
Denote by $\sD'\subset \PP(e_{\sH}^*\Omega_X)$ the unique irreducible component of $\PP(\sS)$ dominating $\Univ_{\sH}$. Then $\sD'$ is a rational section of the projectivised bundle $\PP(e_{\sH}^*\Omega_X)\longrightarrow \Univ_{\sH}$.

\begin{definition}
	In the situation above, the tangent variety of $\sH$ is defined to be 
	$$
	\widetilde{e}_{\sH}(\sD')=: \sD \subset \PP(\Omega_X).
	$$
    The induced rational map
	\[
	\begin{tikzcd}[column sep=large]
		\tau_{\sH}: \Univ_{\sH} \arrow[r,dashed]  & \sD\subset \PP(\Omega_X)
	\end{tikzcd}
	\]
	is called the tangent map (associated to $\sH$). In more geometric terms, for a general point $y \in \Univ_{\sH}$, the tangent map $\tau_{\sH}$ sends the point $y$ to $\PP(\Omega_{e_{\sH}(C_y),e_{\sH}(y)})$, where $C_y$ is the unique $q_{\sH}$-fibre passing through $y$.
\end{definition}

Let $x\in X$ be a general point. Denote by $\sH_x$ the fibre of $e_{\sH}$ over $x$. Then $\sH_x$ is a normal (maybe reducible) projective scheme. Moreover, since $x \in X$ is general, a general curve passing through $x$ is smooth at $x$. Thus the induced morphism
\[
\sH_x\longrightarrow T_x\subset \Chow{X}_1
\]
maps birationally onto its image, where $T_x$ is the subvariety of $\Chow{X}$ parametrising members in $\sH$ passing through $x$. On the other hand, as one can easily see that the birational morphism $\sH_x\rightarrow T_x$ is also finite, thus $\sH_x$ is actually the normalisation of $T_x$. For this reason, we will call $\sH_x$ the \emph{normalised space parametrising curves in $\sH$ passing through $x$}.

\begin{remark} 
	\label{remarkdefinitionpoint}
	    Let $x \in X$ be a general point and let $\sH_x \subset \Univ_{\sH}$ be the normalised space parametrising curves in $\sH$ passing through $x$. The locus where the morphism
		\[
		e_{\sH}^*\Omega_X \longrightarrow \Omega_{\Univ_{\sH}/\sH}
		\]
		is not surjective has codimension at least one, so it does not contain any irreducible component of $\sH_x$.  Thus the restriction of the tangent map $\tau_{\sH}$
		to $\sH_x$ is identified generically to the tangent map
		$$
		\tau_{\sH,x} : \sH_x \dashrightarrow \PP(\Omega_{X,x}),  \ [C] \mapsto [\PP(\Omega_{C,x})]. 
		$$
		Denote by $\sD_x$ the fibre of $\sD \rightarrow X$ over the general point $x \in X$.	Then a dimension count shows that $\sD_x$ coincides with the closure of the image of $\tau_{\sH,x}$, which is usually called the \emph{tangent variety of $\sH$ at $x$}.
\end{remark}

\subsection{Tangent varieties of \'etale webs of curves}

\label{Section:Etale-Webs}

The aim of this section is to study the tangent variety of an \'etale web $\sH$ of curves along a general member $l$ parametrised by $\sH$.

\subsubsection{Setup}

\label{Section:Setup}

Let $X$ be a projective manifold of dimension at least three equipped with an \'etale web $\sH$ of curves.
Denote by $\holom{e_\sH}{\Univ_{\sH}}{X}$ and $\holom{q_\sH}{\Univ_{\sH}}{\sH}$
the natural maps on the normalised universal family.

Let $\mu:\sH^s\rightarrow \sH$ be a log resolution such that $\mu$ is an isomorphism over $\sH_o$ (see Definition \ref{defn:etale-web-curves}). Let $\widetilde{\mu}:\Univ_{\sH}^s\rightarrow \Univ_{\sH}$ be a log resolution that resolves the indeterminacies of $\mu^{-1} \circ q_\sH$ such that $\widetilde{\mu}$ is an isomorphism over the Zariski open subset $\Univ_{\sH_o}:=(q_{\sH})^{-1}(\sH_o)$.
The map $\mu^{-1}(\sH_o) \rightarrow \sH_o$ is an isomorphism, so for simplicity's sake, we will
identify $\sH_o$ with its preimage in $\mathcal H^s$.
Let $f:\Univ_{\sH}^s\rightarrow X$ be the composition
\[ 
\Univ_{\sH}^s\stackrel{\widetilde{\mu}}{\longrightarrow} \Univ_{\sH} \stackrel{e_{\sH}}{\longrightarrow} X.
\]
Denote by $\tau_{\sH}^s:\Univ_{\sH}^s\dashrightarrow \PP(\Omega_X)$ the composition 
\[
\Univ_{\sH}^s\xrightarrow{\widetilde{\mu}}\Univ_{\sH}\stackrel{\tau_{\sH}}{\dashrightarrow} \PP(\Omega_X),
\]
so we get a commutative diagram
\begin{equation}
	\label{Eq:Smooth-Universal-Family}
	\begin{tikzcd}[column sep=large, row sep=large]
		\Univ_{\sH_o} \arrow[r,hook] \arrow[d]
		   & \Univ_{\sH}^s \arrow[r,"\widetilde{\mu}"] \arrow[d,"q_{\sH}^s" left]  \arrow[rr,"f",bend left]
		      &  \Univ_{\sH} \arrow[d,"q_{\sH}"]  \arrow[r,"e_{\sH}"]
		        & X \\
		\sH_o \arrow[r,hook]
		   & \sH^s   \arrow[r,"\mu"]                                     
		      &  \sH \arrow[r,hookleftarrow]
		        & \sH_o 
	\end{tikzcd}
\end{equation}

We will now introduce the somewhat technical notation that we will use in the rest of the paper and that is necessary to carry out the proof of Theorem \ref{thm:irreduciblity-tangent-H}.

\begin{notation} 	\label{notation-hjbar-bjbar}
    In the situation above, we set $\bar{\sH}:=f^*\sH$,
    and denote by 
\[
\bar{\sH}=\bar{\sH}_0\cup \left(\bigcup_{j=1}^m \bar{\sH}_j\right)
\]
the decomposition of $\bar{\sH}$ into irreducible components such that
$\bar{\sH}_0$ is the irreducible component whose general members correspond to  $q_{\sH}^s$-fibres.

Following \cite{Hwa17a} we denote by 
$$
f^j_{\flat}:\bar{\sH}_j\dashrightarrow \sH^s
$$ 
the induced dominant rational map. 

We denote by 
$\bar{\sD}_j\subset \PP(\Omega_{\Univ_{\sH}^s})$ 
 the tangent variety of $\bar{\sH}_j$, and by
$$
\nu_j:\bar{\sD}_j\dashrightarrow \sD
$$
the dominant rational map given by the construction.
\end{notation}

\begin{remarks}
\begin{enumerate}
    \item By Proposition \ref{Prop:Pull-Back-Trivial-Normal-Bundle}, each irreducible component $\bar \sH_j$ of the inverse image $\bar \sH$ is an \'etale web of curves on $\Univ_{\sH}^s$.
    \item By definition of $\bar \sH_0$, we have ${\rm Fib}_{q_{\sH}^s}(f^*\sH)=\bar{\sH}_0$ (see Definition \ref{Defn:pull-back-forward-web}).  Moreover, by 
    Proposition \ref{prop:finiteness-multiple-components}, we know that
    	$\Mult(f^*\sH)$ is contained in $\Fin_{q_{\sH}^s}(f^*\sH)$.
    	\end{enumerate}
\end{remarks}

Let $l$ be a general member parametrised by $\sH$. Then $f^{-1}(l)$ is of pure dimension one. Thus we can write
\[
f^{-1}(l)=\bigcup_{i=0}^r C_l^i
\]
into irreducible components, where $C_l^0$ is the unique irreducible component of $f^{-1}(l)$
that is contracted by $q^s_\sH$. By Proposition \ref{Prop:Pull-Back-Trivial-Normal-Bundle}, each irreducible component $C_l^i$ has again trivial normal bundle. For each $i\geq 0$, there exists a unique $j(i)$ such that $[C_l^i]\in \bar{\sH}_{j(i)}$. Conversely, since $f^j_{\flat}$ is dominant for each $0\leq j\leq m$, there exists some $i(j)$ such that $[C_{l}^{i(j)}]\in \bar{\sH}_j$. We remark that in general $i(j)$ may be not unique.

\begin{definition}
	\label{Defn:Spaces-of-lines}
	For $1\leq i\leq r$, denote by $B_l^i$ the image $q_{\sH}^s(C_l^i)$. Set $B_l=\cup_{i=1}^r B_l^i$. Since $\sH^s\rightarrow \Chow{X}_1$ is birational onto its image and $l$ is general, the curve $B_l$ is birational onto its image in $\Chow{X}_1$. Thus we can regard $B_l$ as the space of curves parametrised by $\sH$ that meet $l$.
\end{definition}

\subsubsection{Projection along general members: local construction}

\label{Section:projection-local}

Given a general member of an \'etale web of curves, we will now describe the projection of the tangent variety
to its normal directions.
More precisely, let $[l]\in \sH$ be a general member and denote by $f_l:\widetilde{l}\rightarrow l$ the normalisation. Consider the exact sequence
\[
0\longrightarrow T_{\widetilde{l}} \longrightarrow f_l^*T_X \longrightarrow \sO_{\widetilde{l}}^{\oplus (n-1)}\longrightarrow 0,
\]
and denote by $l_0\subset \PP(f_l^*\Omega_X)$ the section corresponding to the quotient $f_l^*\Omega_X\rightarrow \Omega_{\widetilde{l}}$. Then we have a natural projection map
\[
P_{l_0}: \PP(f_l^*\Omega_X) \dashrightarrow \PP(\sO_{\widetilde{l}}^{\oplus (n-1)})=\widetilde{l}\times \PP^{n-2}.
\]
Denote by 
$$
\bar{P}_{l_0}:\PP(f_l^*\Omega_X)\dashrightarrow \PP^{n-2}$$
the composition
$P_{l_0}$ with the natural second projection  $\widetilde{l} \times \PP^{n-2} \stackrel{p_2}{\longrightarrow} \PP^{n-2}$.
Denote by $\zeta_l$ the tautological divisor class $c_1(\sO_{\PP(f_l^*\Omega_X)}(1))$ and by 
$$
\widetilde{\zeta}_l = c_1(p_2^*\sO_{\PP^{n-2}}(1))
$$ 
the tautological divisor class of $\PP(\sO_{\widetilde{l}}^{\oplus (n-1)})$. 

\begin{lemma}
	\label{Lemma:Projection-Local}
	In the situation of Setup \ref{Section:Setup}, let $[l]\in \sH$ be a general member and let $C\subset \PP(f_l^*\Omega_X)\rightarrow \widetilde{l}$ be an irreducible horizontal curve such that $C\not=l_0$. 
	Denote by $P_{l_0}(C)$ the closure of the image of the generic point of $C$.
	Then $\zeta_l\cdot C\geq 0$ and the following statements hold.
	\begin{enumerate}
		\item If $\zeta_l \cdot C=0$, then $C$ is disjoint from $l_0$ and $P_{l_0}(C)$ is a fibre of $p_2$.
		
		\item If $\zeta_l \cdot C=1$, then one of the following statements holds.
		\begin{enumerate}
			\item The intersection $C\cap l_0$ is not empty and $P_{l_0}(C)$ is a fibre of $p_2$.
			
			\item The curve $l$ is a rational curve, the curve $C$ is disjoint from $l_0$, the image $\bar{P}_{l_0}(C)$ is a line in $\PP^{n-2}$ and the induced morphism $C\longrightarrow \bar{P}_{l_0}(C)$ is an isomorphism.
		\end{enumerate}
	\end{enumerate}
\end{lemma}

\begin{proof}
	Let $\nu:Z\rightarrow\PP(f_l^*\Omega_X)$ be the blow-up along $l_0$ and let $g:Z\rightarrow \PP(\sO_{\widetilde{l}}^{\oplus (n-1)})$ be the induced morphism. Denote by $E$ the exceptional divisor of $\nu$. Then we have 
	\[
	\nu^*\zeta_l-E=g^*\widetilde{\zeta}_l.
	\]
	
	Denote by $C'$ the strict transform of $C$ in $Z$.
	Then we have $g(C')=P_{l_0}(C)$ and we obtain
	\[
	\zeta_l \cdot C = \nu^*\zeta_l\cdot C'=\deg(g|_{C'})\widetilde{\zeta}_l\cdot P_{l_0}(C)+E\cdot C'.
	\]
	
	As $\widetilde{\zeta}_l$ is nef and $C'$ is not contained in $E$, one can easily obtain that $\zeta_l\cdot C\geq 0$ with equality if and only if $\widetilde{\zeta}_l\cdot P_{l_0}(C)=E\cdot C'=0$. As a consequence, if $\zeta_l\cdot C=0$, then $P_{l_0}(C)$ is fibre of $p_2$ and $C'$ is disjoint from $E$. In particular, the curve $C$ is disjoint from $l_0$.
	
	If $\zeta_l\cdot C=1$, then we obtain either $\widetilde{\zeta}_l\cdot P_{l_0}(C)=0$ and $E\cdot C'=1$, or $\deg(g|_{C'})\widetilde{\zeta}_l\cdot P_{l_0}(C)=1$ and $E\cdot C'=0$. In the former case, the curve $P_{l_0}(C)$ is fibre of $p_2$ and $C'\cap l_0\not=\emptyset$. In the latter case, we obtain that $\deg(g|_{C'})=\widetilde{\zeta}_l\cdot P_{l_0}(C)=1$ and $C$ is disjoint from $l_0$. In particular, the image $p_2(P_{l_0}(C))$ is a line in $\PP^{n-2}$ and the morphism
	\[
	\bar{P}_{l_0}|_C: C\longrightarrow p_2(P_{l_0}(C))=\PP^1
	\]
	is birational. Hence, the curves $C$ and $\widetilde{l}$ are also isomorphic to $\PP^1$.
\end{proof}

\subsubsection{Projection along general members: global construction}

Now we proceed to give a global construction of the projection to normal directions along general members in $\sH$. We use the notation of Setup \ref{Section:Setup}. 
The exact sequence of vector bundles 
\[
0\rightarrow (q_{\sH}^s)^*\Omega_{\sH_o} \longrightarrow \Omega_{\Univ_{\sH_o}} \longrightarrow \Omega_{\Univ_{\sH_o}/\sH_o} \longrightarrow 0.
\]
defines a projection map $\PP(\Omega_{\Univ_{\sH_o}}) \dashrightarrow \PP((q_{\sH}^s)^*\Omega_{\sH_o})$ which we extend to a rational map
\[
\begin{tikzcd}[column sep=large, row sep=large]
	P_{\bar{\sD}_0}: \PP(\Omega_{\Univ_{\sH}^s}) \arrow[r,dashed] & \PP((q_{\sH}^s)^*\Omega_{\sH^s})
\end{tikzcd}
\]
In particular, over the Zariski open subset $\Univ_{\sH_o}$, the map $P_{\bar{\sD}_0}$ is exactly the projection from the rational section $\bar{\sD}_0\subset \PP(\Omega_{\Univ_{\sH}^s})$ (see Notation \ref{notation-hjbar-bjbar}). We also have the following composition
\[
\begin{tikzcd}[column sep=large]
	\bar{P}_{\bar{\sD}_0}:\PP(\Omega_{\Univ_{\sH}^s})\arrow[r, dashed, "P_{\bar{\sD}_0}"] & \PP((q_{\sH}^s)^*\Omega_{\sH^s}) \arrow[r]  & \PP(\Omega_{\sH^s})
\end{tikzcd}
\]

\begin{notation} \label{notation-sZj}
For $1\leq j\leq m$, let us denote by $\sZ_j\subset \PP(\Omega_{\sH^s})$ the image $\bar{P}_{\bar{\sD}_0}(\bar{\sD}_j)$, i.e. the closure of the image of the generic point of  $\bar{\sD}_j$.
\end{notation}

Set $t=[l]\in \sH_o$. Then $C_l^0$ is the fibre of $q_{\sH}^s$ over $t$ and the induced morphism $f|_{C_l^0}:C_l^0\rightarrow l$ is nothing but the normalisation of $l$. Moreover, as $f$ is unramified over $\Univ_{\sH_o}$, we have a natural isomorphism $f^*T_X|_{C_l^0}\cong T_{\Univ_{\sH}^s}|_{C_l^0}$. Then, under this identification, the curve $l_0$ defined in Section \ref{Section:projection-local} is exactly the curve $\bar{\sD}_0\cap \PP(\Omega_{\Univ_{\sH}^s}|_{C_l^0})$ and $(q^s_{\sH})^*\Omega_{\sH^s}|_{C_l^0}$ identifies to the co-normal bundle of $l$. Hence, restricted to the fibre $C_l^0$, the projections $P_{\bar{\sD}_0}$ and $\bar{P}_{\bar{\sD}_0}$ are exactly the projections $P_{l_0}$ and $\bar{P}_{l_0}$ defined in Section \ref{Section:projection-local}, respectively. In particular, we shall denote by $P_{\bar{\sD}_l^0}$ (resp. $\bar{P}_{\bar{\sD}_l^0}$) the restriction of $P_{\bar{\sD}_0}$ (resp. $\bar{P}_{\bar{\sD}_0}$) to $\PP(\Omega_{\sU_{\sH}^s}|_{C_l^0})$.

\begin{theorem}
	\label{Thm:Projection-VMRT}
	In the situation of Setup \ref{Section:Setup}, let $\bar{\sH}_j \subset \bar{\sH}$ be an irreducible component such that $1\leq j\leq m$. If the projection 
	\[
	\bar{P}_j:=\bar{P}_{\bar{\sD}_0}|_{\bar{\sD}_j}:\bar{\sD}_j\dashrightarrow \sZ_j
	\]
	is not generically finite, then $\bar{\sH}_j$ is contained in $\Inf_{q_{\sH}^s}(f^*\sH)$. In particular, if $\bar{\sH}_j$ is contained in $\Mult(f^*\sH)$, then $\bar{P}_j:\bar{\sD}_j\dashrightarrow \sZ_j$ is generically finite.
\end{theorem}

The following picture illustrates the statement: since projecting from $\bar{\sD}_0$ corresponds to taking images under the tangent map of $q^s_\sH$,
the projection map is not generically finite if and only if 
there are infinitely many curves in $\bar \sH_j$ which are mapped 
onto the same curve in $\sH^s$.

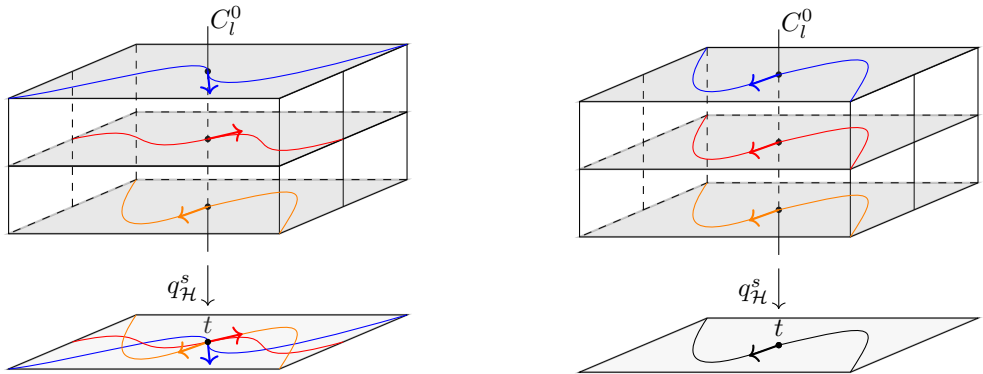
\begin{minipage}{0.4\linewidth}%
	\begin{tikzpicture}[scale=.6, z={(-.707,-.3)}]
		\coordinate (A) at (3,3,-2);
		\fill[black]  (A) circle [radius=2pt];
		
		\coordinate (B) at (3,1.5,-2);
		\fill[black]  (B) circle [radius=2pt];
		
		\coordinate (C) at (3,0,-2);
		\fill[black]  (C) circle [radius=2pt];
		
		\draw (2,3,-4) node[above] {$C_l^0$};
		\draw (3,-3,-2) node[above] {$t$};
		
		\draw (6,0,0) -- (0,0,0) -- (0,3,0) -- (6,3,0) -- cycle; %frontside
		\draw (6,0,0) -- (6,0,-4) -- (6,3,-4) -- (6,3,0);  %rightside
		\draw (6,3,0) -- (6,3,-4) -- (0,3,-4) -- (0,3,0);  %topside
		\draw (0,1.5,0) -- (6,1.5,0) -- (6,1.5,-4);  %midplan
		\draw (0,-3,0) -- (6,-3,0) -- (6,-3,-4) -- (0,-3,-4) -- cycle; %underplan
		\draw (3,0,-2) -- (3,-1,-2);  %bottommidline
		\draw (3,4,-2) -- (3,3,-2);   %topmidline
		\draw (6,0,-2) -- (6,3,-2);   %rightsidemidline
		\draw[style=dashed] (6,0,-4) -- (0,0,-4) -- (0,0,0); %bottomside
		\draw[style=dashed] (6,1.5,-4) -- (0,1.5,-4) -- (0,1.5,0); %midplan
		\draw[style=dashed] (0,0,-4) -- (0,3,-4); %leftbackside
		\draw[style=dashed] (3,3,-2) -- (3,0,-2); %midline
		\draw[style=dashed] (0,0,-2) -- (0,3,-2); %leftsidemidline

		\draw[thick,fill=gray!90,opacity=0.2] (0,0,0) -- (6,0,0) -- (6,0,-4) -- (0,0,-4) -- cycle;		
		\draw[thick,fill=gray!90,opacity=0.2] (0,3,0) -- (6,3,0) -- (6,3,-4) -- (0,3,-4) -- cycle;
		\draw[thick,fill=gray!90,opacity=0.2] (0,1.5,0) -- (6,1.5,0) -- (6,1.5,-4) -- (0,1.5,-4) -- cycle;
		\draw (0,1.5,0) -- (6,1.5,0) -- (6,1.5,-4);
		\draw[thick,fill=gray!30,opacity=0.2] (6,-3,0) -- (0,-3,0) -- (0,-3,-4) -- (6,-3,-4) -- cycle;

		\draw[blue] (0,3,0) .. controls (1,3,-0.5) and (2,3,-3.5) .. (3,3,-2);
		\draw[blue] (3,3,-2) .. controls (4,3,-0.5) and (5,3,-1) .. (6,3,-4);
		
		\draw[red] (0,1.5,-2) .. controls (1,1.5,-3.5) and (2,1.5,-0.5) .. (3,1.5,-2);
		\draw[red] (3,1.5,-2) .. controls (4,1.5,-3.5) and (5,1.5,0.5) .. (6,1.5,-2);
		
		\draw[orange] (0,0,-4) .. controls (2,0,-0.01) and (2.8,0,-0.1) .. (3,0,-2);
		\draw[orange] (3,0,-2) .. controls (4,0,-3.8) and (4.2,0,-3.5) .. (6,0,0);
		
		\draw[thick,blue,->] (3,3,-2) -- (4.25,3,-0.3);
		\draw[thick,red,->] (3,1.5,-2) -- (3.35,1.5,-2.6);
		\draw[thick,orange,->] (3,0,-2) -- (2.9,0,-1.2);
		
		\draw[blue] (0,-3,0) .. controls (1,-3,-0.5) and (2,-3,-3.5) .. (3,-3,-2);
		\draw[blue] (3,-3,-2) .. controls (4,-3,-0.5) and (5,-3,-1) .. (6,-3,-4);
		
		\draw[red] (0,-3,-2) .. controls (1,-3,-3.5) and (2,-3,-0.5) .. (3,-3,-2);
		\draw[red] (3,-3,-2) .. controls (4,-3,-3.5) and (5,-3,0.5) .. (6,-3,-2);
		
		\draw[orange] (0,-3,-4) .. controls (2,-3,-0.01) and (2.8,-3,-0.1) .. (3,-3,-2);
		\draw[orange] (3,-3,-2) .. controls (4,-3,-3.8) and (4.2,-3,-3.5) .. (6,-3,0);
		
		\draw[thick,blue,->] (3,-3,-2) -- (4.25,-3,-0.3);
		\draw[thick,red,->] (3,-3,-2) -- (3.35,-3,-2.6);
		\draw[thick,orange,->] (3,-3,-2) -- (2.9,-3,-1.2);
		
		\coordinate (D) at (3,-3,-2);
		\fill[black]  (D) circle [radius=2pt];
		
		\draw[->] (3,-1.3,-2) -- (3,-2.2,-2);
		\draw (3,-1.8,-2) node[left] {$q_{\sH}^s$};
	\end{tikzpicture}
\end{minipage}
\hfill
\begin{minipage}{0.4\linewidth}
	\begin{tikzpicture}[scale=.6, z={(-.707,-.3)}]
		\coordinate (A) at (3,3,-2);
		\fill[black]  (A) circle [radius=2pt];
		
		\coordinate (B) at (3,1.5,-2);
		\fill[black]  (B) circle [radius=2pt];
		
		\coordinate (C) at (3,0,-2);
		\fill[black]  (C) circle [radius=2pt];
		
		\coordinate (D) at (3,-3,-2);
		\fill[black]  (D) circle [radius=2pt];
		
		\draw (6,0,0) -- (0,0,0) -- (0,3,0) -- (6,3,0) -- cycle; %frontside
		\draw (6,0,0) -- (6,0,-4) -- (6,3,-4) -- (6,3,0);  %rightside
		\draw (6,3,0) -- (6,3,-4) -- (0,3,-4) -- (0,3,0);  %topside
		\draw (0,1.5,0) -- (6,1.5,0) -- (6,1.5,-4);  %midplan
		\draw (0,-3,0) -- (6,-3,0) -- (6,-3,-4) -- (0,-3,-4) -- cycle; %underplan
		\draw (3,0,-2) -- (3,-1,-2);  %bottommidline
		\draw (3,4,-2) -- (3,3,-2);   %topmidline
		\draw (6,0,-2) -- (6,3,-2);   %rightsidemidline
		\draw[style=dashed] (6,0,-4) -- (0,0,-4) -- (0,0,0); %bottomside
		\draw[style=dashed] (6,1.5,-4) -- (0,1.5,-4) -- (0,1.5,0); %midplan
		\draw[style=dashed] (0,0,-4) -- (0,3,-4); %leftbackside
		\draw[style=dashed] (3,3,-2) -- (3,0,-2); %midline
		\draw[style=dashed] (0,0,-2) -- (0,3,-2); %leftsidemidline
		
		\draw[thick,fill=gray!90,opacity=0.2] (0,0,0) -- (6,0,0) -- (6,0,-4) -- (0,0,-4) -- cycle;		
		\draw[thick,fill=gray!90,opacity=0.2] (0,3,0) -- (6,3,0) -- (6,3,-4) -- (0,3,-4) -- cycle;		
		\draw[thick,fill=gray!90,opacity=0.2] (0,1.5,0) -- (6,1.5,0) -- (6,1.5,-4) -- (0,1.5,-4) -- cycle;
		\draw[thick,fill=gray!30,opacity=0.2] (6,-3,0) -- (0,-3,0) -- (0,-3,-4) -- (6,-3,-4) -- cycle;

		\draw[blue] (0,3,-4) .. controls (2,3,-0.01) and (2.8,3,-0.1) .. (3,3,-2);
		\draw[blue] (3,3,-2) .. controls (4,3,-3.8) and (4.2,3,-3.5) .. (6,3,0);
		
		\draw[red] (0,1.5,-4) .. controls (2,1.5,-0.01) and (2.8,1.5,-0.1) .. (3,1.5,-2);
		\draw[red] (3,1.5,-2) .. controls (4,1.5,-3.8) and (4.2,1.5,-3.5) .. (6,1.5,0);
		
		\draw[orange] (0,0,-4) .. controls (2,0,-0.01) and (2.8,0,-0.1) .. (3,0,-2);
		\draw[orange] (3,0,-2) .. controls (4,0,-3.8) and (4.2,0,-3.5) .. (6,0,0);
		
		\draw[thick,blue,->] (3,3,-2) -- (2.9,3,-1.2);
		\draw[thick,red,->] (3,1.5,-2) -- (2.9,1.5,-1.2);
		\draw[thick,orange,->] (3,0,-2) -- (2.9,0,-1.2);

		\draw (0,-3,-4) .. controls (2,-3,-0.01) and (2.8,-3,-0.1) .. (3,-3,-2);
		\draw (3,-3,-2) .. controls (4,-3,-3.8) and (4.2,-3,-3.5) .. (6,-3,0);

		\draw[thick,->] (3,-3,-2) -- (2.9,-3,-1.2);
		
		\coordinate (D) at (3,-3,-2);
		\fill[black]  (D) circle [radius=2pt];
		
		\draw (2,3,-4) node[above] {$C_l^0$};
		\draw (3,-3,-2) node[above] {$t$};
		\draw[->] (3,-1.3,-2) -- (3,-2.2,-2);
		\draw (3,-1.8,-2) node[left] {$q_{\sH}^s$};
	\end{tikzpicture}
\end{minipage}

\begin{proof}
	Let $t=[l]\in \sH_o$ be a general member and let $y\in C_l^0$ be a general point. Set $x=f(y)$. Then there exists an Euclidean neighbourhood $x\in U_x\subset X$ of $x$ satisfying the following properties:
	\begin{enumerate}
		\item the \'etale web $\sH$ of curves defines a web-structure $W_i\subset T_{U_x}$ $(1\leq i\leq d)$ as in Definition \ref{defn:webstructure};
		
		\item if $V_y$ is the connected component of $f^{-1}(U_x)$ containing $y$,
		the induced morphism $f_y:=f|_{V_y}:V_y\rightarrow U_x$ is biholomorphic and $O_t:=q_{\sH}^s(V_y)$ is contained $\sH_{o}$;
		
		\item the \'etale web of curves $\bar{\sH}_j$ defines a web-structure $\bar{W}_i\subset T_{V_y}$ $(1\leq i\leq d')$ as in Definition \ref{defn:webstructure}.
	\end{enumerate}
    From the construction of $\bar{\sH}_j$, it follows that $\bar{W}_i\subset T_{V_y}$ $(1\leq i\leq d')$ is actually a sub-web-structure of the pull-back of the web-structure $W_i\subset T_{U_x}$ $(1\leq i\leq d)$ on $U_x$ via the biholomorphism $f_y:V_y\rightarrow U_x$. After renumbering, we may assume that $\bar{W}_i\subset T_{V_y}$ is exactly the pull-back of $W_i\subset T_{U_x}$. By the definition of $\bar{\sD}_j$, we have
    \[
    \bar{\sD}_j\cap \PP(\Omega_{V_y}) = \bigcup_{i=1}^{d'} \PP(\bar{W}_i^*)\subset \PP(\Omega_{V_y}),
    \]
    where $\bar{W}_i^*$ is the dual of $\bar{W}_i$. Moreover, we note that $T_{V_y/O_t}\cap \bar{W}_i$ is just the zero section for any $1\leq i\leq d'$. It follows that $\bar{P}_j$ is well-defined along $\bar{\sD}_j\cap \PP(\Omega_{V_y})$. 
    \[
    \begin{tikzcd}[column sep=large, row sep=large]
    	\bigcup_{i=1}^{d'}\PP(\bar{W}_i^*)=\bar{\sD}_j\cap \PP(V_y) \arrow[d,"\pi_1" left] \arrow[r,hookrightarrow] \arrow[rrr,bend left,"\bar{P}_j"]
    	    & \PP(\Omega_{V_y}) \arrow[d,"\pi_1" left] \arrow[r,dashed,"\bar{P}_{\bar{\sD}_0}"] 
    	        & \PP(\Omega_{O_t}) \arrow[r,hookleftarrow] \arrow[d,"\pi_2"]
    	            & \sZ_j\cap \PP(\Omega_{O_t}) \arrow[d,"\pi_2"] \\
        V_y \arrow[r,equal]
            & V_y \arrow[r,"q_{\sH}^s"]
                & O_t \arrow[r,equal]
                    & O_t
    \end{tikzcd}
    \]
    
    Since $\bar{P}_j:\bar{\sD}_j\dashrightarrow \sZ_j$ is not generically finite by our assumption, there exists some $1\leq i\leq d'$ such that $\PP(\bar{W}_i^*)\rightarrow \bar{P}_j(\PP(\bar{W}_i^*))$ is not generically finite. As $\PP(\bar{W}_i^*)$ is a section of $\pi_1:\PP(\Omega_{V_y})\rightarrow V_y$, it follows that the induced morphism
    \[
    \PP(\bar{W}_i^*) \xrightarrow{\bar{P}_j} \sZ_j\cap \PP(\Omega_{O_t}) \xrightarrow{\pi_2} O_t 
    \]
    is surjective and has relative dimension at least one. On the other hand, since $V_y\rightarrow O_t$ is of relative pure dimension one with connected and smooth fibres, it follows that the morphism $\bar{P}_j(\PP(\bar{W}_i^*))\rightarrow O_t$ is birational. In particular, after replacing $x$ by a general point in $U_x$ and then shrinking $U_x$, we may assume that $\bar{P}_j(\PP(\bar{W}_i^*))\rightarrow O_t$ is an biholomorphism; that is, $\bar{P}_j(\PP(\bar{W}_i^*))$ is a section of $\pi_2:\PP(\Omega_{O_t})\rightarrow O_t$. Denote by $L_i\subset T_{O_t}$ the $1$-dimensional distribution such that $\PP(L_i^*)=\bar{P}_j(\bar{W}_i^*)\subset \PP(\Omega_{O_t})$. 
    
    Let $C$ be a curve parametrised by $\bar{\sH}_j$ such that $C\cap V_y$ is a leaf of $W_i$. Then the image $q_{\sH}^s(C)\cap O_t$ is a leaf of $L_i$. In particular, if $\{C_s\}_{s\in \Delta}$ are the curves parametrised by $\bar{\sH}_j$ meeting $C_l^0\cap V_y$ such that $C_s\cap V_y$ are leaves of $W_i$, then $q_{\sH}^s(C_s)\cap O_t$ are leaves of $L_i$ passing through the fixed point $t$. Hence, by the uniqueness of leaves, the image $q_{\sH}^s(C_s)\cap O_t$ is  independent of $s\in \Delta$. Then taking Zariski closure shows that the image $q_{\sH}^s(C_s)$ is independent of $s\in \Delta$. Hence, $\bar{\sH}_j$ is contained in $\Inf_{q_{\sH}^s}(f^*\sH)$. 
\end{proof}

\begin{remark}
	\label{Remark:Generical-finiteness}
	Note that the restriction $\bar{\sD}_j\cap \PP(\Omega_{\Univ_{\sH}^s}|_{C_l^0})$ is a finite union of horizontal irreducible curves $\sD_{ji}$ which are different from $\bar{\sD}_0\cap \PP(\Omega_{\Univ_{\sH}^s}|_{C_l^0})$, and we have the following commutative diagram whose vertical morphisms are generically finite. 
		\[
		\begin{tikzcd}[column sep=large, row sep=large]
			\bar{\sD}_j \arrow[d] \arrow[r,dashed,"\bar{P}_j"] 
			    & \sZ_j \arrow[d] \\
			\sU_{\sH}^s \arrow[r,"q_{\sH}^s"] 
			    & \sH^s
		\end{tikzcd}
		\]
	Thus the local description of the projections given in Section \ref{Section:projection-local} shows that the map $\bar{P}_j:\bar{\sD}_j\dashrightarrow\sZ_j$ is generically finite if and only the image $P_{\bar{\sD}^0_{l}}(\sD_{ji})$ is not a fibre of the projection $p_2:C_l^0\times \PP^{n-2}\rightarrow \PP^{n-2}$.
\end{remark}

\section{Dual variety of minimal rational tangents}

\subsection{Standard covering family of rational curves}

Let $X$ be a uniruled projective manifold of dimension $n$, and let $C\subset X$ be a rational curve. We say that $C$ is a \emph{standard} rational curve if for the normalisation $f:\PP^1\rightarrow C$, we have
\begin{equation}
	\label{Eq:Standard-curves}
	f^*T_X\cong \sO_{\PP^1}(2)\oplus \sO_{\PP^1}(1)^{\oplus p}\oplus \sO_{\PP^1}^{\oplus (n-p-1)},
\end{equation}
where $p+2=-K_X \cdot C$ is the anti-canonical degree of $C$. Let $\sH$ be an irreducible covering family of rational curves on $X$. We say that $\sH$ is \emph{standard} if general curves parametrised by $\sH$ are standard rational curves. We say $\sH$ is \emph{minimal} if $\sH$ has minimal degree with respect to some fixed ample line bundle $A$ on $X$. It is well-known that a covering family $\sH$ of minimal rational curves is standard and a uniruled projective manifold always carries a covering family of minimal rational curves. If there is no danger of confusion, we shall write the tangent maps $\tau_{\sH}$ and $\tau_{\sH,x}$ as $\tau$ and $\tau_x$, respectively.

\begin{proposition}
	\label{p.tangentstanadard}
	Let $\sH$ be a standard covering family of rational curves on a projective manifold $X$. Denote by $\sD_x\subset \PP(\Omega_{X,x})$ the tangent variety of $\sH$ at a general point $x\in X$. Then the following statements hold.
	\begin{enumerate}
		\item The tangent map $\tau_x:\sH_x\dashrightarrow \sD_x$ is an immersion at the point in $\sH_x$ corresponding to a standard rational curve.
		
		\item Let $[C]\in \sH_x$ be a standard rational curve with $y=\tau_x([C])$ and denote by $\widehat{T}_{y}\subset T_{X,x}$ the $(p+1)$-dimensional subspace corresponding to the positive factors of the splitting 
		\[
		f^*T_X\cong \sO_{\PP^1}(2)\oplus \sO_{\PP^1}(1)^{\oplus p}\oplus \sO_{\PP^1}^{\oplus (n-p-1)},
		\]
		where $f:\PP^1\rightarrow C$ is the normalisation. Then $\PP(\widehat{T}^*_{y})\subset \PP(\Omega_{X,x})$ is the projective tangent bundle of $\sD_x$ at $y$.
		
		\item If we assume moreover that $\sH$ is minimal, then $\sH_x$ is smooth and $\tau_x:\sH_x \dashrightarrow \sD_x$ coincides with the normalisation of $\sD_x$.
	\end{enumerate}
\end{proposition}

\begin{proof}
	The statements (1) and (2) are proved in \cite[Proposition 1.4 and Proposition 2.3]{Hwa01} under the assumption that $\sH$ is minimal. However, the proof only uses the splitting type of $T_X$ along general curves parametrised by $\sH$. The statement (3) is proved by \cite{Keb02a} and \cite{HM04}.
\end{proof}

	In this paper, we shall consider simultaneously 
	two families of rational curves, one of them being minimal. As in the preceding sections we will denote a general covering family of curves by $\sH$. We will distinguish the family of minimal rational curves by using the notation that is most common in the literature:

\begin{notation}
Let $X$ be a projective manifold. We shall 
denote by $\sK$ a family of minimal rational curves on $X$, and somewhat abusively, by the same letter the covering family obtained by taking the the normalisation of its closure in $\Chow{X}_1$.  For such a family, we denote the tangent variety by
$$
\sC\subset \PP(\Omega_{X})
$$
and call it the total variety of minimal rational tangents (total VMRT for short). For a general point $x\in X$ 
the tangent variety $\sC_x\subset \PP(\Omega_{X,x})$ is called the variety of minimal rational tangents (VMRT for short).
\end{notation}

\begin{lemma}
	\label{Lemma:Standard}
	Let $X$ be a uniruled projective manifold equipped with a covering family $\sH$ of rational curves with anti-canonical degree four such that the general member is immersed. If $\sH$ is not standard, then the tangent variety $\sD_x\subset \PP(\Omega_{X,x})$ of $\sH$ at a general point $x\in X$ is contained in a finite union of projective lines.
\end{lemma}

\begin{proof}
	Since the total tangent variety $\sD$ is irreducible, each irreducible component of $\sD_x$ has the same degree. In particular, it is enough to show that there exists an irreducible component of $\sD_x$ which is contained in a projective line. 

  Let $[C]\in \sH$ be a general member with normalisation $f:\PP^1\rightarrow X$, then $f$ is free.
 Since by assumption $f$ is immersive and $\sH$ is not standard, we have
  	\[
	f^*T_X\cong \sO_{\PP^1}(2)\oplus \sO_{\PP^1}(2)\oplus \sO_{\PP^1}^{\oplus (n-2)}.
	\]
  Let now $\sH_x^i \subset \sH_x$ be an irreducible component and  $S=\text{Locus}(\sH_x^i)\subset X$ be the Zariski closure of the union of curves parametrised by $\sH_x^i$.  
  Let $y\in S$ be a general point, so that $S$ is smooth at $y$
  and there exists a curve $[C] \in \sH_x^i$ having splitting type as above. Since $f$ is not standard (i.e. not minimal in the terminology of \cite[IV, Defn.2.8]{Ko96}), we know by \cite[IV, Cor.2.9]{Ko96}
  that we can deform the rational curve $[C]$ keeping $x$ and $y$ fixed, i.e.
  there exists an irreducible curve $\delta_y\subset \sH_x^i$ such that every curve $C$ parametrised by $\delta_y$ passes through $x$ and $y$.
  Thus the general fibre of the evaluation morphism
  $$
  ev: \sH_x^i \times \PP^1 \rightarrow S
  $$
  has positive dimension. Since $\sH_x^i$ has dimension two, this shows
  that $\dim S=2$. Since $x \in X$ is general, and $\sH_x^i \subset \sH_x$ is any irreducible component, we obtain
  that $\dim \text{Locus}(\sH_z)=2$ for every general point $z \in X$.
  
  Let $\sH_y^j$ be an irreducible component of $\sH_y$ containing $\delta_y$. Then we have
	\[
	S=\text{Locus}(\delta_y)\subset\text{Locus}(\sH_y^j).
	\]
	As both $S$ and $\text{Locus}(\sH_y^j)$ are irreducible surfaces, we obtain $S=\text{Locus}(\sH^j_y)$. In particular, after replacing $x$ by $y$, we may assume that $S$ is smooth at $x$. Then $\tau_{\sH}(\sH_x^i)$ is contained in $\PP(\Omega_{S,x})\subset \PP(\Omega_{X,x})$ which is a projective line.
\end{proof}

\subsection{Dual variety of tangent variety}
Let $[C]\in \sH$ be a standard rational curve. Then a \emph{minimal section $\bar{C}$ over $C$} is the image of a section of $\PP(f^*T_X)\rightarrow \PP^1$ corresponding to the quotient
\[
f^*T_X\cong \sO_{\PP^1}(2)\oplus \sO_{\PP^1}(1)^{\oplus p}\oplus \sO_{\PP^1}^{\oplus (n-p-1)} \longrightarrow \sO_{\PP^1},
\]
where $f:\PP^1\rightarrow C$ is the normalisation.

\begin{definition}
	\label{Def:Dual-VMRT}
	let $\sH$ be a standard covering family  of rational curves over a projective manifold $X$. 
 The total dual tangent variety $\check{\sD}\subset \PP(T_X)$ of $\sH$ is defined as the Zariski closure of the union of all minimal sections over standard rational curves in $\sH$
	\[
	\check{\sD}=\overline{\bigcup_{[C]\in \sH:\ \text{standard rational curve}} \bar{C}}^{\rm Zar} \subset \PP(T_X).
	\]
\end{definition}

Let us recall the definition of dual varieties of projective varieties . Let $V$ be a complex vector space of dimension $N+1$, and let $Z\subset \PP^N=\PP(V)$ be a projective variety. We denote by $T_{Z,z}$ the tangent space at any smooth point $z\in Z^{\rm sm}$, where $Z^{\rm sm}$ is the non-singular locus of $Z$. We can also define the \emph{embedded projective tangent space} $\hat{T}_{Z,z}\subset \PP^N$ as follows:
\[
\hat{T}_{Z,z}=\PP(T_{{\rm Cone}(Z),v}),
\]
where ${\rm Cone}(Z)\subset V$ is the cone over $Z$, $v$ is any non-zero point on the line $z$, and we consider $T_{{\rm Cone}(Z),v}$ as a linear subspace of $V$. A hyperplane $H\subset \PP^N$ is a tangent hyperplane of $Z$ if $\hat{T}_{Z,z}\subset H$ for some $z\in Z^{\rm sm}$. 

\begin{definition}
	Let $Z\subset \PP^N=\PP(V)$ be a projective variety.
	
	\begin{enumerate}
		\item The closure of the set of all tangent hyperplanes is called the projectively dual variety $\check{Z}\subset \check{\PP}^N=\PP(V^*)$, where $V^*$ is the dual space of $V$.
		
		\item The dual defect $\defect(Z)$ of $Z$ is defined as $N-\dim(\check{Z})-1$, and $Z$ is called dual defective if $\defect(Z)>0$.
	\end{enumerate}
\end{definition}

The next result justifies the terminology in Definition \ref{Def:Dual-VMRT}:

\begin{proposition}
	\label{Prop:Dual-Defect-VMRT}
	\cite[Proposition 5.14]{MunozOcchettaSolaCondeWatanabeEtAl2015}
	With the same notation as in Definition \ref{Def:Dual-VMRT}, let $x\in X$ be a general point. Then $\check{\mathcal{D}}_x\subset \PP(T_{X,x})$ is the projectively dual variety of $\mathcal{D}_x\subset \PP(\Omega_{X,x})$.
\end{proposition}

\begin{proof}
	In the proof of \cite[Proposition 5.14]{MunozOcchettaSolaCondeWatanabeEtAl2015}, $\sH$ is assumed to be minimal. However, the proof uses only the splitting type of $T_X$ along general curves parametrised by $\sH$, so it still works for standard covering families of rational curves.
\end{proof}

Let $\sK$ be a covering family of minimal rational curves on $X$. Then the total dual tangent variety of $\sK$ is denoted by $\check{\sC}\subset \PP(T_X)$ and we shall call it the \emph{total dual VMRT} of $\sK$. Denote by $c$ the dual defect of $\sC_x\subset \PP(\Omega_{X,x})$. If $c=0$, then the total dual VMRT $\check{\sC}\subset \PP(T_X)$ is a prime divisor. In particular, there exists a Cartier divisor $\Delta$ on $X$ such that
\[
[\check{\sC}] = m\check{\zeta} + \check{\pi}^*\Delta,
\]
where $\check{\pi}:\PP(T_X)\rightarrow X$ is the projectivisation and $\check{\zeta}$ is tautological divisor class $c_1(\sO_{\PP(T_X)}(1))$. Moreover, by Proposition \ref{Prop:Dual-Defect-VMRT}, we also see that $m$ is equal to the degree of the dual variety of $\sC_x\subset \PP(\Omega_{X,x})$ at a general point $x\in X$, which is a hypersurface in $\PP(T_{X,x})$ as $c=0$.

\begin{remark}
	Let $\sH$ be a standard covering family of rational curves on an $n$-dimensional projective manifold $X$. If a general member of $\sH$ has $(-K_X)$-degree $n+1$, then its tangent variety $\sD$ is the total space $\PP(\Omega_{X})$ and the total dual tangent variety $\check{\sD}$ is empty by our definition. However, if we assume that $\sH$ is a covering family of minimal rational curves, then we know by  \cite{CMS02,Keb02}
	that if $X$ is not isomorphic to $\PP^n$, then the anti-canonical degree of a general member in $\sH$ is less than or equal to $n$; that is, we have $n-p-1>0$ in \eqref{Eq:Standard-curves} if $X\not\cong \PP^n$.
\end{remark}

\begin{proposition}
	\label{Prop:Bigness-Criterion}
	Assume that $X$ is a Fano manifold of Picard number one equipped with a covering family $\sK$ of minimal rational curves such that $c=0$. Then $T_X$ is big if and only if $\Delta$ is anti-ample on $X$.
	In this case $\check \sC$ is a generator of the pseudoeffective cone $\mbox{\rm Pseff}(\PP(T_X))$.
\end{proposition}

\begin{proof}
	If $\Delta$ is anti-ample, then $T_X$ is big by \cite[Lemma 2.3]{HLS20}. Conversely, assume that $T_X$ is big. Then by \cite[Lemma 2.3]{HLS20}, there exists an anti-big $\Q$-divisor $H$ on $X$ such that $\check{\zeta}+\check{\pi}^*H$ is pseudoeffective. As $X$ has Picard number one, it follows that $H$ is actually anti-ample. Note that $\check{\sC}$ is dominated by minimal sections $\bar{C}$ and 
	\[
	(\check{\zeta}+\check{\pi}^*H)\cdot \bar{C}=\check{\pi}^*H\cdot \bar{C}=H\cdot C<0.
	\]
	Therefore, the divisor $\check{\zeta}+\check{\pi}^*H$ is not modified nef. Thus, by \cite[Corollary 2.6]{HLS20}, there exists a prime divisor $D$ such that $[D]$ generates an extremal ray of the pseudoeffective cone of $\PP(T_X)$. As $T_X$ is big, there exists a positive integer $m'$ and an anti-ample $\Q$-divisor $H'$ such that
	\[
	\frac{1}{m'}D\equiv \check{\zeta}+\check{\pi}^*H'.
	\] 
	On the other hand, the same computation as above shows that the restriction $D|_{\check{\sC}}$ is not pseudoeffective. Hence, we have $\check{\sC}=D$ and consequently we obtain $m=m'$ and $\Delta=mH'$ is anti-ample.
\end{proof}

In this paper we are mainly interested in the case where $p=0$. In particular, we have $c=0$ as $c\leq p$ (see \cite[Proposition 5.12]{MunozOcchettaSolaCondeWatanabeEtAl2015}). To compute the cohomological class of the total dual VMRT in this case, we start with a general construction. 

Let $B$ be a smooth projective curve and let $E\rightarrow B$ be a vector bundle of rank $r\geq 2$. Denote by $\eta:\PP(E)\rightarrow B$ the projectivisation. Let $C = \cup_{i=1}^r C_i \subset \PP(E)$ be a finite union of curves which does not contain any $\eta$-vertical irreducible components. We define the total dual variety $\check{C}\subset \PP(E^*)$ of $C$ as follows: for a general point $b\in B$, the fibre of $\check{\eta}^{-1}(b)\cap \check{C}$ is the dual variety of $\eta^{-1}(b)\cap C$, where $\check{\eta}:\PP(E^*)\rightarrow B$ is the projectivisation. Then we define the 
$\check{C}$ to be the Zariski closure of the union of these dual varieties over general points of $B$. Since $\eta^{-1}(b)\cap C$ is a finite union of points in $\PP(E_b)$, it follows that $\check{\eta}^{-1}(b)\cap \check{C}$ is a finite union of hyperplanes in $\PP(E^*_b)$. Hence, the variety $\check{C}\subset \PP(E^*)$ is a divisor. Let $\check{\zeta}$ be the tautological divisor class of $\PP(E^*)$ and let $\check{F}$ be a fibre of $\check{\eta}$. Then there exist integers $a$ and $b$ such that
\[
[\check{C}] \equiv a\check{\zeta} + b\check{F}.
\]

\begin{lemma}
	\label{Lemma:Class-Dual-Curves}
	Let $\zeta$ be the tautological divisor class of $\PP(E)$ and let $F$ be a $\eta$-fibre. Then we have
	\begin{center}
		$a=F\cdot C$\quad and \quad $b=\zeta \cdot C$.
	\end{center}
\end{lemma}

\begin{proof}
Note that $\check{C} = \sum_{i=1}^r \check{C}_i$ where $\check{C}_i$ is the total dual of the irreducible curve $C_i$. Thus,
	without loss of generality, we may assume that $C$ is irreducible. Let $f:\bar{C}\rightarrow C$ be the normalisation. Then we have a Cartesian diagram
	\[
	\begin{tikzcd}[column sep=large, row sep=large]
		\PP(\nu^*E) \arrow[r,"\bar{\nu}"] \arrow[d,"\bar{\eta}" left]  &  \PP(E) \arrow[d,"\eta"] \\
		\bar{C} \arrow[r,"\nu"]                                &  B
	\end{tikzcd}
	\]
	where $\nu:\bar{C}\rightarrow B$ is the composition $\eta\circ f$. Moreover, there exists a quotient line bundle $\nu^*E\rightarrow L$ such that the corresponding section $\mu:\bar{C}\rightarrow \PP(\nu^*E)$ satisfies
	\begin{center}
		$(\bar{\nu}\circ \mu)(\bar{C})=C$\quad and\quad $c_1(L)\cdot \bar{C}=\zeta\cdot C$.
	\end{center}
    On the other hand, the total dual variety $\check{\widetilde{C}}\subset \PP(\nu^*E^*)$ of $\widetilde{C}:=\mu(\bar{C})\subset \PP(\nu^*E)$ is exactly the projective sub-bundle $\PP(\nu^*E^*/L^*)\subset \PP(\nu^*E^*)$. In particular, we have
    \[
    [\check{\widetilde{C}}] \equiv \check{\bar{\zeta}} + (c_1(L)\cdot \bar{C}) \check{\bar{F}},
    \]
    where $\check{\bar{\zeta}}$ is the tautological divisor class of $\PP(\nu^*E^*)$ and $\check{\bar{F}}$ is a fibre of $\bar{\eta}$. On the other hand, note that we have $\bar{\nu}_*\check{\widetilde{C}}=\check{C}$. Therefore, we get
    \[
    [\check{C}]\equiv \deg(\nu) \check{\zeta} + (c_1(L)\cdot \bar{C}) \check{F} = (F\cdot C)\check{\zeta} + (\zeta\cdot C) \check{F}.
    \]
\end{proof}

\subsection{Big tangent bundle} \label{subsectionbig}

In this subsection, we shall always make the following

\begin{assumption}{\rm \label{assumptionbig}
Let $X$ be a Fano manifold of dimension $n$ and with Picard number one, and  denote by $A$ the ample generator of $\pic(X)$.  Assume that the tangent bundle $T_X$ is big and $X$ admits an \'etale web of rational curves $\sK$, i.e., 
for a general point $x\in X$, there exists a rational curve which has degree two with respect to $-K_X$.
}
\end{assumption}

Let us  recall that by \cite[Thm.3.3(2)]{Keb02a} the curves in the \'etale web passing through a general point $x$ are immersed and smooth in $x$. However it is possible that these curves are singular in some other point. 
Our goal is now to deduce further properties of the family of rational curves when $T_X$ is big. We will summarise these properties in Theorem \ref{Thm:Summary-VMRT} and Corollary \ref{c.lines-B_l^1-properties}, introducing the notation along the way of their proof.

In the sequel of this section, we will follow the notations given as in Section \ref{Section:Etale-Webs} except that we replace the notations $\sH$ by $\sK$ and $\sD$ by $\sC$ to emphasise that $\sK$ is an \'etale web of rational curves. In particular, we have the following commutative diagram
\begin{equation}
	\label{Eq:Universal-Family-K}
	\begin{tikzcd}[column sep=large, row sep=large]
		\Univ_{\sK}^s \arrow[r,"\widetilde{\mu}"] \arrow[d,"q_{\sK}^s" left] \arrow[rr, bend left, "f"]&  \Univ_{\sK} \arrow[r,"e_{\sK}"] \arrow[d,"q_{\sK}"] &  X \\
		\sK^s \arrow[r,"\mu"]                                           &   \sK                                       & 
	\end{tikzcd}
\end{equation}

Let now
\[
\bar{\sK}:=f^*\sK=\bar{\sK}_0\cup \left(\bigcup_{j=1}^m \bar{\sK}_j\right)
\]
be the decomposition into irreducible components such that $\bar{\sK}_0$ corresponds to general fibres of $q_{\sK}^s$. In particular, we have $[C_l^0]\in \bar{\sK}_0$ for general members $[l]\in \sK$. Moreover, up to renumbering, we can assume that $\bar{\sK}_j$ is an irreducible component of $\Mult(f^*\sK)$ for $1\leq j\leq k$ $(\leq m)$ and $\bar{\sK}_j$ is an irreducible component of $\Bir(f^*\sK)$ for $k+1\leq j\leq m$. 

Denote by $\bar{\sC}_j$ the total tangent variety of $\bar{\sK}_j$. For a general member $[l]\in \sK$, we decompose
\[
\bar{\sC}_l^j:=\bar{\sC}_j\cap\PP(\Omega_{\Univ_{\sK}^s}|_{C_l^0})=\bigcup_{i} \bar{\sC}^{ji}_l 
\]
into irreducible components. We know that $\bar{\sC}_l^0$ is irreducible since it is exactly the section corresponding to the quotient $\Omega_{\Univ_{\sK}^s}|_{C_l^0}\rightarrow \Omega_{C_l^0}$. 

Let us denote by $\bar{\zeta}$ the tautological divisor class of $\PP(\Omega_{\sU_{\sK}^s})$, and by $\zeta_l$ its restriction to $\PP(\Omega_{\Univ_{\sK}^s}|_{C_l^0})$.

\subsubsection{VMRT along minimal rational curves} \label{subsubsectionVMRTalong}

We start by deducing some numerical consequence of the bigness of $T_X$:
\begin{proposition}
	\label{Prop:Degree-VMRT}
	Under the Assumption \ref{assumptionbig}, let $l$ be a general member of $\sK$. Then $A\cdot l=1$ and the following statements hold.
	\begin{enumerate}
		\item $\Mult(\bar \sK)=\Fin_{q_{\sK}^s}(\bar \sK)=\bar{\sK}_1$ and $\bar{\sC}_l^1$ is an irreducible curve such that  $\bar{\sC}_l^1$ is not contracted by $\bar{P}_{\bar{\sC}_0}$ and $\bar{\zeta}\cdot \bar{\sC}_l^1=1$.
		
		\item For each $j\geq 2$, we have $\bar{\zeta}\cdot \bar{\sC}_{l}^{ji}=0$ for arbitrary $i$. In particular, $\bar{\sC}_l^{ji}$ is a rational curve which is disjoint from $\bar{\sC}_l^0$ if $j\geq 1$. 
	\end{enumerate}
    Moreover the projection $\bar{P}_j: \bar{\sC}_j \dashrightarrow \sZ_j$ is generically finite if and only if $j=1$.
\end{proposition}

\begin{remark*}
Before we start the proof, let us recall that if $[l] \in \sK$ is general, then $\check \sC \times_X l$ has no vertical component. Since $C^0_l \rightarrow l$ is the normalisation and $f$ is unramified near $C^0_l$, we can identify the irreducible components of $\check \sC \times_X l$ with their preimage under the pull-back
$\PP(f^*T_X|_{C_l^0}) \rightarrow \PP(T_X)$.
\end{remark*}

\begin{proof}
	Since $X$ has Picard number $1$ and $T_X$ is big, by Proposition \ref{Prop:Bigness-Criterion}, there exist positive integers $a$ and $b$ such that
	\[
	\check{\sC}\equiv a\check{\zeta} - b\check{\pi}^*A.
	\]
	Denote by $\check{\bar{\sC}}_l$ the pull-back of $\check{\sC}\subset \PP(T_X)$ to $\PP(T_{\Univ_{\sK}^s}|_{C_l^0})=\PP(f^*T_X|_{C_l^0})$. Then $\check{\bar{\sC}}_l$ is just the total dual variety of 
	\[
	\bar{\sC}_l=\bigcup_{j=0}^m \bar{\sC}_l^j \subset \PP(\Omega_{\Univ_{\sK}^s}|_{C_l^0}).
	\]
    
	Then, according to Lemma \ref{Lemma:Class-Dual-Curves}, we have	$-bA\cdot l = \zeta_l\cdot \bar{\sC}_l$. As $b>0$ and $\zeta_l\cdot \bar{\sC}_l^0=-2$, it follows that 
	\[
	\zeta_l\cdot \sum_{j=1}^m\bar{\sC}_l^{j}\leq 1	
	\]
	and if equality holds then $b=1$ and $A \cdot l=1$.
	Note that for a fixed $j$, the $\bar{\zeta}$-degree of $\bar{\sC}_l^{ji}$ is independent of $i$. Thus, thanks to Lemma \ref{Lemma:Projection-Local}, one of the following statements hold.
	\begin{enumerate}
		\item One has $\bar{\zeta}_l\cdot \sum_{j=1}^m\bar{\sC}_l^{j}=0$. In particular, for any $1\leq j\leq m$, we have $\bar{\zeta}_l\cdot \bar{\sC}_l^{ji}=0$. 
		
		\item One has $\bar{\zeta}_l\cdot \sum_{j=1}^m\bar{\sC}_l^{j}=1$. In particular, there exists a unique integer $1\leq j\leq m$ such that $\bar{\sC}_l^j$ is an irreducible curve satisfying $\bar{\zeta}_l\cdot \bar{\sC}_l^j=1$ and $\bar{\zeta}_l\cdot \bar{\sC}_l^{j'i}=0$ for any $1\leq j'\not=j$.	

	\end{enumerate}
    
    By Corollary \ref{cor:existence-mult-components}, the set $\Mult(f^*\sK)$ is not empty. Consider the irreducible component $\bar{\sK}_1 \subset \Mult(f^*\sK)$. Then, by Theorem \ref{Thm:Projection-VMRT}, the rational map $\bar{P}_1:\bar{\sC}_1\dashrightarrow \sZ_1$ is generically finite. In particular, $\bar{\sC}_l^1$ is not contracted by the projection
    \[
    \bar{P}_{\bar{\sC}^0_l}:\PP(\Omega_{\Univ_{\sK}^s|_{C_l^0}}) \dashrightarrow \PP^{n-2}.
    \] 
    Then Lemma \ref{Lemma:Projection-Local} implies $\bar{\zeta}_l \cdot \bar{\sC}^1_l > 0$. As a consequence, only the statement (2) above is possible. Then we must have $\bar{\sK}_j=\bar{\sK}_1$ and $\Mult(f^*\sK)=\bar{\sK}_1$. Moreover, Lemma \ref{Lemma:Projection-Local} also implies that $\bar{\sC}_l^{ji}$ is disjoint from $\bar{\sC}_l^0$ for each $j\geq 1$. 
    
    Finally, as $\bar{\zeta}_l\cdot \bar{\sC}_l^{ji}=0$ if $j\geq 2$, then Lemma \ref{Lemma:Projection-Local} says that $\bar{\sC}_{l}^{j}$ is contracted by $\bar{P}_{\bar{\sC}_l^0}$. In particular, the map $\bar{P}_j:\bar{\sC}_j\dashrightarrow \sZ_j$ is not generically finite if $j\geq 2$ (see also Remark \ref{Remark:Generical-finiteness}) and Theorem \ref{Thm:Projection-VMRT} implies that $\bar{\sK}_j$ is contained in $\Inf_{q_{\sK}^s}(f^*\sK)$ for $j\geq 2$.
\end{proof}

As $A\cdot l=1$ by Proposition \ref{Prop:Degree-VMRT}, the minimal family $\sK$ is  unsplit and the normalised universal family $\Univ_{\sK}\rightarrow \sK$ is a smooth $\PP^1$-bundle such that the $e_{\sK}^* A$-degree of the fibres of $q_{\sK}$ is $1$.
Set $V_\sK=(q_{\sK})_*\sO_{\Univ_{\sK}}(e_{\sK}^*A)$. Then $V_\sK$ is a vector bundle of rank $2$ over $\sK$ such that $\Univ_{\sK}\cong \PP(V_\sK)$. Thus we may set $\Univ_{\sK}^s=\PP(V)$ in the diagram \eqref{Eq:Universal-Family-K}, where $V=\mu^*V_\sK$. Then we have the following commutative diagram:

\begin{equation}
	\label{Eq:Diagram-Bigness}
	\begin{tikzcd}[column sep=large, row sep=large]
		\Univ_{\sK}^s=\PP(V) \arrow[d,"q_{\sK}^s"] \arrow[rr,bend left, "f"] \arrow[r,"\widetilde{\mu}"]  
		&  \Univ_{\sK}=\PP(V_\sK) \arrow[r,"e_{\sK}"] \arrow[d,"q_{\sK}"] 
		&  X     \\
		\sK^s \arrow[r,"\mu"]         &  \sK            &
	\end{tikzcd}
\end{equation}

\begin{remark}
	\label{remarkVample}
	Note that we have $V \cong (q_{\sK}^s)_*\sO_{\Univ_{\sK}^s}(f^*A)$, so the vector bundle
	$V$ is nef and big. Moreover, since $f$ is uniramified, hence finite, on $\Univ_{\sK_o}$, we know
	that if $B \subset \sK^s$ is an irreducible curve such that $B \cap \sK_o \neq \emptyset$,
	the restriction $V|_B$ is ample. Indeed for such a curve $B$, the morphism $\PP(V|_{B})\rightarrow X$ is finite onto its image, so the restriction of $f^* A$ to $\PP(V|_{B})$ is ample.
\end{remark}

Let $\sC_l \subset \PP(\Omega_X|_l)$ be the image of $f^{-1}(l)$ under the tangent map $\tau_{\sK}^s:\sU_{\sK}^s\dashrightarrow \PP(\Omega_X)$. Then $\sC_l$ coincides with $\sC\cap \PP(\Omega_X|_l)$. Recall that we have the following birational morphism
\[
\PP(\Omega_{\sU_{\sK}^s}|_{C_l^0})=\PP(f_l^*\Omega_X)\longrightarrow \PP(\Omega_X|_l),
\]
where $f_l:=f|_{C_l^0}:\PP^1=C_l^0\rightarrow l$ is the normalisation. Then $\bar{\sC}_l$ is nothing but the strict transform of $\sC_l$. Let us denote $\tau_{\sK}^s(C_l^i)$ by $\sC_l^i$. Then Proposition \ref{Prop:Degree-VMRT} can be rephrased as follows:

\begin{corollary}
	\label{Cor:Degree-VMRT}
	Notation and assumption as in Proposition \ref{Prop:Degree-VMRT}. There exists an integer $1\leq i(1)\leq r$ such that $C_l^{i(1)}$ is an irreducible component of $f^{-1}(l)$ satisfying $\zeta\cdot \sC_l^{i(1)}=1$ and  $\zeta\cdot \sC_l^i=0$ for any $1\leq i\not=i(1)$, where $\zeta$ is the tautological divisor class of $\PP(\Omega_X)$. 
	
	In particular, we have $\nu_1(\bar{\sC}_l^1)=\sC_l^{i(1)}$, where $\nu_1:\bar{\sC}_1\dashrightarrow \sC$ is the rational map defined in Notation \ref{notation-hjbar-bjbar}.
\end{corollary}

\subsubsection{Multiple components of inverse image}

We have seen in Proposition \ref{Prop:Dual-Defect-VMRT} that 
$\Mult(\bar \sK)=\Fin_{q_{\sK}^s}(\bar \sK)=\bar{\sK}_1$ is irreducible
and $\bar{\sC}_l^1$ is an irreducible curve. We now pursue the analysis
by studying the irreducible components of $\fibre{f}{l}$. Let $[l]\in \sK$ be a general member and 
$$
B_l^{\rm \small mult} := \cup B_l^i := q_{\sK}^s(C_l^i)
$$ 
where the union runs over the curves $C_l^i$ such that $[C_l^i]\in \bar{\sK}_1$, i.e., the map $C_l^i\rightarrow l$ is not birational. 

\begin{proposition}
	\label{Prop:Uniqueness-Multple-Component}
	Under the Assumption \ref{assumptionbig},  let $x\in X$ be a general point and let $[l]$, $[l']$ be two members in $\sK_x$.
	\begin{enumerate}
		\item The curve $C_l^{i(1)}$ (see Corollary \ref{Cor:Degree-VMRT} above ) is the unique irreducible component of $f^{-1}(l)$ which is not birational onto $l$. In particular, each irreducible component of $f^{-1}(l)$ is a rational curve with trivial normal bundle.
		
		\item The point $[l]$ is contained in $B_{l'}^{\rm \small mult}$ if and only if $[l']$ is contained in $B_{l}^{\rm \small mult}$.
	\end{enumerate}
\end{proposition}

\begin{proof}
	Denote by $B_{C_l^0}^1\subset \bar{\sK}_1$ the space of curves parametrised by $\bar{\sK}_1$ that meet $C_l^0$ (cf. Definition \ref{Defn:Spaces-of-lines}). Since $\bar{\sC}_l^1$ is irreducible by Proposition \ref{Prop:Degree-VMRT}(1), $B_{C_l^0}^1$ is also irreducible. Let $C_l^{i(1)}$ be the irreducible component of $f^{-1}(l)$ such that 
	\[
	\tau_{\sK}^s(C_l^{i(1)})=\nu_1(\bar{\sC}_l^1)=\sC_l^{i(1)},
	\]
	where $\nu_1:\bar{\sC}_1\dashrightarrow \sC$ is the rational map defined in Notation \ref{notation-hjbar-bjbar} (cf. Corollary \ref{Cor:Degree-VMRT}). Then we obtain 
	\[
	f_{\flat}^1(B_{C_l^0}^1)=B_{l}^{i(1)}=q_{\sK}^s(C_l^{i(1)}),
	\]
	where $f_{\flat}^1:\bar{\sK}_1\dashrightarrow \sK^s$ is the natural dominant rational map.  We note that all the arguments above still work if we replace $l$ by $l'$. 
	
	\bigskip
	
	\textbf{Claim 1.} \textit{$[l']$ is contained in $B_{l}^{\rm \small mult}$ if and only if $[l]$ is contained in $B_{l'}^{i(1)}=q_{\sK}^s(C_{l'}^{i(1)})$.}
	
	\bigskip
	
	\textit{Proof of Claim 1.} By definition, $[l']$ is contained in $B_l^{\rm \small mult}$ if and only if there exists an irreducible component $C_{l}^i$ of $f^{-1}(l)$ such that $C_l^i\rightarrow l$ is not birational and $C_l^i\cap C_{l'}^0\not=\emptyset$; that is, 
	\[
	[C_l^i]\in B_{C_{l'}^0}^1\subset\Mult(f^*\sK)=\bar{\sK}_1.
	\]
	This implies that $[l]$ is contained in $B_{l'}^{i(1)}=f_{\flat}^1(B_{C_{l'}^0}^1)$. Conversely, if $[l]$ is contained in $B_{l'}^{i(1)}$, then there exists an irreducible component $C_l^i$ of $f^{-1}(l)$ such that 
	\[
	[C_l^i]\in B_{C_{l'}^0}^1\subset \bar{\sK}_1=\Mult(f^*\sK).
	\]
	Then $C_l^i\rightarrow l$ is not birational and $C_l^i\cap C_l^0\not=\emptyset$; that is, $[l']\in B_l^{\rm \small mult}$. This finishes the proof of Claim 1.
	
	\bigskip
	
	\textbf{Claim 2.} \textit{The morphism $C_l^{i(1)}$$\rightarrow l$ is not birational.}
	
	\bigskip
	
	\textit{Proof of Claim 2.} Assume to the contrary that $C_l^{i(1)}\rightarrow l$ is birational. Then there exists a unique member $[l_{i(1)}]\in \sK_x$ such that $[l_{i(1)}]\in B_l^{i(1)}$. Then we have the following surjective map $\Phi_x:\sK_x\rightarrow \sK_x$ by sending $[l]$ to $[l_{i(1)}]$. As $\sK_x$ is a finite set, $\Phi_x$ is also injective. 
	
	Let $C_{l_{i(1)}}^i$ be an irreducible component of $f^{-1}(l_{i(1)})$ such that 
	\[
	[C_{l_{i(1)}}^i] \in B_{C_l^0}^1\quad {\rm and}\quad f_{\flat}^1([C_{l_{i(1)}}^{i}])=[l_{i(1)}].
	\]
	Since $C_{l_{i(1)}}^i\rightarrow l$ is not birational, there exists another member $[l]\not= [\bar{l}]\in \sK_x$ such that $C_{\bar{l}}^0\cap C_{l_{i(1)}}^i\not=\emptyset$. Note that here we use the fact the all the members in $\sK_x$ are smooth at $x$. 
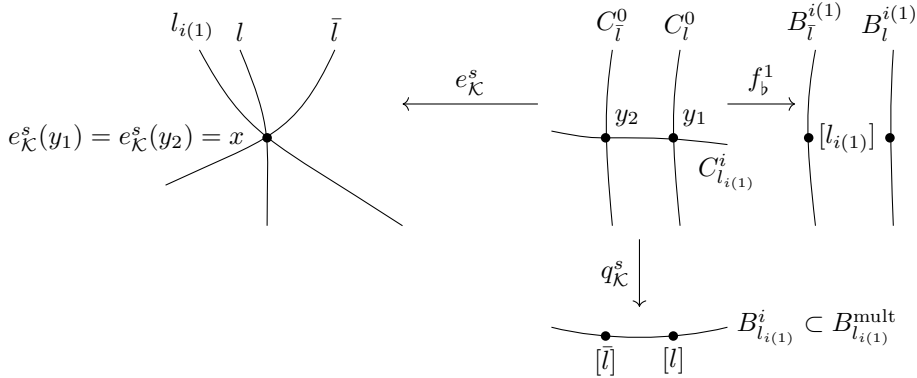
	\[
	\begin{tikzpicture}[scale=0.9]
	  \coordinate   (A) at (-4,0);
	  \fill[black]  (A) circle [radius=2pt];
	  \draw (-4.2,0) node[left] {$e_{\sK}^s(y_1)=e_{\sK}^s(y_2)=x$};
	  
	  \draw (-3,1.3) .. controls (-3.6,0) and (-4,0) .. (-5.5,-0.7);
	  \draw (-3,1.3) node[above] {$\bar{l}$};
	  
	  \draw (-4.4,1.3) .. controls (-3.9,0.1) and (-4,0) .. (-4,-1.3);
	  \draw (-4.4,1.3) node[above] {$l$};
	  
	  \draw (-5,1.3) .. controls (-4.3,0) and (-4,0) .. (-2,-1.3);
	  \draw (-5.1,1.3) node[above] {$l_{i(1)}$};
	  
	  \coordinate    (B) at (1,0);
	  \fill[black]   (B) circle [radius=2pt];
	  \draw (1,0.3) node[right] {$y_2$};
	  
	  \coordinate    (C) at (2,0);
	  \fill[black]   (C) circle [radius=2pt];
	  \draw (2,0.3) node[right] {$y_1$};
	  
	  \draw (0.2,0.1).. controls (0.7,0) and (0.8,0) .. (1,0);
	  \draw (1,0) .. controls (1.8,0) and (2,0) .. (2.8,-0.1);
	  \draw (2.8,-0.1) node[below] {$C_{l_{i(1)}}^i$};
	  
	  \draw (1.1,1.3) .. controls (0.98,0.5) and (0.98,0) .. (1.1,-1.3);
	  \draw (1.1,1.3) node[above] {$C_{\bar{l}}^0$};
	  
	  \draw (2.1,1.3) .. controls (1.97,0.5) and (1.97,0) .. (2.1,-1.3);
	  \draw (2.1,1.3) node[above] {$C_{l}^0$};
	  
	  \draw[->] (0,0.5) -- (-2,0.5);
	  \draw (-1,0.5) node[above] {$e_{\sK}^s$};
	  
	  \draw[->] (2.8,0.5) -- (3.8,0.5);
	  \draw (3.3,0.5) node[above] {$f_{\flat}^1$};
	  
	  \draw[->] (1.5,-1.5) -- (1.5,-2.5);
	  \draw (1.5,-2) node[left] {$q_{\sK}^s$};
	  
	  \draw (0.2,-2.8) .. controls (1,-3) and (2,-3) .. (2.8,-2.8);
	  \draw (2.8,-2.8) node[right] {$B_{l_{i(1)}}^i\subset B_{l_{i(1)}}^{\rm \small mult}$};
	  
	  \coordinate    (E) at (1,-2.93);
	  \fill[black]   (E) circle [radius=2pt];
	  \draw (1,-2.95) node[below] {$[\bar{l}]$};
	  
	  \coordinate    (G) at (2,-2.93);
	  \fill[black]   (G) circle [radius=2pt];
	  \draw (2,-2.95) node[below] {$[l]$};
	  
	  \coordinate    (D) at (4,0);
	  \fill[black]   (D) circle [radius=2pt];
	  \draw (4,0) node[right] {$[l_{i(1)}]$};
	  
	  \draw (4.1,1.3) .. controls (3.95,0.5) and (3.95,0) .. (4.1,-1.3);
	  \draw (4.1,1.3) node[above] {$B_{\bar{l}}^{i(1)}$};
	  
	  \coordinate    (F) at (5.2,0);
	  \fill[black]   (F) circle [radius=2pt];
	  %\draw (5,0) node[right] {$[l_{i(1)}]$};
	  
	  \draw (5.25,1.3) .. controls (5.2,0.5) and (5.2,0) .. (5.25,-1.3);
	  \draw (5.2,1.3) node[above] {$B_{l}^{i(1)}$};
      \end{tikzpicture}
	\]
	Then $[\bar{l}]$ is contained in $B_{l_{i(1)}}^{\rm \small mult}$ and hence $[l_{i(1)}]$ is contained in $B_{\bar{l}}^{i(1)}$ by Claim 1. This imlies
	\[
	\Phi_x([l])=\Phi_x([\bar{l}])=[l_{i(1)}],
	\]
	which is a contradiction. This completes the proof of Claim 2.
	
	\bigskip
	
	Let $C_l^i$ be an irreducible component of $f^{-1}(l)$ such that $C_l^i\rightarrow l$ is not birational. Then $[C_l^i]$ is contained in $\bar{\sK}_1$. On the other hand, as $[C_{l}^{i(1)}]\in \bar{\sK}_1$ by Claim 2 above, it follows that we have
	\[
	\zeta\cdot \sC_l^i = \zeta\cdot \tau_{\sK}^s(C_l^i) = \zeta\cdot \tau_{\sK}^s(C_{l}^{i(1)})=\zeta\cdot \sC_l^{i(1)}=1. 
	\]
	Then Corollary \ref{Cor:Degree-VMRT} implies that $\sC_l^i=\sC_l^{i(1)}$ and hence $C_l^i=C_{l}^{i(1)}$. This shows that $C_l^{i(1)}$ is the unique irreducible component of $f^{-1}(l)$ which is not birational onto $l$. In particular, $C_l^i$ is a rational curve if $i\not=i(1)$ as $l$ is a rational curve. On the other hand, note that $\sC_l^{i(1)}$ is a rational curve by Lemma \ref{Lemma:Projection-Local} and the tangent map $C_l^{i(1)}\dashrightarrow \sC_l^{i(1)}$ is birational, it follows that $C_l^{i(1)}$ is also a rational curve. The remaining part of statement (1) follows from Proposition \ref{Prop:Pull-Back-Trivial-Normal-Bundle}.
	
	To prove the statement (2), note that, by Claim 1, $[l']$ is contained in $B_{l}^{\rm \small mult}$ if and only if $[l]$ is contained in $B_{l'}^{i(1)}$. On the other hand, it follows from the statement (1) that $B_{l'}^{\rm \small mult}=B_{l'}^{i(1)}$. This finishes the proof of the statement (2).
\end{proof}

\begin{notation}
	\label{Notation:C_l^1}
	In the sequel of this paper, we shall always assume that $C_l^1$ is the unique irreducible component of $f^{-1}(l)$ such that $C_l^1\rightarrow l$ is not birational. In particular, $\nu_1(\bar{\sC}_l^1)=\sC_l^1$ and $B_l^{\rm \small mult}=B_l^1$ is irreducible.
\end{notation}

\begin{lemma}
	\label{Lemma:Birationality-Multiple-Components}
	Under the Assumption \ref{assumptionbig}, the morphism $C_l^1 \rightarrow B_l^1$ is birational. Moreover, let $n:\PP^1\rightarrow B_l^1$ be the normalisation. Then there exists a positive integer $a\geq d+1$ such that
	\[
	n^*V\cong \sO_{\PP^1}(a)\oplus \sO_{\PP^1}(d),
	\]
	where $d \geq 2$ is the degree of the finite morphism $C_l^1\rightarrow l$.
\end{lemma}

\begin{proof}
	Assume to the contrary that $C_l^1\rightarrow B_l^1$ is not birational. Let $t=[l']\in B_l^1$ be a general point. Then $C_l^1\cap C_{l'}^0$ contains two different points $y_1$ and $y_2$. Moreover, we may also assume that $C_l^1$ is smooth at $y_1$ and $y_2$ and $B_l^1$ is smooth at $t$. Then it follows that both $\PP(\Omega_{C_l^1,y_1})$ and $\PP(\Omega_{C_l^1,y_2})$ are contained in $\bar{\sC}_{l'}^1$ as $[C_l^1]\in \bar{\sK}_1$.
	
	
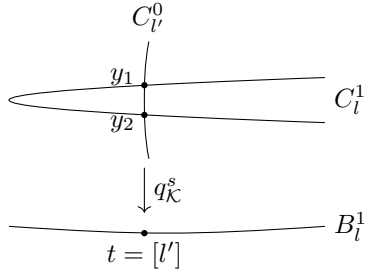
\[
	\begin{tikzpicture}[scale=0.6]
		\draw[rotate=90] (-0.5,-7) parabola bend (0,0) (0.5,-7);
		
		\coordinate   (A) at (3,0.33);
		\fill[black]  (A) circle [radius=2pt];
		\draw (3,0.52) node[left] {$y_1$};
		
		\coordinate   (B) at (3,-0.33);
		\fill[black]  (B) circle [radius=2pt];
		\draw (3,-0.52) node[left] {$y_2$};
		
		\draw (3.1,1.3) .. controls (3,0.8) and (3,0.33) .. (3,0);
		\draw (3,0) .. controls (3,-0.33) and (3,-0.8) .. (3.1,-1.3);
		\draw (3.1,1.3) node[above] {$C_{l'}^0$};
		
		\draw[->] (3,-1.5) -- (3,-2.5);
		\draw (3,-2) node[right] {$q_{\sK}^s$};
		
		\draw (7,0) node[right] {$C_l^1$};
		
		\draw (0,-2.8) .. controls (3,-3) and (4,-3) .. (7,-2.8);
		\draw (7,-2.8) node[right] {$B_l^1$};
		
		\coordinate   (C) at (3,-2.95);
		\fill[black]  (C) circle [radius=2pt];
		\draw (3,-2.95) node[below] {$t=[l']$};
	\end{tikzpicture}
    \]
    On the other hand, note that we also have
    \[
    \bar{P}_{\bar{\sC}_{l'}^0}(\PP(\Omega_{C_l^1,y_1}))=\bar{P}_{\bar{\sC}_{l'}^0}(\PP(\Omega_{C_l^1,y_2}))=\PP(\Omega_{B_l^1,t}).
    \]
    Yet, by Proposition \ref{Prop:Degree-VMRT} and Lemma \ref{Lemma:Projection-Local}, the projection $\bar{\sC}_{l'}^1\rightarrow \bar{P}_{\bar{\sC}_{l'}^0}(\bar{\sC}_{l'}^1)$ is an isomorphism, which is a contradiction. Hence, the map $C_l^1\rightarrow B_l^1$ is birational.
    
    Let $n':\PP^1\rightarrow C_l^1$ be the normalisation of $C_l^1$. Then it is clear that the composition $\PP^1\rightarrow C_l^1\rightarrow B_l^1$ is the normalisation of $B_l^1$. Moreover, as $f^*H\cdot C_l^1=d$, there exists a line bundle quotient
    \[
    n^*V\cong \sO_{\PP^1}(a)\oplus \sO_{\PP^1}(b) \rightarrow \sO_{\PP^1}(d)
    \]
    such that $a\geq b>0$ (see Remark \ref{remarkVample}) and the corresponding section $\sigma:\PP^1\rightarrow \PP(n^*V)$ satisfies the following commutative diagram
    \[
    \begin{tikzcd}[column sep=large, row sep=large]
    	\PP(n^*V) \arrow[r, "\bar{n}"] \arrow[d, "p"]
    	    &  \PP(V|_{B_l^1}) \arrow[d] \\
    	\PP^1 \arrow[r,"n"] \arrow[u, bend left, "\sigma"] 
    	    &  B_l^1
    \end{tikzcd}
    \]
    and $(\bar{n}\circ\sigma)(\PP^1)=C_l^1$. On the other hand, we consider the following generically surjective morphism
    \[
    (f\circ \bar{n})^*\Omega_X \longrightarrow \Omega_{\PP(n^*V)/\PP^1}.
    \]
    By the definition of the tangent map $\tau$ and the fact $\zeta\cdot \sC_l^1=1$, we have the following factorisation of the morphism above
    \[
    (f\circ \bar{n})^*\Omega_X|_{\sigma(\PP^1)} \twoheadrightarrow \sO_{\PP^1}(1) \rightarrow \Omega_{\PP(\bar{n}^*V)/\PP^1}|_{\sigma(\PP^1)}\cong \sO_{\PP^1}(-2d+a+b),
    \]
    where $\zeta$ is the tautological divisor of $\PP(\Omega_X)$ and the last isomorphism follows from the following facts
    \[
    \Omega_{\PP(\bar{n}^*V)/\PP^1}\cong \sO_{\PP(\bar{n}^*V)}(-2)\otimes p^*\sO_{\PP^1}(a+b)\ \ \text{and}\ \ \sO_{\PP(\bar{n}^*V)}(1)|_{\sigma(\PP^1)}\cong \sO_{\PP^1}(d).
    \]
    As a consequence, we get $a+b\geq 2d+1$. In particular, we must have $a\geq d+1$. On the other hand, since there exists a line bundle quotient $n^*V\rightarrow \sO_{\PP^1}(d)$, it follows that $b=d$.
\end{proof}

\subsubsection{Intersection of two minimal rational curves} 

This subsection is devoted to study the intersection of two minimal rational curves. The main goal is to prove the following result:

\begin{proposition}
	\label{Prop:Intersects-one-point}
	Under the Assumption \ref{assumptionbig}, let $x\in X$ be a general point and let $[l]$ be a member in $\sK_x$. Denote by $C_l^1$ the unique irreducible component of $f^{-1}(l)$ such that $C_l^1\rightarrow l$ is not birational. If $[l']$ is a member in $\sK_x$ such that $[l']\in B_l^1$, then $l\cap l'=\{x\}$. In particular, $C_{l}^0\cap C_{l'}^i=\emptyset$ if $i\not=1$.
\end{proposition}

Firstly we show that an analogue of Lemma \ref{Lemma:Birationality-Multiple-Components} still holds for $i\geq 2$.

\begin{proposition}
	\label{Prop:Universal-Family-Trivial-Degree}
	Under the Assumption \ref{assumptionbig}, let $l$ be a general curve parametrised by $\sK$. Then $C_{l}^i\rightarrow B_{l}^i$ is birational if $i\geq 2$ and we have 
	\[
	n^*V\cong \sO_{\PP^1}(1) \oplus \sO_{\PP^1}(1),
	\]
	where $n:\PP^1\rightarrow B_{l}^i$ is the normalisation.
\end{proposition}

\begin{proof}
	We will follow the commutative diagram \eqref{Eq:Diagram-Bigness}. Recall that by our assumption the morphism $C_l^i\rightarrow l$ is always birational for $i\geq 2$ (see Proposition \ref{Prop:Uniqueness-Multple-Component} and Notation \ref{Notation:C_l^1}).
	By Remark \ref{remarkVample} the vector bundle $n^*V$ is ample; so there exist positive integers $a\geq b\geq 1$ such that 
	\[
	n^*V\cong \sO_{\PP^1}(a)\oplus \sO_{\PP^1}(b).
	\]
	Let $j(i)\geq 2$ be the unique integer such that $[C_{l}^i]\in \bar{\sK}_{j(i)}$. Thanks to Proposition \ref{Prop:Degree-VMRT}, the projection $\bar{\sC}_{j(i)}\dashrightarrow \sZ_{j(i)}\subset \PP(\Omega_{\sK^s})$ is not generically finite. Thus Theorem \ref{Thm:Projection-VMRT} implies that $\bar{\sK}_{j(i)}$ is contained in $\Inf_{q_{\sK}^s}(f^*\sK)$. As a consequence, there exist infinitely many curves $C_{l_t}$ parametrised by $\bar{\sK}_{j(i)}$ such that 
	\[
	q_{\sK}^s(C_{l_t})=B_{l}^i,
	\]
	and $l_t=e_{\sK}^s(C_{l_t})$ is a member in $\sK$. Moreover, the morphism $C_{l_t}\rightarrow l_t$ is birational as $\Mult(f^*\sK)=\bar{\sK}_1$ and $j(i)\geq 2$. 
	
	Denote by $n':\PP^1\rightarrow C_{l}^i$ the normalisation. Then the composition $\PP^1\rightarrow C_l^i\rightarrow B_l^i \hookrightarrow \sK^s$ factors through $n$. In particular, we have
	\[
	n'^*V\cong \sO_{\PP^1}(da)\oplus \sO_{\PP^1}(db),
	\]
	where $d$ is the degree of $C_{l}^i\rightarrow B_{l}^i$. Moreover, as $f^*A\cdot C_{l}^i=1$, it follows that $n'^*V$ admits a line bundle quotient $n'^*V\rightarrow \sO_{\PP^1}(1)$. As a consequence, we obtain 
	\[
	d=b=1.
	\]
	In particular, we can identify $n'^*V$ with $n^*V$. Let $\widetilde{C}_{l_t}$ be the strict transform of $C_{l_t}$ in $\PP(n^*V)$. As $f^*A\cdot C_{l_t}=1$, the same argument shows that $\widetilde{C}_{l_t}$ is again a section of $\PP(n^*V)\rightarrow \PP^1$ such that
	\[
	c_1(\sO_{\PP(n^*V)}(1))\cdot \widetilde{C}_{l_t}=f^*A\cdot C_{l_t}=1.
	\]
	Consequently, $n^*V$ admits infinitely many line bundle quotients $n^*V\rightarrow \sO_{\PP^1}(1)$. Hence, we obtain $a=b=1$.
\end{proof}

Now we are in the position to prove Proposition \ref{Prop:Intersects-one-point}.

\begin{proof}[Proof of Proposition \ref{Prop:Intersects-one-point}]
	By Lemma \ref{Lemma:Birationality-Multiple-Components} and Proposition \ref{Prop:Universal-Family-Trivial-Degree}, for each irreducible component $C_{l'}^i$ of $f^{-1}(l')$ with $i\geq 1$, the morphism $C_{l'}^i\rightarrow B_{l'}^i$ is birational. Note that the morphism $f:\sU_{\sK}^s\rightarrow X$ is unramified along $C_{l}^0$. In particular, if $l\cap l'$ consists a point $x'\not=x$, then $f^{-1}(l')\cap C_{l}^0$ consists at least two points. However, as $x$ is general, the latter situation happens if and only if there exists an irreducible component $C_{l'}^i$ of $f^{-1}(l')$ such that $i\geq 2$ and 
	\[
	B_{l'}^i = q_{\sK}^s(C_{l'}^i) = q_{\sK}^s (C_{l'}^1) = B_{l'}^1.
	\]
	
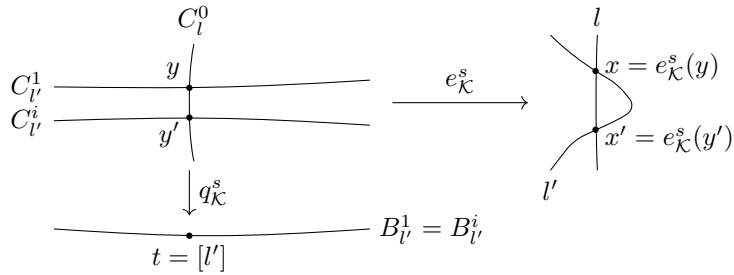
\[
	\begin{tikzpicture}[scale=0.6]
			
		\draw (0,0.35) .. controls (3,0.315) and (4,0.315) .. (7,0.5);
		\draw (0,0.35) node[left] {$C_{l'}^1$};
		
		\coordinate   (A) at (3,0.33);
		\fill[black]  (A) circle [radius=2pt];
		\draw (3,0.7) node[left] {$y$};
		
		\draw (0,-0.4) .. controls (3,-0.31) and (4,-0.31) .. (7,-0.5);
		\draw (0,-0.4) node[left] {$C_{l'}^i$};
		
		\coordinate   (B) at (3,-0.33);
		\fill[black]  (B) circle [radius=2pt];
		\draw (3,-0.7) node[left] {$y'$};
		
		\draw (3.1,1.3) .. controls (3,0.8) and (3,0.33) .. (3,0);
		\draw (3,0) .. controls (3,-0.33) and (3,-0.8) .. (3.1,-1.3);
		\draw (3.1,1.3) node[above] {$C_{l}^0$};
		
		\draw[->] (3,-1.5) -- (3,-2.5);
		\draw (3,-2) node[right] {$q_{\sK}^s$};
		
		\draw (0,-2.8) .. controls (3,-3) and (4,-3) .. (7,-2.8);
		\draw (7,-2.8) node[right] {$B_{l'}^1=B_{l'}^i$};
		
		\coordinate   (C) at (3,-2.95);
		\fill[black]  (C) circle [radius=2pt];
		\draw (3,-2.95) node[below] {$t=[l']$};
		
		\coordinate   (D) at (12,0.7);
		\fill[black]  (D) circle [radius=2pt];
		\draw (12,0.8) node[right] {$x=e_{\sK}^s(y)$};
		
		\coordinate   (D) at (12,-0.6);
		\fill[black]  (D) circle [radius=2pt];
		\draw (12,-0.8) node[right] {$x'=e_{\sK}^s(y')$};

		\draw (12.04,1.5) .. controls (12,0.7) and (12,-0.6) .. (12.04,-1.5);
		\draw (12.04,1.5) node[above] {$l$};
		
		\draw[->] (7.5,0) -- (10.5,0);
		\draw (9,0) node[above] {$e_{\sK}^s$};
		
		\draw [rounded corners] (11,-1.5) -- (11.5,-0.85) .. controls (12,-0.6) and (12.5,-0.4) .. (12.9,-0.1)-- (12.5,0.4) .. controls (12,0.7) and (11.5,1) .. (11,1.5);
		\draw (11,-1.5) node[below] {$l'$};
	\end{tikzpicture}
	\]
	Let $n:\PP^1\rightarrow B_{l'}^1=B_{l'}^i$ be the normalisation. Then by Lemma \ref{Lemma:Birationality-Multiple-Components} we have
	\[
	n^*V\cong \sO_{\PP^1}(a)\oplus \sO_{\PP^1}(d)
	\]
	for some integer $a\geq d+1$ and here $d\geq 2$ is the degree of $C_l^1\rightarrow l$. On the other hand, by Proposition \ref{Prop:Universal-Family-Trivial-Degree} we also have
	\[
	n^*V\cong \sO_{\PP^1}(1)\oplus \sO_{\PP^1}(1),
	\]
	which is absurd. Hence, the intersection $l\cap l'$ is a single point.
\end{proof}

\subsubsection{Conclusion}

Let us summarise what we have proved in this section in the following theorem. We still follow the notation as in Section \ref{Section:Etale-Webs}.

\begin{theorem}
	\label{Thm:Summary-VMRT}
	Let $X$ be a Fano manifold of dimension $n$ and with Picard number one, and  denote by $A$ the ample generator of $\pic(X)$.  Assume that the tangent bundle $T_X$ is big and $X$ admits an \'etale web of rational curves $\sK$.

 Let $l$ be a general curve parametrised by $\sK$. Then $A\cdot l=1$ and the following statements hold.
	\begin{enumerate}
		\item For each $j\geq 1$, the tangent variety $\bar{\sC}_l^j$ is disjoint from $\bar{\sC}_l^0$. Moreover, let $\bar{\zeta}$ be the tautological divisor class of $\PP(\Omega_{\sU_{\sK}^s})$, then the following statements hold.
		\begin{enumerate}
			\item The curve $\bar{\sC}_l^1$ is an irreducible rational curve such that $\bar{\zeta}\cdot \bar{\sC}_l^j=1$. Moreover, the morphism $\bar{\sC}_l^1\rightarrow \bar{P}_{\bar{\sC}_l^0}(\bar{\sC}_l^1)=\PP^1$ is an isomorphism. 
			
			\item For $j\geq 2$, each irreducible component $\bar{\sC}_{l}^{ji}$ of $\bar{\sC}_l^j$ is a rational curve such that $\bar{\zeta}\cdot \bar{\sC}_l^{ji}=0$. In particular, the image $\bar{P}_{\bar{\sC}_l^0}(\bar{\sC}_{l}^{ji})$ is a point. 
		\end{enumerate}
	    As a consequence, we have $[\check{\sC}]\equiv a\check{\zeta}-\check{\pi}^*A$ for some $a>0$, where $\check{\zeta}$ is the tautological divisor class of the projectivised tangent bundle $\check{\pi}:\PP(T_X)\rightarrow X$.
		
		\item Each irreducible component of $f^{-1}(l)$ is a rational curve with trivial normal bundle and there exists a unique irreducible component $C_l^1$ of $f^{-1}(l)$ such that $C_l^1\rightarrow l$ is not birational. Moreover, we have $\tau_{\sK}^s(C_l^1)=\nu_1(\bar{\sC}_l^1).$
		
		\item For a general member $[l']\in B_l^1$, the intersection $l\cap l'$ consists of a single point and $[l]$ is contained in $B_{l'}^1$.
	\end{enumerate}
\end{theorem}

Since the family of rational curves $\sK$ has degree one with respect to an ample line bundle, we will, by an abuse of terminology, call the members of this family ``lines'', there by distinguishing them from the members of the second family of rational curves that we will introduce in the next section. With this terminology in mind, we have the following:

\begin{corollary}
	\label{c.lines-B_l^1-properties}
	Let $X$ be a Fano manifold of dimension $n$ and with Picard number one, and  denote by $A$ the ample generator of $\pic(X)$.  Assume that the tangent bundle $T_X$ is big and $X$ admits an \'etale web of rational curves $\sK$.
	
 Let $l$ be a general line parametrised by $\sK$. Then the following statements hold.
	\begin{enumerate}
		\item For a general point $x\in l$, there exists at least two lines $l_i$ $(1\leq i\leq 2)$ parametrised by $B_l^1$ such that $x\in l_i$ for $i=1$ and $i=2$.
		
		\item For a general point $x\in l$, for two different lines $l_i$ $(1\leq i\leq 2)$ parametrised by $B_l^1$ such that $x\in l_i$ for $i=1$ and $i=2$, the linear span of $T_{l,x}$, $T_{l_1,x}$ and $T_{l_2,x}$  in $T_{X, x}$ is $3$-dimensional.
		
		\item For a general point $x\in l$, there exists a $3$-dimensional subspace $F_x$ of $T_{X,x}$ such that $T_{l,x}\subset F_x$ and $T_{\bar{l},x}\subset F_x$ for any line $\bar{l}$ parametrised by $B_{l}^1$ such that $x\in \bar{l}$.
	\end{enumerate}
\end{corollary}

The last statement is our first hint that the bigness of $T_X$ might lead to a restriction on $\dim X$. Indeed the number of lines $[\bar l] \in B^1_l$
such that $x \in \bar l$ might be very high, so it seems unlikely that they generate only a subspace of dimension three.

\begin{proof}[Proof of Corollary \ref{c.lines-B_l^1-properties}.]
	The statement (1) follows from the facts that $C_l^1\rightarrow l$ is not birational and $C_l^1\rightarrow B_l^1$ is birational. So for a general point $x\in l$, there exist at least two points in $B_l^1$ which parametrise a line passing through $x$.
	
	For  statement (2), it is enough to show that $T_{l,x}$, $T_{l_1,x}$ and $T_{l_2,x}$ are not contained in a subspace of $T_{X,x}$ of dimension $2$. Assume that this does not hold. Then the images of $T_{l_1,x}$ and $T_{l_2,x}$ in the quotient space $T_{X,x}/T_{l,x}$ coincide. In particular, the projection $\bar{\sC}_l^1\rightarrow \bar{P}_{\bar{\sC}_l^0}(\bar{\sC}_l^1)$ is not injective, which is a contradiction.
	
	For statement (3), recall that the image $\bar{P}_{\bar{\sC}_l^0}(\bar{\sC}_l^1)$ is a projective line. It follows that there exists a $2$-dimensional subspace $\bar{F}_x$ of the quotient $T_{X,x}/T_{l,x}$ such that the image of $T_{\bar{l},x}$ in the quotient $T_{X,x}/T_{l,x}$ is contained in $\bar{F}_x$ for any line $\bar{l}$ parametrised by $B_{l}^1$ such that $x\in l$. Then the preimage $F_x\subset T_{X,x}$ of $\bar{F}_x$ satisfies the desired properties.
\end{proof}

\section{Families of conics and their tangent varieties}

In this section we introduce a new tool for the study of manifolds with big tangent bundles: we consider the family of rational curves obtained by smoothing a pair of minimal rational curves.

\subsection{Family of conics}

In order to simplify the exposition, we will focus on the setting that is relevant for the proof of Theorem \ref{thm:big-finite-VMRT}.

\begin{definition}
    Let $X$ be a Fano manifold of dimension $n$ such that $-K_X = 2A$ with $A$ an ample Cartier divisor.
     A family of lines on $X$ is an irreducible component of $\Chow{X}$ such that the
     general point corresponds to an irreducible rational curve $C$ such that $A \cdot C=1$.
\end{definition}

Since we do not suppose that the polarisation $A$ is very ample, the terminology ``line'' is a somewhat abusive extension of \cite[Defn.3.7]{Hwa14}. Our hope is that the choice of terminology will improve the readability of the text.

\begin{remark} \label{remarkline}
Denote by $\sK \rightarrow \Chow{X}$ the normalisation of an irreducible component of $\Chow{X}$ parametrising lines.
    Since $A$ is ample, the family is unsplit: all the curves parametrised by $\sK$
    are given by the images of morphisms $f:\PP^1 \rightarrow X$ that are birational onto their image. 
    A well-known Riemann-Roch computation \cite[Ch.2.11]{Deb01} shows that that $\dim \mathcal \sK \geq n-1$ and equality holds if the family is covering \cite[Prop.4.9]{Deb01}. 
    
  If we consider lines passing through a general point $x \in X$, they belong to a minimal family of rational curves such that $-K_X \cdot l=2$. 
    Thus by \cite[Thm.3.3(2)]{Keb02a} all the lines passing through $X$ are immersed and smooth in the fixed point $x$ (but might be singular at some other point).
\end{remark}

\begin{definition}
    Let $X$ be a Fano manifold of dimension $n$ such that $-K_X = 2A$ with $A$ an ample Cartier divisor. 
    A family of conics on $X$ is an irreducible component $\Chow{X}$ such that the general point corresponds to an irreducible rational curve $C$ such that $A \cdot C=2$.
    
    We say that a family of conics is standard if for a general 
    member the normalisation $\holom{f}{\PP^1}{C \subset X}$ satisfies
    $$
    f^* T_X \simeq \sO_{\PP^1}(2) \oplus  \sO_{\PP^1}(1)^{\oplus 2} \oplus  \sO_{\PP^1}^{\oplus n-3}
    $$
\end{definition}

\begin{remark} \label{remarkconic}
If $C$ is a general member of a standard family of conics, the normalisation $\holom{f}{\PP^1}{C \subset X}$ is an embedding \cite[II,Thm.3.14.3]{Ko96}. Thus we will identify the curve $C$
with its normalisation $\PP^1$.

As in the case of lines we see that a family of conics has dimension at least $n+1$ and 
equality holds if the family is covering. 
\end{remark}

Let $T\subset \Chow{X}$ be a standard family of conics,
and let $\sH \rightarrow T$ be its normalisation. Denote by $\Univ_\sH$ the normalisation of the 
pull-back of the universal family over $T$, and as in Section \ref{Section:VOT}, denote by
$$
\holom{q_\sH}{\Univ_\sH}{\sH} \qquad
\holom{e_\sH}{\Univ_\sH}{X} 
$$
the natural morphisms. By construction the general fibre of $q_{\sH}$ is irreducible,
and, as a normalisation of a universal family, the morphism $q_\sH$ is equidimensional.
Thus the fibres of $q_\sH$ are connected and have pure dimension one,
in particular $\holom{q_\sH}{\Univ_\sH}{\sH}$ is a well-defined family of algebraic cycles in $X$
 (see \cite[Thm. 3.17]{Ko96}). 

For a general point $x\in X$, we denote by $\sH_x$ the fibre of the evalution map $e_\sH$.
Then $\sH_x$ is normal and the map $\sH_x \rightarrow q_\sH(\sH_x)$ is finite and birational onto its image, since a general member of the family is smooth.
Thus $\sH_x$ identifies to the normalisation of $q_\sH(\sH_x)$. For this reason, we call $\sH_x$ the \emph{normalised space of conics passing through $x$}.

Denote by $\sH_{r}$ and $\sH_{nr}$ the subset of points of $\sH$ parametrising the cycles of the form $l_1+l_2$ with $l_1\not=l_2$ and the cycles of the form $2l$, respectively. Then 
$$
\Delta := \sH_{r} \cup \sH_{nr}
$$ 
is a closed subvariety of $\sH$ that we call the discriminant locus. 
Recall that by \cite[II, Thm.2.8]{Ko96} the restriction of $q_\sH$ to
$$
\Univ_\sH \setminus \fibre{q_\sH}{\Delta} \rightarrow \sH \setminus \Delta
$$
is a $\PP^1$-bundle. We can say a bit more: there exist a smooth open subset $\sH^\dagger \subset \sH \setminus \sH_{nr}$
such that the complement has codimension at least two in $\sH \setminus \sH_{nr}$ and $q_\sH$ is flat over $\sH^\dagger$.
Thus the fibre of a point in $\sH^\dagger$ is a reduced Gorenstein curve of arithmetic genus zero. Since $e_\sH^* A$
is relatively ample of degree two, we see that the fibres are conics, and call $\sH^\dagger$ the conic bundle
locus of the universal family.

\begin{setup} 	\label{setupconics} {\rm
    Let $X$ be a Fano manifold of dimension $n$ such that $-K_X = 2A$ with $A$ an ample Cartier divisor. 	Assume that $X$ contains a covering family of lines $\sK \rightarrow \Chow{X}$
    such that through a general point $x \in X$ pass at least two members of $\sK$.
Let $x\in X$ be a general point and let $l_1:=l_{1,x}$ and $l_2:=l_{2,x}$ be two distinct lines passing through $x$. Then both $l_1$ and $l_2$ are standard rational curves. In particular, they are free and according to \cite[Prop. 4.24]{Deb01}, we can smooth the couple of lines $l_1+l_2$ to obtain a free rational curve, i.e. a covering family $\sH$ of conics such that $[l_1+l_2]\in \sH_x$ and a general member in $\sH_x$ is an irreducible free rational curve. Moreover, according to \cite[3.14.4]{Ko96}, a general member in $\sH_x$ is an immersed free rational curve.

{\em The family $\sH$ can depend on the choice of the couple $(l_1, l_2)$. When we specify such a choice, as in Subsection \ref{subsection-tgt-conics}, we choose the family of conics such that
the cycles $(l_{1,x} +l_{2,x})_{x \in X \mbox{\small general}}$ define an irreducible component
of the discriminant locus~$\Delta$.}

Let now $\Delta^j \subset \Delta$ be an irreducible component such that for a general point $x \in X$
the intersection $q_\sH(\sH_x) \cap \Delta^j$ is not empty and a general point 
of $\Delta^j$ parametrises a reducible cycle. Since covering families of lines have dimension $n-1$
and for a given line $l$ and a general point $x' \in l$ there exists another line $l_{x'}$ passing through 
$x'$, the variety $\Delta^j$ has dimension $n$.

Let $\Univ_{\Delta^j}\rightarrow \Delta^j$ be the universal family,
and let  $\holom{\nu}{\widetilde{\Univ}_{\Delta^j}}{\Univ_{\Delta^j}}$ be the normalisation.
Let $$
\tilde q_{\Delta_j}\colon \widetilde{\Univ}_{\Delta^j} \rightarrow \widetilde{\Delta}^j \qquad \mbox{and} \qquad g\colon \widetilde{\Delta}^j\rightarrow \Delta^j
$$ 
be the Stein factorisation of the morphism $q_\sH \circ \nu$. 
The general fibres of $\tilde q_{\Delta_j}$ are smooth rational curves and any fibre of $\tilde q_{\Delta_j}$ has degree one with respect to $\nu^*e_\sH^*A$. Hence, any fibre of $\tilde q_{\Delta_j}$ is irreducible and generically reduced. Thus by \cite[II, Thm.2.8]{Ko96}, the morphism $\tilde q_{\Delta_j}\colon \widetilde{\Univ}_{\Delta^j}\rightarrow \widetilde{\Delta}^j$ is a $\PP^1$-bundle. 

Denote by $\Chow{X}_{1}$ the subset of $\Chow{X}$ corresponding to lines on $X$. Then there exists a natural morphism
	\[
	\eta\colon \widetilde{\Delta}^j\rightarrow \Chow{X}_{1}
	\]
	such that $\widetilde{\Univ}_{\Delta^j}$ is the pull-back of the normalised universal family
	$\PP(V) \rightarrow \widetilde{\Chow{X}}_{1}$ over the normalisation of
	$\Chow{X}_{1}$ (cf. Subsection \ref{subsubsectionVMRTalong} for the notation $V$).
Since 	$\widetilde{\Delta}^j$ is obtained by Stein factorisation, its (at most two) irreducible components have dimension $n$. Since covering families of lines have dimension $n-1$, we see that the general fibre of $\eta$ has dimension one. 
}
\end{setup}

    Finally let us observe that the divisor $\Delta^j \subset \sH$ intersects the conic bundle
locus $\sH^\dagger \subset \sH$. Thus there exists an irreducible subvariety $N_j \subset \Univ_{\Delta^j}$
that surjects onto $\Delta^j$ and the intersection with a general fibre coincides with the singular point of the reducible conic. 

\begin{notation}
We denote by $N$ the union of all the possible irreducible closed subvarieties $N_j \subset \Univ_\sH$ obtained in this way.
\end{notation}

We start with an elementary property:

\begin{lemma}\label{lemmadeforminWx} 
In the situation of Setup \ref{setupconics}, let $x \in X$ be a general point and $\sH_x$ be the normalised space of conics passing through $x$. Let $y \in \sH_x$ be a point parametrising a cycle $l_1+l_2$ such that $l_1\not=l_2$ and $x \in l_i$ for $i=1$ and $i=2$. 

    Then for $i=1,2$ there exists an irreducible curve $\delta_i \subset \sH_x$ such that 
    $y \in \delta_i$ and a general point $y_t \in \delta_i$ corresponds to a cycle
    $l_i+l_t$ where $l_t$ depends on $t$, i.e. the map $\delta_i \rightarrow \Chow{X}, \ y_t \ \mapsto [l_t]$ is finite onto its image.
\end{lemma}

\begin{proof}
    Since $x$ is general, the cycle $l_1+l_2$ is contained in an irreducible component $\Delta^j \subset \Delta$ as in Setup \ref{setupconics}. Let us denote by $t_i \in \tilde \Delta^j$ a point such that $\eta(t_i)=[l_i]$
    and $g(t_i)=q_\sH(y)$. 
Choose an irreducible curve $\tilde \delta_i \subset \fibre{\eta}{\eta(t_i)}$ such that $t_i \in \tilde \delta_i$, and
   set $\delta'_i = g(\tilde \delta_i)$.
    In particular, $\delta'_i$ is contained in the image of $\sH_x\rightarrow \sH$ and $q_\sH(y) \in \delta'_i$. 
    Since the curves $l_i$ are smooth in $x$ (cf. Remark \ref{remarkline}), the point $y$
    is the unique preimage of $q_\sH(y)$ in $\sH_x$.
    Thus we can choose an irreducible curve $\delta_i\subset \sH_x$ passing through $y$ such that $q_{\sH}(\delta_i)=\delta'_i$. 
    The surface $\Univ_{\delta_i}:= \Univ_\sH \times_{\sH} \delta_i$ is reducible since it contains the image of the
    fibre product
    $$
    \PP(V) \times_{\mathcal H} \tilde \delta_i \simeq l_i \times \tilde \delta_i
    $$
    as an irreducible component. In particular a general point $y_t \in \delta_i$ parametrises a reducible cycle that contains the line $l_i$. The second component of the reducible cycle is a line that depends on the point $y_t$, since otherwise the map
    $\delta_i \rightarrow \sH \rightarrow \Chow{X}$ is not finite onto its image. Yet by construction $\sH$ is finite onto its image in $\Chow{X}$.
\end{proof}

The next step is to understand the tangent map $\tau$ introduced in Section \ref{Section:VOT}.
We start with an easy observation:

\begin{lemma} \label{lemmaindeterminacy}
In the situation of Setup \ref{setupconics}, let $x \in X$ be a general point and $\sH_x$ be the normalised space of conics passing through $x$.
Let $y \in \sH_x$ be a point parametrising a cycle $l_1+l_2$ such that $x \in l_i$ for $i=1$ and $i=2$.

  Then the points $[\PP(\Omega_{l_1,x})]$ and $[\PP(\Omega_{l_2,x})]$ are in the image of the tangent map $\tau_x$, but $\tau_x$ is not defined in $y$.      
  In particular the set $N \subset \Univ_\sH$ is contained in the indeterminacy locus of $\tau$.
 \end{lemma}
 
 \begin{remark} \label{remarkindeterminacy}
 For later use, let us observe that the proof will show more precisely that if $S \subset \sH_x$
 is the unique irreducible component containing the point $y$, then the points $[\PP(\Omega_{l_i,x})]$
 are in the closure of the image of $S$ under $\tau$.
 \end{remark}

\begin{proof}
We start by deducing the second statement: if $N$ is not contained in the indeterminacy locus, we can choose a general point $y \in N$ mapping onto a general point $x \in X$.
By the construction of $N$ (see Setup \ref{setupconics}) we know that $y \in \sH_x$ parametrises a cycle $l_1+l_2$ such that $x \in l_i$ for $i=1$ and $i=2$. If $\tau$ is well-defined in $y$, then its restriction $\tau_x$ is well-defined in $y$, a contradiction.

Let us now prove the first statement: by Lemma \ref{lemmadeforminWx} there exist
curves $\delta_i \subset \sH_x$ passing through $y$ such that a point $y_t \in \delta_i$
parametrises a cycle $l_i+l_{t}$ where $l_{t}$ depends on $t$. Since there are only finitely many lines passing through $x$, we see that for $y_t \in \delta_i$ general, the curve
$l_{t}$ does not pass through $x$. Thus the tangent map $\tau_x$ is well-defined in the generic point of $\delta_i$ and maps it onto the point $[\PP(\Omega_{l_i, x})]$. In particular the points $[\PP(\Omega_{l_1, x})]$
and $[\PP(\Omega_{l_2, x})]$ are in the image of $\tau_x$.

Arguing by contradiction assume that the tangent map $\tau_x$ is defined in $y$, so it is continuous.
Then by what precedes 
$$
[\PP(\Omega_{l_1,x})] = \tau_x(y) = [\PP(\Omega_{l_2,x})].
$$
Since the tangent map is birational for minimal rational curves passing through a general point \cite{HM04}, we obtain that $l_1=l_2$, a contradiction.
\end{proof}

We will now investigate the indeterminacies of $\tau$ in a general point of $N$:

\begin{lemma} \label{lemmaconnect}
In the situation of Setup \ref{setupconics}, let $x \in X$ be a general point and $\sH_x$ be the normalised space of conics passing through $x$.
Let $y \in \sH_x$ be a point parametrising a cycle $l_1+l_2$ such that $x \in l_i$ for $i=1$ and $i=2$.
 Denote by $S \subset \sH_x$ the unique component containing $y$, and by $\mathcal D_S \subset \PP(\Omega_{X,x})$ the closure of the image of $\tau_x|_S$.

Then there exists a connected, maybe reducible curve $C \subset \PP(\Omega_{X,x})$ of degree at most two such that 
$C \subset \mathcal D_S$ and $[\PP(\Omega_{l_i, x})] \in C^{\rm sm}$.
\end{lemma}

\begin{remark*} 
By  Lemma \ref{lemmaindeterminacy} we already know that the points $[\PP(\Omega_{l_i, x})]$ are in $\mathcal D_S$. The additional information is that these points can be connected by a line
or a (maybe reducible) conic contained in $\mathcal D_S$.
\end{remark*}

\begin{proof}
For the proof we will compute the sheaf $\sS$ appearing in the definition of the total variety of tangents in Section \ref{Section:VOT}. Let $\sH^\dagger \subset \sH$ be the conic bundle locus, and denote by
$\Univ_{\sH^\dagger} \rightarrow \sH^\dagger$ the restriction of the universal family.
Our goal is to compute the image of the morphism
$$
e_\sH^* \Omega_X|_{\Univ_{\sH^\dagger}} \rightarrow \Omega_{\Univ_{\sH^\dagger}/\sH^\dagger}.
$$
Set $N^\dagger := N \cap \Univ_{\sH^\dagger}$.
Since $\Univ_{\sH^\dagger} \rightarrow \sH^\dagger$ is a conic bundle with reduced fibres and singular locus $N^\dagger$, its dualising
sheaf $\omega_{\Univ_{\sH^\dagger}/\sH^\dagger}$ is invertible and
the cotangent sheaf
$\Omega_{\Univ_{\sH^\dagger}/\sH^\dagger}$ is locally free in the complement of $N^\dagger$. So we obtain a morphism
$$
\Omega_{\Univ_{\sH^\dagger}/\sH^\dagger} \rightarrow \omega_{\Univ_{\sH^\dagger}/\sH^\dagger}.
$$
which is an isomorphism outside $N^\dagger$.
On the one hand since all the curves parametrised by $\sH^\dagger$ are generically immersed into $X$ and general ones are immersed, the locus where the composition
$$
e_\sH^* \Omega_X|_{\Univ_{\sH^\dagger}} \rightarrow \Omega_{\Univ_{\sH^\dagger}/\sH^\dagger} \rightarrow \omega_{\Univ_{\sH^\dagger}/\sH^\dagger}
$$
is not surjective has codimension at least two in $\Univ_{\sH^\dagger}$. 
On the other hand we know
by Lemma \ref{lemmaindeterminacy}  that $N^\dagger$ is contained in the indeterminacy locus of $\tau$, so the composition is not surjective in a generic point of $N^\dagger$.
Consequently, we obtain a morphism
$$
e_\sH^* \Omega_X|_{\Univ_{\sH^\dagger}} \rightarrow \Omega_{\Univ_{\sH^\dagger}/\sH^\dagger} \rightarrow \sI_{N^\dagger} \otimes \omega_{\Univ_{\sH^\dagger}/\sH^\dagger}.
$$ 
Fix now $N_j \subset N^\dagger$ an irreducible component.
We claim that in the generic point of $N_j$, the map above is surjective.

{\em Assuming the claim, let us show how to conclude.}
Up to replacing $\sH^\dagger$ by a big open subset we can assume that 
we have a surjection $e_\sH^* \Omega_X|_{\Univ_{\sH^\dagger}} \rightarrow \sI_{N^\dagger} \otimes \omega_{\Univ_{\sH^\dagger}/\sH^\dagger}$.
By \cite[Prop.A.2.2.d]{Eis95} this induces a surjective map between the symmetric algebras:
$$
S^{\bullet} e_\sH^* \Omega_X|_{\Univ_{\sH^\dagger}} \rightarrow S^{\bullet} (\sI_{N^\dagger} \otimes \omega_{\Univ_{\sH^\dagger}/\sH^\dagger})
$$
Since $\sI_{N^\dagger}$ is an ideal sheaf, it is well-known \cite{Mic64} that
$$
S^{\bullet} (\sI_{N^\dagger})/\mbox{Tor} \simeq R(\sI_{N^\dagger}),
$$
where $R(\sI_{N^\dagger})$ is the Rees algebra of the ideal. Thus if $\Univ_{\sH^\dagger}' \rightarrow \Univ_{\sH^\dagger}$
is the blow-up of $\sI_{N^\dagger}$ we obtain an inclusion
$$
\Univ_{\sH^\dagger}'  \hookrightarrow \PP(e_\sH^* \Omega_X).
$$
In particular the fibres of the blowup are naturally embedded in the fibres of the projectivised bundle $\PP(e_\sH^* \Omega_X) \rightarrow \sH$.
Since the blowup of a variety is irreducible \cite[II,Prop.7.16]{Har77}, the variety $\Univ_{\sH^\dagger}'$
coincides with the unique irreducible component of $\PP(\sS|_{\Univ_{\sH^\dagger}})$ that dominates $\Univ_{\sH^\dagger}$. Thus the closure of its image is the total variety of tangents.

If $\Univ_{\sH}$ is generically smooth along $N_j$, the exceptional curve over a general point  $y \in N_j$ is a line that connects the strict transforms of the curves $\delta_1, \delta_2$
passing through $y$. By construction these strict transforms intersect
$\PP(e_\sH^*(\Omega_{X})_y) = \PP(\Omega_{X,x})$, and hence the exceptional line, exactly in the points $\PP(\Omega_{l_i,x})$.

If $Z$ is singular along $N_j$, it has compound $A_k$-singularities in a general point $y \in N$ \cite[Lemma 3.3]{DH17}. In this case the fibre of the blow-up is either an irreducible conic (for $cA_1$-singularities) or a chain of two lines (for $cA_k$-singularities with $k \geq 2$). The strict transforms of the curves $\delta_1$ and $\delta_2$ intersect the exceptional curve in its smooth locus \cite[Lemma 3.2]{DH17} and the claim follows as before.

{\em Proof of the claim.}
Let $[l_1+l_2]$ be a general point in $q_\sH(N_j)$.
The fibre
over the point $[l_1+l_2]$ is the union of two copies of $\PP^1$ intersecting exactly in the point $y = \fibre{\pi}{[l_1+l_2]} \cap N$.
Since $\omega_{\Univ_{\sH^\dagger}/\sH^\dagger}$ is locally free and $q_\sH$ is flat over the conic bundle locus,
the restriction of $\omega_{\Univ_{\sH^\dagger}/\sH^\dagger}$ to $\fibre{\pi}{[l_1+l_2]}$ coincides with $\omega_{\fibre{\pi}{[l_1+l_2]}}$.
The canonical sheaf of a reducible conic is well-known: if $\PP^1 \hookrightarrow l_1 \cup l_2$
is one of the irreducible components we 
have 
$$
\omega_{\Univ_{\sH^\dagger}/\sH^\dagger}  \otimes \sO_{\PP^1} \simeq \omega_{\fibre{\pi}{[l_1+l_2]}} \otimes \sO_{\PP^1} \simeq \sO_{\PP^1}(-1)
$$ 
and the canonical inclusion
$$
 \sO_{\PP^1}(-2) \simeq \omega_{\PP^1}  \rightarrow \omega_{\fibre{\pi}{[l_1+l_2]}} \otimes \sO_{\PP^1} \simeq \sO_{\PP^1}(-1)
$$
vanishes exactly on the intersection point. 
For a general point $[l_1+l_2] \in \Delta^j$ corresponding to $N_j$, the 
corresponding lines $\holom{f_i}{\PP^1}{X}$ are immersed, so the cotangent map
$f_i^* \Omega_X \rightarrow \Omega_{\PP^1} \simeq \omega_{\PP^1}$ is surjective.
Since the map
$$
e_\sH^* \Omega_X \otimes \sO_{\PP^1} \rightarrow  \omega_{\Univ_{\sH^\dagger}/\sH^\dagger}  \otimes \sO_{\PP^1}   
$$
factors through $f_i^*$, this shows the claim.
\end{proof}

\subsection{Tangent variety of conics} \label{subsection-tgt-conics}

We will now work under the following

\begin{assumption}{\rm \label{assumptionbig2}
Let $X$ be a Fano manifold of dimension $n\geq 3$ and Picard number $1$
such that $-K_X=2A$ where $A$ is the ample generator of the $\pic(X)$.  Assume that the tangent bundle $T_X$ is big, and there exists a covering family $\sK$ of minimal rational curves on $X$ with anti-canonical degree two.
}
\end{assumption}

Moreover, we shall also use the notations introduced in Section \ref{Section:Setup}, the commutative diagram \eqref{Eq:Diagram-Bigness} and Theorem \ref{Thm:Summary-VMRT}. 

Using the terminology of the preceding subsection, let $\sH$ be a covering family of conics on $X$ obtained by smoothing a couple of lines passing through a general point of $X$. Let $\Delta$ be the discriminant locus of $\sH$. As explained in Setup \ref{setupconics} we can assume that 
the following holds:

\textit{There exists an irreducible component $\Delta^j$ of $\Delta$ such that for a general point $x\in X$, there exists a point $y\in \Delta^j$ parametrising a cycle $l_1+l_2$ satisfying $l_1\not=l_2$, $x\in l_i$ for $i=1$ and $i=2$, and $[l_1]\in B_{l_2}^1$ (see Notation \ref{Notation:C_l^1}). We remark that by Theorem \ref{Thm:Summary-VMRT}, we also have $[l_2]\in B_{l_1}^1$.}

\begin{notation} \label{notationdelta}
We denote by $\Delta^\sharp$ the union of all the irreducible components $\Delta^j$ of $\Delta$ satisfying the assumption above. 
\end{notation}

By \eqref{Eq:Diagram-Bigness}, we have the following diagram 
\[
\begin{tikzcd}[column sep=large, row sep=large]
	M:=\sU_{\sK}^s=\PP(V) \arrow[r,"f"] \arrow[d,"q:=q_{\sK}^s"] 
	    & X \\
	\sK^s 
	    &  
\end{tikzcd}
\]

Let $\bar{\sH}=f^*\sH$ be the pull-back of $\sH$ by the generically finite morphism $f$ (see Definition \ref{Defn:pull-back-forward-web}). Given an irreducible component $\bar{\sH}_k$ of $\bar{\sH}$, let $\sU_{\bar{\sH}_k}\subset \bar{\sH}_k\times M$ be the universal family. Then we have the following morphism
\[
\Phi_k:=(id,f) : \bar{\sH}_k \times M \rightarrow \bar{\sH}_k\times X.
\]
Then $\Phi_k(\sU_{\bar{\sH}_k})\rightarrow \bar{\sH}_k$ is a well-defined family of $1$-cycles (cf. \cite[I, Thm. 3.17 and Thm. 6.8]{Ko96}). In particular, there exists a morphism $\bar{\sH}_k\rightarrow T\subset \Chow{X}$, where $T$ is the image of $\sH$ in $\Chow{X}$. As $\bar{\sH}_k$ is normal by construction and $\sH$ is the normalisation of $T$, we get a natural morphism $f_{\flat,c}^k:\bar{\sH}_k\rightarrow \sH$ (see also Definition \ref{Defn:pull-back-forward-web}). Denote by $\bar{\Delta}^{\sharp}_k$ the union of irreducible components of $(f_{\flat,c}^k)^{-1}(\Delta^{\sharp})$ which dominates an irreducible component of $\Delta^{\sharp}$. Let $\bar{\Delta}_k^{\sharp,j}$ be an irreducible component of $\bar{\Delta}_k^{\sharp}$ and let $\sU_{\bar{\Delta}^{\sharp,j}_k}\rightarrow \bar{\Delta}^{\sharp,j}_k$ be the universal family. Take the normalisation $\widetilde{\sU}_{\bar{\Delta}^{\sharp,j}_k}$ of $\sU_{\bar{\Delta}^{\sharp,j}_k}$ and then take the Stein factorisation
\[
\widetilde{\sU}_{\bar{\Delta}^{\sharp,j}_k} \rightarrow \widetilde{\Delta}^{\sharp,j}_k \rightarrow \bar{\Delta}^{\sharp,j}_k.
\]
Then it is clear that a general fibre of $\widetilde{\sU}_{\bar{\Delta}^{\sharp,j}_k}\rightarrow \widetilde{\Delta}^{\sharp,j}_k$ is a reduced and irreducible $1$-cycle parametrised by  $\bar{\sK}=f^*\sK$. In particular, we have a rational map $\widetilde{\Delta}^{\sharp,j}_k\dashrightarrow \bar{\sK}$, which maps each irreducible component of $\widetilde{\Delta}^{\sharp,j}_k$ onto an irreducible component of $\bar{\sK}$.

\begin{notation}
	\label{d.H^v}
	We denote by  $\bar{\sH}^v$ the union of the irreducible components $\bar{\sH}_k$ of $\bar{\sH}$ such that $\bar{\sK}_0$ is contained in the image of the rational map $\widetilde{\Delta}^{\sharp,j}_k\dashrightarrow \bar{\sK}$ for some $j$ \footnote{Recall that $\bar{\sK}_0$ is the irreducible component of $\bar{\sK}$ parametrising the fibres of $q$.}.
\end{notation}

%For an irreducible component $\bar{\sH}_k$ of $\bar{\sH}^v$, we have the following simple observation.

%\begin{lemma}
	%Let $\bar{\sH}_k$ be an irreducible component of $\bar{\sH}^v$ and let $x\in X$ be a general point. Let $\bar{\Delta}^k_j$ be an irreducible component of $\bar{\Delta}^k$ satisfying the assumption of Notation \ref{d.H^v} and denote by $\Delta_j$ the its image in $\Delta$. Then for a general point $\bar{y}\in \Delta_j$ parametrising a cycle $l_1+l_2$ such that $x\in l_i$ for $i=1$ and $i=2$, then there exists a point $\bar{x}\in Y$ such that $f(\bar{x})=x$ and there exists a point $\bar{y}\in \bar{\Delta}^k_j$ parametrising a cycle $C_{\bar{y}}$ and satisfying $f_{\flat,c}^k(\bar{y})=y$ and $\bar{x}\in C_{l_1}^0\subset \supp(C_{\bar{y}})$ (exchanging the indexes of $l_1$ and $l_2$ if necessary).
%\end{lemma}

\subsubsection{Pull-back of reducible conics} 

In this subsection, we study the behavior of the pull-back of reducible conics by $f$. We know by Theorem \ref{Thm:Summary-VMRT}(2) that the preimage $\fibre{f}{l}$ of a general line $l$ can have three different types of irreducible components: the curve $C^0_l$ that belong to the universal family, the distinguished component $C^1_l$ that is not birational onto $l$, and possibly some other components $C^i_l$ with $i \geq 2$. Our first goal is to understand how these components fit together when looking at the preimage
$\fibre{f}{l_1 \cup l_2}$. 

\begin{proposition}
	\label{p.0-1} Under the Assumption \ref{assumptionbig2},
	let $\bar{\sH}_k$ be an irreducible component of $\bar{\sH}^v$. Let $x\in X$ be a general point and let $y\in \Delta^{\sharp}$ be a point parametrising a cycle $l_1+l_2$ such that $x\in l_i$ for $i=1$ and $i=2$. If $\bar{y}\in \bar{\Delta}^{\sharp}_{k}$ is a point parametrising a cycle $C_{\bar{y}}$ such that $f_{\flat,c}^k(\bar{y})=y$ and $\supp(C_{\bar{y}})$ contains the irreducible curve $C_{l_i}^0$ for $i=1$ or $2$, then $C_{l_{3-i}}^1$ is also contained in $\supp(C_{\bar{y}})$. 
\end{proposition}

\begin{proof}
	Without loss of generality, we may assume that $\supp(C_{\bar{y}})$ contains $C_{l_1}^0$. Let $\bar{x}\in C_{l_1}^0\subset M$ be the unique point such that $f(\bar{x})=x$. Denote by $\bar{\sH}_{k,\bar{x}}$ (resp. $\sH_x$) the normalised space of curves passing through $\bar{x}$ (resp. $x$). Then the induced morphism $\bar{\sH}_{k,\bar{x}}\rightarrow \sH_x$ is generically finite onto its image. Note that this map may be not surjective in general and its image is a union of irreducible components of $\sH_x$. By Lemma \ref{lemmadeforminWx}, for $i=1$ and $i=2$, there exists an irreducible curve $\delta_i\subset \sH_x$ such that $y\in \delta_i$ and a general point $y_t\in \delta_i$ corresponds to a cycle $l_i+l_t$ where $[l_t]\in B_{l_i}^1$ depends on $t$. As $f_{\flat,c}^k(\bar{y})=y$, there exists an irreducible curve $\bar{\delta}_i\subset \bar{\sH}_{k,\bar{x}}$ such that $\bar{y}\in \bar{\delta}_i$ and $f_{\flat,c}^k(\bar{\delta}_i)=\delta_i$ for $i=1$ and $i=2$.
	
	Let $\bar{y}_t\in \bar{\delta}_2$ be a general point parametrising a cycle $C_{\bar{y}_t}$. Then $y_t=f_{\flat,c}^k(\bar{y}_t)\in \delta_2$ parametrises a cycle $l_2+l_t$ such that $f(\supp(C_{\bar{y}_t}))=l_2\cup l_t$. Since $l_t$ is a general line parametrised by $B_{l_2}^1$, we have $x\not\in l_t$. Note that $l_2$ is smooth at $x$ and $\bar{x}\in C_{l_2}^1$ as $[l_1]\in B_{l_2}^1$. As a consequence, as $\bar{x}\in \supp(C_{\bar{y}_t})$, it follows that $\supp(C_{\bar{y_t}})$ contains the curve $C_{l_2}^1$. Hence, the support $\supp(C_{\bar{y}})$ contains the curve $C_{l_2}^1$ by specialisation.
\end{proof}

\begin{lemma}	
	\label{Lemma:Reducedness-pull-back-reducible-conics} 
	Under the Assumption \ref{assumptionbig2}, let $y\in \Delta^{\sharp}$ be a general point parametrising a cycle $l_1+l_2$ such that $x\in l_i$ for $i=1$ and $i=2$. Then we have the following equality as $1$-cycles
	\[
	\sum_{k=1}^s \sum_{\bar{y}\in \bar{\sH}_k^y} C_{\bar{y}} = \sum_{i=0}^r (C_{l_1}^i + C_{l_2}^i),
	\]
	where $\bar{\sH}_{k}^y= (f_{\flat,c}^{k})^{-1}(y)$. In particular, the cycle $C_{\bar{y}}$ is reduced for any point $\bar{y}\in \bar{\sH}_k^y$ and any $1\leq k\leq s$. Moreover, for any couple of points $\bar{y}_1\in \bar{\sH}_k^y$ and $\bar{y}_2\in \bar{\sH}_{k'}^{y}$ (maybe $k=k'$) the $1$-cycles $C_{\bar{y}_1}$ and $C_{\bar{y}_2}$ have no common components.
\end{lemma}

\begin{proof}
	Firstly we show that the support of $f^{-1}(l_1\cup l_2)$ is contained in the support of the $1$-cycle on the left hand of the equality.
	
	Let $\sU_{\sH}\subset \sH\times X$ be the universal family. As $\sH$ is normal and $\Delta^{\sharp}\subset \sH$ is a divisor, we may assume that the point $y$ is contained in the smooth locus of $\bar{\sH}$. Since the morphism $f:M:=\sU_{\sK}^s\rightarrow X$ is generically finite and both $M$ and $X$ are smooth, there exists a Zariski open subset $X^\circ\subset X$ such that $\codim(X\setminus X^{\circ})\geq 2$ and the induced morphism $M^\circ:=f^{-1}(X^\circ)\rightarrow X^\circ$ is \'etale. Denote by $\sH^\circ$ the largest Zariski open subset of the smooth locus of $\sH$ such that the cycles parametrised by $\bar{\sH}^\circ$ are contained in $X^{\circ}$. By \cite[II, Proposition 3.7]{Ko96}, the variety $\sH^{\circ}$ is not empty and we may also assume that $\bar{\sH}_k^y\subset \sH^{\circ}$. Consider the following morphism
	\[
	\Phi^{\circ}:=(id,f): \sH^{\circ}\times M^{\circ} \rightarrow \sH^{\circ}\times X^{\circ}.
	\]
	Then $\Phi^{\circ}$ is \'etale. Let $\sU^{\circ}_k\rightarrow \sH^{\circ}$ be an irreducible component of $(\Phi^{\circ})^{-1}(\sU_{\sH^{\circ}})$ and let $\widetilde{\sU}_{k}^{\circ}\rightarrow \widetilde{\sH}^{\circ}_k\rightarrow \sH^{\circ}$ be the Stein factorisation with $\widetilde{\sU}^{\circ}_k \rightarrow \sU^{\circ}_k$ normalisation. Then the fibration $\widetilde{\sU}_{k}^{\circ}\rightarrow \widetilde{\sH}^{\circ}_k$ is a well-defined family of algebraic cycles in $M^{\circ}$. In particular, there exists a natural morphism $\widetilde{\sH}_k^{\circ}\rightarrow \Chow{M}_1$ such that the Zariski closure of its image is an irreducible component of $\bar{T}$, where $\bar{T}$ is the image of $\bar{\sH}$ in $\Chow{M}_1$. On the other hand, by the definition of $\bar{\sH}$, each irreducible component of $\bar{T}$ is obtained in such a way. Let $\widetilde{\sH}_k^y\subset \widetilde{\sH}_k^{\circ}$ be the inverse image of $y\in \sH^{\circ}$ under the finite morphism $\widetilde{\sH}^{\circ}_k\rightarrow \sH^{\circ}$. Then it is clear from the construction above that the support of $f^{-1}(l_1+l_2)$ is contained in the following curve
	\[
	\supp\left(\sum_{k=1}^s \sum_{\bar{y}\in \bar{\sH}_k^y} C_{\bar{y}}\right) = \supp \left(\sum_{k=1}^s \sum_{\bar{y}\in\widetilde{\sH}^y_k} C_{\bar{y}}\right).
	\]
	
	Secondly, for each $1\leq k\leq s$, the induced morphism $\bar{\sH}_k^{\circ}:=(f_{\flat,c}^k)^{-1}(\sH^{\circ})\rightarrow \sH^{\circ}$ is a finite morphism. Moreover, after removing a subvariety of codimension two, we may assume that $\bar{\sH}_k^{\circ}$ is smooth. In particular, the restriction of $f_{\flat,c}^k$ to $\bar{\sH}_k^{\circ}$ is flat. Denote the degree of $f_{\flat,c}^k$ by $d_k$. Then we have $\sharp(\bar{\sH}_k^y)\leq d_k$ by the flatness. On the other hand, let $u\in \sH^{\circ}$ be a general point parametrising a conic. Then we obtain $\sharp(\bar{\sH}_k^u)=d_k$, where $\bar{\sH}_k^u=(f_{\flat,c}^k)^{-1}(u)$. In particular, we have
	\[
	\sum_{\bar{y}\in \bar{\sH}_k^y} f^*A\cdot C_{\bar{y}} = \sharp(\bar{\sH}_k^y) f^*A\cdot \bar{\sH}_k\leq d_k f^*A\cdot\bar{\sH}_k = \sum_{\bar{u}\in \bar{\sH}_k^u} f^*A\cdot C_{\bar{u}}.
	\]
	On the other hand, by the construction of $\bar{\sH}$, we have
	\[
	\sum_{k=1}^s \sum_{\bar{u}\in\bar{\sH}_k^{u}} C_{\bar{u}} = f^{-1}(C)_{\red}. 
	\]
	This yields
	\[
	2d=\sum_{k=1}^s \sum_{\bar{u}\in\bar{\sH}_k^{u}} f^*A\cdot C_{\bar{u}} \geq \sum_{k=1}^s \sum_{\bar{y}\in\bar{\sH}_k^{y}} f^*A\cdot C_{\bar{y}}\geq f^*A\cdot f^{-1}(l_1+l_2)_{\red}=2d,
	\]
	where $d$ is the degree of $f:M\rightarrow X$.
\end{proof}

\begin{corollary}
	\label{c.components-pull-back-conics}
	Under the Assumption \ref{assumptionbig2}, let $\bar{\sH}_k$ be an irreducible component of $\bar{\sH}$ such that $\bar{\sH}_k$ is not contained in $\bar{\sH}^v$. Let $y\in \Delta^{\sharp}$ be a general point parametrising a cycle $l_1+l_2$ such that $x\in l_i$ for $i=1$ and $i=2$. Let $\bar{y}\in \bar{\sH}_k$ be a point parametrising a cycle $C_{\bar{y}}$ such that $f_{\flat,c}^k(\bar{y})=y$. Then for $1\leq i\leq 2$ and $0\leq j\leq 1$, the curve $C_{l_i}^j$ is not contained in $\supp(C_{\bar{y}})$.
\end{corollary}

\begin{proof}
	Since $\bar{\sH}$ is not contained in $\bar{\sH}^v$, it follows from the definition of $\bar{\sH}^v$ that for $1\leq i\leq 2$, the curve $C_{l_i}^0$ is not contained in $\supp(C_{\bar{y}})$. Moreover, by Lemma \ref{Lemma:Reducedness-pull-back-reducible-conics}, there exist irreducible components $\bar{\sH}_{k_1}$ and $\bar{\sH}_{k_2}$ of $\bar{\sH}$ (maybe $k_1=k_2$) such that there exists a point $\bar{y}_i\in \bar{\sH}_{k_i}$ parametrising a cycle $C_{\bar{y}_i}$ such that $f_{\flat,c}^k(\bar{y}_i)=y_i$ and $C_{l_i}^0$ is contained in $\supp(C_{\bar{y}_i})$. Then Proposition \ref{p.0-1} shows that $C_{l_{3-i}}^1$ is also contained in $\supp(C_{\bar{y}_i})$. In particular, Lemma \ref{Lemma:Reducedness-pull-back-reducible-conics} applies to say that $C_{l_i}^1$ is not contained in $\supp(C_{\bar{y}})$ for $i=1$ and $i=2$.
\end{proof}

\subsubsection{Linearity of tangent variety of $\sH$}

From now on, we will investigate the tangent variety $\sD_x\subset\PP(\Omega_{X,x})$ of $\sH$ at a general point $x\in X$. The main efforts goes into showing that $\sD_x$ is a finite union of projective planes. We start with some simple but useful observations.

\begin{lemma}
	\label{l.inclusion-tangents-lines}
	Under the Assumption \ref{assumptionbig2}, let $x\in X$ be a general point. Let $y\in \Delta^{\sharp}$ be a point parametrising a cycle $l_1+l_2$ such that $x\in l_i$ for $i=1$ and $i=2$. Let $\sH'_x$ be the unique irreducible component of $\sH_x$ containing $y$ and denote by $\sD'_x$ the tangent variety of $\sH'_x$. Then for any line $\bar{l}$ parametrised by $\sK_x$ such that $[\bar{l}]\in B_{l_i}^1$ ($i=1$ or $2$), there exists a point $y_{\bar{l}}\in \sH'_x$ parametrising the cycle $l_i+\bar{l}$. In particular, the point $\PP(\Omega_{\bar{l},x})$ is contained in $\sD'_x$.
\end{lemma}

\begin{proof}
	Without loss of generality, we may assume that $i=1$. By Lemma \ref{lemmadeforminWx}, there exists an irreducible curve $\delta_1\subset \sH'_x$ such that a general point $y_t\in\delta_1$ parametrising a cycle $l_1+l_t$ with $[l_t]\in B_{l_1}^1$ depending on $t$. Then there exists a natural morphism $\delta_1\rightarrow \Chow{X}_1$ whose image is exactly the image of $B_{l_1}^1$ in $\Chow{X}_1$. In particular, for any line $\bar{l}$ parametrised by $\sK_x$ such that $[\bar{l}]\in B_{l}^1$, there exists a point $y_{\bar{l}}\in \delta_1$ parametrising the cycle $l_1+\bar{l}$. Then it follows from Lemma \ref{lemmaindeterminacy} that $\PP(\Omega_{\bar{l},x})$ is contained in $\sD'_x$.
\end{proof}

\begin{lemma} \label{lemmaconicsstandard} Under the Assumption \ref{assumptionbig2}, the covering family $\sH$ of conics are standard. In particular, the tangent variety $\sD_x$ is two dimensional.
\end{lemma}

\begin{proof}
	Assume to the contrary that $\sH$ is not standard. By Lemma \ref{Lemma:Standard}, the tangent variety $\sD_x$ is contained in a finite union of projective lines. Let $y\in \Delta^{\sharp}$ be a point parametrising a cycle $l_1+l_2$ such that $x\in l_1$ and $x\in l_2$. Let $\sH'_x$ be the unique irreducible component of $\sH_x$ containing $y$ and denote by $\sD'_x$ the tangent variety of $\sH'_x$. Then $\sD'_x\subset \PP(\Omega_{X,x})$ is contained in a projective line by our assumption. On the other hand, by Corollary \ref{c.lines-B_l^1-properties} (1), there exists a third line $l_3$ parametrised by $\sK_x$ such that $[l_3]\in B_{l_1}^1$. By Lemma \ref{l.inclusion-tangents-lines}, the points $\PP(\Omega_{l_i,x})$ $(1\leq i\leq 3)$ are contained in $\sD'_x$. In particular, the tangent directions $T_{l_i,x}$ $(1\leq i\leq 3)$ are co-plane, which contradicts Corollary \ref{c.lines-B_l^1-properties} (2).
\end{proof}

\begin{definition} \label{definitionuniversalvertex}
Let $M$ be a projective manifold, and let $P \rightarrow M$ be a projective bundle 
over $M$. Let $C \subset D \subset P$ be irreducible subvarieties such that $C$ is generically finite onto $M$. We say that $C$ is a universal vertex for $D$ if for 
a general point $m \in M$ and every irreducible component $D'_m \subset D_m$ there
exists a point $v \in C_m$ such that $D'_m \subset P_x$ is a cone with vertex $v$.
\end{definition}

\begin{lemma}
	\label{l.cone-vertex}
	Under the Assumption \ref{assumptionbig2}, the total VMRT $\sC$ is the universal vertex of $\sD$.
\end{lemma}

\begin{proof}
	As $\sH$ is standard by Lemma \ref{lemmaconicsstandard}, the total dual tangent variety $\check{\sD}\subset \PP(T_{X})$ of $\sH$ is dominated by curves with $\check{\zeta}$-degree $0$ (cf. Proposition \ref{Prop:Dual-Defect-VMRT}). In particular, by Theorem \ref{Thm:Summary-VMRT}(1), the variety $\check{\sD}$ is contained in $\check{\sC}$. Let $x\in X$ be a general point. Given an irreducible component $\check{\sD}'_x$ of $\check{\sD}_x$, there exists an irreducible component $\check{\sC}'_x$ such that $\check{\sD}'_x$ is contained in $\check{\sC}'_x$. Denote by $l$ the line passing through $x$ such that $\check{\sC}'_x$ is the dual hyperplane defined by $\PP(\Omega_{l,x})\in \PP(\Omega_{X,x})$. Then it follows from \cite[Thm.1.25]{Tev03} that the variety $\sD'_{x} \subset \PP(\Omega_{X,x})$ is a cone with vertex $\PP(\Omega_{l,x})$. 
\end{proof}

\begin{remark}
	Since the total VMRT $\sC$ is irreducible, thus for any line $[l]\in \sK_x$, there exists an irreducible component $\sD'_x$ of $\sD_x$  whose vertex is $\PP(\Omega_{l,x})$.
\end{remark}

\begin{theorem}
	\label{t.unionplanes} Under the Assumption \ref{assumptionbig2}, the tangent variety
	$\sD_x \subset \PP(\Omega_{X,x})$ at a general point $x \in X$ is a finite union of projective planes.
\end{theorem}

\begin{remark*}
	Let $Z\subset \PP^{d}$ be an irreducible cone of dimension two with vertex $v\in Z$. 
	Projecting from the vertex $v$ we obtain the following basic fact:
	The cone $Z$ is a projective plane if one of the following holds
	\begin{enumerate}
		\item there exists a projective line contained in $Z$ not passing through $v$;
		
		\item there exists an irreducible conic contained in $Z$ passing through $v$;
		
		\item there exists another point $v\not=v'\in Z$ such that $v'$ is also a vertex of $Z$.
	\end{enumerate}
\end{remark*}

The idea of the proof of Theorem \ref{t.unionplanes} is to combine the remark above with Lemma \ref{lemmaconnect} which assures the existence of curves of low degree in the irreducible component $\sD_x'$ of the tangent variety $\sD_x$. There are however complications due to the fact that we must ensure that these curves are not just the lines passing through the vertex of the cone $\sD_x'$. This issue is solved by working with the pull-back $f^* \sH$ of our family of conics
to the universal family of lines $M=\Univ_\sK^s$.

\begin{proof}[Proof of Theorem \ref{t.unionplanes}]
	The statement is trivial for $n=3$ by Lemma \ref{lemmaconicsstandard}. Thus we may assume that $n\geq 4$. As the total tangent variety $\sD$ is irreducible, every irreducible component of $\sD_x$ has the same degree. Thus it is enough to show that there exists an irreducible component of $\sD_x$ which is a projective plane. We will divide the proof into two different cases.
	
	\bigskip
	
	\textit{Case 1. There exists an irreducible component $\bar{\sH}_k$ of $\bar{\sH}^v$ with total tangent variety $\bar{\sD}_k$ such that $\bar{\sC}_0$ is the universal vertex of $\bar{\sD}_k$, where $\bar{\sC}_0$ is the total tangent variety of $\bar{\sK}_0$.}
	
	\bigskip
	
	Let $\bar{\Delta}^{\sharp,j}_k$ be an irreducible component of $\bar{\Delta}_k^{\sharp}$ satisfying the assumption in Notation \ref{d.H^v}. Let $\Delta^{\sharp,j}$ be the irreducible component of $\Delta^{\sharp}$ dominated by $\bar{\Delta}^{\sharp,j}_k$. Let $y\in \Delta^{\sharp,j}$ be a point parametrising a cycle $l_1+l_2$ such that $x\in l_i$ for $i=1$ and $i=2$. Let $\bar{y}\in \bar{\Delta}^{\sharp,j}_k$ be a point parametrising a cycle $C_{\bar{y}}$ such that $f_{\flat,c}^k(\bar{y})=y$. Then $\supp(C_{\bar{y}})$ contains either $C_{l_1}^0$ or $C_{l_2}^0$. Without loss of generality, we assume that $\supp(C_{\bar{y}})$ contains the curve $C_{l_1}^0$. 
	
	Let $\bar{x}\in C_{l_1}^0\subset M$ be the unique point on $C_{l_1}^0$ such that $f(\bar{x})=x$. Let $\sH'_x$ be the unique irreducible component of $\sH_x$ containing $y$ and let $\bar{\sH}'_{k,\bar{x}}$ be the unique irreducible component of $\bar{\sH}_{k,\bar{x}}$ containing $\bar{y}$. As $\sH_x$ is normal, the irreducible components of $\sH_x$ are disjoint. In particular, we have $f_{\flat,c}^k(\bar{\sH}'_{k,\bar{x}})=\sH'_{x}$ as $f_{\flat,c}^k(\bar{y})=y$. Denote by $\sD'_x$ and $\bar{\sD}'_{\bar{x}}$ the tangent variety of $\sH'_x$ and $\bar{\sH}'_{k,\bar{x}}$, respectively. As $x\in X$ is general, the morphism $f$ is unramified at $\bar{x}$. In particular, we can identify the two projective varieties $\sD'_x\subset \PP(\Omega_{X,x})$ with $\bar{\sD}'_{\bar{x}}\subset \PP(\Omega_{M,\bar{x}})$ via the natural isomorphism $df:T_{M,\bar{x}}\rightarrow f^*T_{X}|_{\bar{x}}$. 
	
	On the other hand, by our assumption, the point $\PP(\Omega_{C_{l_1}^0,\bar{x}})$ is the vertex of $\bar{\sD}'_{\bar{x}}$. Therefore, the point $\PP(\Omega_{l_1,x})$ is the vertex of $\sD'_x$. By Corollary \ref{c.lines-B_l^1-properties} and Lemma \ref{l.inclusion-tangents-lines}, there exists a third line $l_3$ parametrised by $\sK_x$ such that $[l_3]\in B_{l_1}^1$ and there exists a point $y_3\in \sH'_x$ parametrising the cycle $l_1+l_3$.
	
	Thanks to Lemma \ref{lemmaconnect}, for each couple $1\leq i<j\leq 3$, there exists a curve $C_{i,j}$ of degree at most two contained in $\sD'_x$ such that $\PP(\Omega_{l_i,x})\in C_{i,j}$ and $\PP(\Omega_{l_j,x})\in C_{i,j}$. In particular, the curve $C_{i,j}$ is either a projective line, or a reducible conic, or an irreducible conic. Moreover, since the type of the curve only depends on the singularities of $M$ along $N$ (cf. proof of Lemma \ref{lemmaconnect}), the type of $C_{i,j}$ does not depend on the the choice of the couple $(i,j)$.\\
	
	\textit{Case (1.a). The curve $C_{i,j}$ is a projective line.} By Corollary \ref{c.lines-B_l^1-properties} and Lemma \ref{l.inclusion-tangents-lines}, there exists a line $l_4$ (which is distinct from $l_1$ and $l_2$, but may coincide with $l_3$) parametrised by $\sK_x$ such that $[l_4]\in B_{l_2}^1$ and there exists a point $y_4\in \sH'_x$ parametrising the cycle $l_2+l_4$. Then Lemma \ref{lemmaconnect} implies that there exists a projective line $C_{2,4}$ contained in $\sD'_x$ such that $\PP(\Omega_{l_2,x})\in C_{2,4}$ and $\PP(\Omega_{l_4,x})\in C_{2,4}$. Moreover, by Corollary \ref{c.lines-B_l^1-properties}, the tangent directions $T_{l_1,x}$, $T_{l_2,x}$ and $T_{l_4,x}$ are not co-plane. Hence, the point $\PP(\Omega_{l_1,x})$ is not contained in $C_{2,4}$. Hence, the cone $\sD'_x$ is a projective plane.\\
	
	\textit{Case (1.b). The curve $C_{i,j}$ is a reducible conic.} By Lemma \ref{lemmaconnect}, the point $\PP(\Omega_{l_1,x})$ is contained in the smooth locus of $C_{1,2}$. In particular, there exists an irreducible component $C_{1,2}'$ of $C_{1,2}$ not passing through $\PP(\Omega_{l_1,x})$. As $C_{1,2}'$ is a projective line, we again obtain that $\sD'_x$ is a projective plane.\\
	
	\textit{Case (1.c). The curve $C_{i,j}$ is an irreducible conic.} Then $\sD'_x$ contains an irreducible conic $C_{1,2}$ passing through the vertex $\PP(\Omega_{l_1,x})$ and therefore $\sD'_x$ is a projective plane.
	
	\bigskip
	
	\textit{Case 2. For any irreducible component $\bar{\sH}_k$ of $\bar{\sH}^v$, the total tangent variety $\bar{\sC}_0$ of $\bar{\sK}_0$ is not the universal vertex of the total tangent variety $\bar{\sD}_k$ of $\bar{\sH}_k$.}
	
	\bigskip
	
	By Lemma \ref{l.cone-vertex} and the assumption above, there exists an irreducible component $\bar{\sH}_k$ of $\bar{\sH}$ not belonging to $\bar{\sH}^v$ such that $\bar{\sC}_0$ is the universal vertex of the total tangent variety $\bar{\sD}_k$ of $\bar{\sH}_k$. Let $y\in \Delta^{\sharp}$ be a point parametrising a cycle $l_1+l_2$ such that $x\in l_i$ for $i=1$ and $i=2$. Let $\bar{y}\in \bar{\Delta}_k^{\sharp}$ be a point parametrising a cycle $C_{\bar{y}}$ such that $f_{\flat,c}^k(\bar{y})=y$. By Lemma \ref{Lemma:Reducedness-pull-back-reducible-conics} and Corollary \ref{c.components-pull-back-conics}, there exist two finite sets $J_i$ of $\{2,\cdots,r\}$ such that
	\[
	C_{\bar{y}} = \sum_{j\in J_1} C_{l_1}^j + \sum_{j\in J_2} C_{l_2}^j.
	\]
	Let $M_o\subset M$ be the largest Zariski open subset of $M$ such that $f$ is unramified over $M_o$. Denote by $X_o$ the image $f(M_o)$ in $X$.\\
	
	\textbf{Claim.} \textit{Let $\bar{C}$ be a general curve parametrised by $\bar{\sH}_k$. Then $C$ is contained in $M_o$ and the induced morphism $\bar{C}\rightarrow f(\bar{C})$ is an isomorphism.}\\
	
	\textit{Proof of Claim.} For each $1\leq i\leq 2$, since $C_{l_i}^j\rightarrow l_i$ is birational for $j\geq 2$ (see Notation \ref{Notation:C_l^1}), by Proposition \ref{prop:existence-multiple-component}, the curve $C_{l_i}^j$ is contained in $M_o$ for $j\geq 2$. In particular, the curve $\supp(C_{\bar{y}})$ is contained in $M_o$ and so is $\bar{C}$ as $\bar{C}$ is general. On the other hand, note that $C=f(\bar{C})$ is also a general member in $\sH$, the curve $C=f(\bar{C})$ is actually a smooth rational curve as $\sH$ is standard (see Remark \ref{remarkconic}). In particular, the curve $\bar{C}$ is also smooth as $\bar{C}\subset M_o$. If $\bar{C}\rightarrow C$ is not birational, then $\bar{C}\rightarrow C$ is ramified at a point $\bar{c}\in \bar{C}$. As a consequence, the morphism $f$ is ramified at $\bar{c}\in M_o$, which is a contradiction. Hence, the morphism $\bar{C}\rightarrow C$ is birational and hence an isomorphism.\\

	According to the Claim above, the degree of $\bar{\sH}_k$ with respect to $f^*A$ is $2$. Thus we must have $\sharp(J_1)=\sharp(J_2)=1$; that is, there exist two integers $2\leq j_1,j_2\leq r$ such that $C_{\bar{y}}=C_{l_1}^{j_1} + C_{l_2}^{j_2}$.As the curve $C_{l_i}^{j_i}$ is a rational curve with trivial normal bundle (see Proposition \ref{Prop:Pull-Back-Trivial-Normal-Bundle}), the covering family  $\bar{\sH}_k$ of curves is obtained by smoothing $C_{\bar{y}}$ and it is clear that the tangent variety of $\bar{\sH}_k$ at a general point of $M$ is two dimensional. Moreover, note that we have 
	\[
	-K_M\cdot \bar{C}=-f^*K_X\cdot \bar{C}=-K_X\cdot C=4.
	\] 
	Thus Lemma \ref{Lemma:Standard} applies to show that $\bar{\sH}_k$ is standard. Let $\check{\bar{\zeta}}$ be the tautological divisor of $\PP(T_M)$ and let $\check{\zeta}$ be the tautological divisor of $\PP(T_X)$. Then we must have 
	\[
	(\bar{f}^*\check{\zeta})|_{\PP(T_{M_o})} = \check{\bar{\zeta}}|_{\PP(T_{M_o})},
	\]
	where $\bar{f}:\PP(T_{M})\dashrightarrow \PP(T_X)$ is the induced rational map. Let $\check{\bar{\sC}}_0\subset \PP(T_{M})$ be the total dual VMRT of $\bar{\sK}_0$. Then by \cite[Corollary 2.13]{HLS20}, we get
	\[
	[\check{\sC}_0] \equiv \check{\bar{\zeta}} + \check{\bar{\pi}}^*K_{M/\sH^s},
	\]
	where $\check{\bar{\pi}}:\PP(T_M)\rightarrow M$ is the projectivisation and $K_{M/\sH^s}$ is the relative canonical divisor. On the one hand, note that we have $M=\PP(V)$ for a rank $2$ vector bundle $V$ over $\sH^s$ and $\sO_{\PP(V)}(1)\cong \sO_M(f^*A)$, we obtain $K_{M/\sH^s}=-2f^*A + q^*c_1(V)$, where $q:=q_{\sK}^s:M=\PP(V)\rightarrow \sH^s$. By the definition of $\bar{\sK}$ and Theorem \ref{Thm:Summary-VMRT}, we must have
	\[
	\sum_{i=0}^r \check{\bar{\sC}}_i|_{\PP(T_{M_o})} = \bar{f}^*\check{\sC}|_{\PP(T_{M_o})} \equiv \bar{f}^*(a\check{\zeta}-\check{\pi}^*A)|_{\PP(T_{M_o})},
	\]
	for $a=\deg(f)\geq 3$. As a consequence, we get
	\begin{align*}
		\sum_{i=1}^r [\check{\bar{\sC}}_i]|_{\PP(T_{M_o})} &  \equiv (a\check{\bar{\zeta}} - \check{\bar{\pi}}^*f^*A)|_{\PP(T_{M_o})} - [\check{\sC}_0]|_{\PP(T_{M_o})}  \\
		          &  \equiv ((a-1)\check{\bar{\zeta}} + \check{\bar{\pi}}^*f^*A - \check{\bar{\pi}}^*q^*c_1(V))|_{\PP(T_{M_o})}.
 	\end{align*}
    Let $\bar{C}$ be a general standard rational curve parametrised by $\bar{\sH}_k$. Then $\bar{C}$ is contained in $M_o$. Let $\bar{C}'$ be a minimal section over $\bar{C}$. Then $\bar{C}'$ is contained in $\PP(T_{M_o})$. In particular, we obtain
    \[
    \sum_{i=1}^r [\check{\bar{\sC}}_i] \cdot \bar{C}' = (f^*A-q^*c_1(V)) \cdot \bar{C} = 2 - q^*c_1(V) \cdot C_{\bar{y}}.
    \] 
    By Proposition \ref{Prop:Universal-Family-Trivial-Degree} we know that $q^*c_1(V)\cdot C_{l_i}^j=2$ for $j\geq 2$. Hence we get
    \[
    \sum_{i=1}^r [\check{\bar{\sC}}_i] \cdot \bar{C}'=-2<0.
    \]
    As a consequence, the total dual tangent variety $\check{\bar{\sD}}_k$ is contained in $\cup_{i=1}^r \check{\bar{\sC}}_i$. In particular, there exists an integer $1\leq i\leq r$ such that $\check{\bar{\sD}}_k$ is contained in $\check{\bar{\sC}}_i$. 
    
    Let $\bar{x}\in M$ be a point such that $f(\bar{x})=x$ and let $C_{l}^0$ be the fibre of $q$ passing through $\bar{x}$. By our assumption, the point $\PP(\Omega_{C_{l}^0,\bar{x}})$ is the vertex of all the irreducible components of $\bar{\sD}_{k,\bar{x}}$. On the other hand, as $\check{\bar{\sD}}_k$ is contained in $\check{\bar{\sC}}_i$, there exists a curve $C_{\bar{l}}^j$ $(j\geq 1)$ parametrised by $\bar{\sK}_i$ passing through $\bar{x}$ such that $f(C_{\bar{l}}^j)=\bar{l}\ni x$ and the point $\PP(\Omega_{C_{\bar{l}}^j,\bar{x}})$ is the vertex of an irreducible component $\bar{\sD}'_{k,\bar{x}}$ of $\bar{\sD}_{k,\bar{x}}$. Then it follows that the points $\PP(\Omega_{\bar{l},\bar{x}})$ and $\PP(\Omega_{l,x})$ are both vertices of $\sD'_x$, where $\sD'_x$ is the image of $\bar{\sD}'_{k,\bar{x}}$ in $\sD_x$. Hence, the variety $\sD'_x\subset \PP(\Omega_{X,x})$ is a projective plane.
\end{proof}

\subsubsection{Irreducibility of tangent variety of $\sH$}

In this subsection, we aim to show that the tangent variety $\sD_x\subset \PP(\Omega_{X,x})$ of $\sH$ at a general point $x\in X$ is actually irreducible. We consider the following commutative diagram
\[
\begin{tikzcd}[column sep=large, row sep=large]
	Y:=\sU_{\sH} \arrow[r,"e_{\sH}"]  \arrow[d,"q_{\sH}"]  
	    &  X  \\
	\sH
	    &
\end{tikzcd}
\]
Let $l$ be a general line parametrised by $\sK$. Denote by $Y_{l}$ the preimage $e_{\sH}^{-1}(l)$ and denote by $\Delta^{\sharp}_{l}$ the union of irreducible components of $Y_l\cap q_{\sH}^{-1}(\Delta^{\sharp})$ dominating $l$.

\begin{lemma}
	\label{l.existence-Delta^j_l}
	Under the Assumption \ref{assumptionbig2},
	let $Y_l^j$ be an irreducible component of $Y_l$. Then there exists an irreducible component $\Delta^{\sharp,j}_{l}$ of $\Delta^{\sharp}_{l}$ such that $\Delta^{\sharp,j}_{l}\subset Y^j_l$.
\end{lemma}

\begin{proof}
	Let $Y\rightarrow \widetilde{X}\rightarrow X$ be the Stein factorisation. By the definition of $\Delta^{\sharp}$ (see Notation \ref{notationdelta}),  the restricted morphism $q_{\sH}^{-1}(\Delta^{\sharp})\rightarrow X$ is surjective and so is the induced morphism $q_{\sH}^{-1}(\Delta^{\sharp})\rightarrow \widetilde{X}$. Note that the general fibre of $Y\rightarrow \widetilde{X}$ is irreducible. In particular, the variety $q_{\sH}^{-1}(\Delta^{\sharp})$ meets the general fibres of $Y\rightarrow \widetilde{X}$. Hence, each irreducible component of the fibre of $Y_l\rightarrow l$ over a general point $x\in l$ contains an irreducible component of the fibre of $\Delta_{l}^{\sharp}\rightarrow l$ over $x$ and we are done.
\end{proof}

Given a general point $x\in l$, the fibre of $Y_l\rightarrow l$ can be identified to the variety $\sH_x$ and hence it is of pure dimension two. Moreover, the fibre $\delta_x$ of $\Delta_{l}^{\sharp}\rightarrow l$ over $x$ is of pure dimension one as $\dim(\Delta^{\sharp})=n$. By identifying $\delta_x$ with its image in $\Delta^{\sharp}$, we can regard $\delta_x$ as the subvariety of $\Delta^{\sharp}$ parametrising reducible conics passing through $x$. In particular, let $\delta^i_x$ be an irreducible component of $\delta_x$. Then a general point $y_t\in \delta_x^i$ parametrising a cycle of the form $l_x^i+l_t$, where $l^i_x$ is a line parametrised by $\sK_x$ such that $[l_t]\in B_{l_x^i}^1$ depending on $t$. 

Let $\Delta^{\sharp,j}_l$ be an irreducible component of $\Delta_{l}^{\sharp}$ and denote by $\widetilde{\Delta}^{\sharp,j}_l$ its normalisation. Take the Stein factorisation $\widetilde{\Delta}^{\sharp,j}_l\rightarrow \widetilde{l}^j\rightarrow l$. Let $x'\in \widetilde{l}^j$ be a point mapped to $x$. Then the fibre of $\widetilde{\Delta}^{\sharp,j}_l\rightarrow \widetilde{l}^j$ over $x'$ can be identified to the normalisation of an irreducible component $\delta_x^i$ of the fibre $\Delta^{\sharp,j}_l\rightarrow l$ over $x$. In particular, the image of the induced natural morphism 
$$
\sigma^j_l:\widetilde{l}^j\rightarrow \Chow{X}_1
$$ 
defined by sending $x'$ to $[l_x^i]$ is either the image of an irreducible component $B_{l}^j$ of $B_l$ in $\Chow{X}_1$ or the single point $[l]$. On the other hand, in the former case, as $\widetilde{l}^j$ is normal and $B_l^j\rightarrow \Chow{X}_1$ is birational, the morphism $\sigma^j_l$ can be lifted to a morphism $\widetilde{l}^j\rightarrow B_{l}^j$. By abuse of notation, we shall again denote it by $\sigma_l^j$.

\begin{proposition}
	\label{p.fibres-irreducibility-components} Under the Assumption \ref{assumptionbig2},
	let $Y_l^j$ be an irreducible component of $Y_l$ and let $\sD_l^j\rightarrow l$ be the tangent variety of $Y_{l}^j$. Then the general fibre of $\sD^j_l\rightarrow l$ is irreducible.
\end{proposition}

\begin{remark*}
	We will use the following simple fact: two projective planes containing three non co-linear points  are the same.
\end{remark*}

\begin{proof}[Proof of Proposition \ref{p.fibres-irreducibility-components}]
	By Lemma \ref{l.existence-Delta^j_l}, there exists an irreducible component $\Delta^{\sharp,j}_l$ of $\Delta_{l}^{\sharp}$ such that $\Delta^{\sharp,j}_l\subset Y^j_l$. Let $x\in l$ be a general point. Denote by $\sD_{l,x}^{j}$ the fibre of $\sD_l^j\rightarrow l$ over $x$. Then $\sD_{l,x}^j$ is exactly the image of the fibre of $Y_{l}^j\rightarrow l$ over $x$ under the tangent map $\tau$. Let $\sH'_x$ be an arbitrary irreducible component of the fibre of $Y_l^j\rightarrow l$ over $x$ and denote by $\sD'_x$ the corresponding tangent variety. Then $\sD'_x$ is an irreducible component of $\sD_{l,x}^j$. By Lemma \ref{l.existence-Delta^j_l}, there exists an irreducible component $\Delta_l^{\sharp,j}$ of $\Delta_l^{\sharp}$ contained in $Y_l^j$. In particular, there exists an irreducible component $\delta'_x$ of $\Delta_l^{\sharp,j}\rightarrow l$ over $x$ such that $\delta_x'\subset \sH'_x$. Let $x'\in \widetilde{l}^j$ be the point corresponding to $\delta'_x$. Then the image $\sigma_l^j(x')$ is a point in $\sK_x$. We shall finish the proof by analysing the intersection $\sD'_x\cap \sC_x$ and the argument will divided into three different cases according to the types of the image of $\sigma_l^j$.\\
	
	\textit{Case 1. The image of the morphism $\sigma^j_l$ is the single point $[l]$.} Then we have $\sigma_l^{j}(x')=[l]$. In other words, a general point $y_t\in \delta'_x$ parametrises a cycle of the form $l+l_t$, where $l_t$ depends on $t$ such that $[l_t]\in B_{l}^1$. Thanks to Lemma \ref{l.inclusion-tangents-lines}, the variety $\sD'_x$ contains the points $\PP(\Omega_{\bar{l},x})$ for all lines $\bar{l}$ parametrised by $\sK_x$ such that $[\bar{l}]\in B_l^1$. On the other hand, by Corollary \ref{c.lines-B_l^1-properties}, there exist two lines $l_1$ and $l_2$ parametrised by $\sK_x$ such that $[l_1]\in B_{l}^1$, $[l_2]\in B_l^1$ and the tangent directions $T_{l,x}$, $T_{l_1,x}$ and $T_{l_2,x}$ are not co-plane. It follows that the three non co-linear points $\PP(\Omega_{l,x})$, $\PP(\Omega_{l_1,x})$ and $\PP(\Omega_{l_2,x})$ are contained in every irreducible component of $\sD^j_{l,x}$. As $\sD_{l,x}^j$ is a finite union of projective planes, one derive that $\sD_{l,x}^j$ is actually an irreducible projective plane. \\
	
	\textit{Case 2. The image of the morphism $\sigma_l^j$ is the irreducible component $B_l^1$.} Then we have $\sigma_l^j(\delta'_x)=[l_1]$ for some line $l_1$ parametrised by $\sK_x$ such that $[l_1]\in B_{l}^1$. In other words, a general point $y_t\in \delta'_x$ parametrising a cycle $l_1+l_t$, where $l_1$ is a line parametrised by $\sK_x$ such that $[l_1]\in B_{l}^1$ and $[l_t]\in B_{l_1}^1$ depending on $t$. Moreover, by Proposition \ref{prop:existence-multiple-component}, we have $[l]\in B_{l_1}^1$ and therefore there exists a point $y\in \delta'_x$ parametrising the cycle $l+l_1$ by Lemma \ref{l.inclusion-tangents-lines}. Then Lemma \ref{lemmadeforminWx} implies that there exists another irreducible curve $\delta_x^l\subset \sH'_x$ such that $y\in \delta_x^l$ and a general point $z_t\in \delta_x^l$ parametrises a cycle $l+l_t$, where $[l_t]\in B_{l}^1$ depending on $t$. Then the same argument as in Case 1 applies to show that there exists a third line $l_2$ parametrised by $\sK_x$ such that $[l_2]\in B_l^1$ and every irreducible components of $\sD_{l,x}^j$ contains the three non co-linear points $\PP(\Omega_{l,x})$, $\PP(\Omega_{l_1,x})$ and $\PP(\Omega_{l_2,x})$. As a consequence, the variety $\sD_{l,x}^j$ is an irreducible projective plane. \\
	
	\textit{Case 3. The image of the morphism $\sigma_l^j$ is an irreducible component $B_l^j$ of $B_l$ for some $j\geq 2$.} Then we have $\sigma_l^j(\delta'_x)=[l_1]$ for the unique\footnote{We recall that the morphism $C_{l}^j\rightarrow l$ and $C_l^j\rightarrow B_l^j$ are birational for $j\geq 2$. So the existence of $l_1$ is unique.}  line $l_1$ parametrised by $\sK_x$ such that $[l_1]\in B_{l}^j$ $(j\geq 2)$. In other words, a general point $y_t\in \delta'_x$ parametrises a cycle $l_1+l_t$, where $[l_t]\in B_{l_1}^1$ depending on $t$ and $l_1$ is the unique line parametrised by $\sK_x$ such that $[l_1]\in B_{l}^j$. Then as in Case 1 above, there exist two lines $l_1^1$ and $l_1^2$ parametrised by $\sK_x$ such that $[l_1^1]\in B_{l_1}^1$, $[l_1^2]\in B_{l_1}^1$ and every irreducible component of $\sD_{l,x}^j$ contains the three non co-linear points $\PP(\Omega_{l_1,x})$, $\PP(\Omega_{l_1^1,x})$ and $\PP(\Omega_{l_1^2,x})$. Hence, the variety $\sD_{l,x}^j$ is an irreducible projective plane.
\end{proof}

\begin{theorem}
	\label{thm:irreduciblity-tangent-H}
	Under the Assumption \ref{assumptionbig2}, the tangent variety $\sD_x$ of $\sH$ at a general point $x\in X$ is an irreducible projective plane.
\end{theorem}

\begin{proof}
	Let $l\subset X$ be a general line parametrised by $\sK$. Let $\sD\subset \PP(\Omega_X)$ be the total tangent variety of $\sH$. Take $\widetilde{\sD}$ the normalisation of $\sD$ and let 
	\[
	\widetilde{\sD}\xrightarrow{\widetilde{\pi}} \widetilde{X}\xrightarrow{\mu} X
	\]
	be the Stein factorisation. The general fibre of $\widetilde{\pi}$ is connected and normal, hence irreducible. Arguing by contradiction we assume that $\mu$ is not birational. By Corollary \ref{cor:existence-mult-components}, there exists an irreducible component $\widetilde{l}$ of $\mu^{-1}(l)$ such that $\widetilde{l}\rightarrow l$ is not birational. In particular, there exists an irreducible component $\sD^j_l$ of the tangent variety $\sD_l$ along $l$ such that the general fibre of $\sD^j_l\rightarrow l$ is not irreducible. Let $Y^j_l$ be the irreducible component of $Y_l$ such that the tangent variety of $Y_l^j$ is exactly $\sD_l^j$. Then we get a contradiction by Proposition \ref{p.fibres-irreducibility-components}.
\end{proof}

\begin{corollary}
	\label{c.dimension-spanning-VMRT}
	Let $X$ be an $n$-dimensional Fano manifold of Picard number $1$ equipped with an \'etale web $\sK$ of rational curves. If $T_X$ is big, then for a general point $x\in X$, the linear subspace of $\PP(\Omega_{X,x})$ spanned by the VMRT $\sC_x$ is two dimensional.
\end{corollary}

\begin{proof}
	Note that by Lemma \ref{lemmaindeterminacy} the total VMRT $\sC$ is contained in $\sD$ and by Theorem \ref{thm:irreduciblity-tangent-H} the tangent variety $\sD_x\subset \PP(\Omega_{X,x})$ at a general point is an irreducible plane. Moreover, by Corollary \ref{c.lines-B_l^1-properties}, the linear span of $\sC_x$ is at least two dimensional. Hence, the linear span of $\sC_x$ is exactly $\sD_x$, which is two dimensional.
\end{proof}

\section{Proofs of the main results}

Now we are in the position to finish the proof of the main theorem of this paper and deduce its corollaries.

\begin{proof}[Proof of Theorem \ref{thm:big-finite-VMRT}]
By assumption $X$ contains a rational curve $l'$ with trivial normal bundle,
we denote by $\sK$ the \'etale web of rational curves determined by $l'$.
Since $-K_X \cdot l'=2$, the family $\sK$ is a covering family of minimal rational curves. In particular, as $T_X$ is big, by Theorem \ref{Thm:Summary-VMRT}, the Fano manifold $X$ is of index two.	

Let $\sF\subset T_X$ be the saturated subsheaf defined by the linear span of the VMRT $\sC_x$. Then $\sF$ has rank $3$ by Corollary \ref{c.dimension-spanning-VMRT}. Let $l$ be a general line parametrised by $\sK$. As $\sF \subset T_X$ is saturated, we may assume by \cite[II, Prop.3.7]{Ko96} that $\sF$ is a subbundle of $T_X$ along $l$. Denote by $\sF_l$ the restriction of $\sF$ to $l$. Let $f:\PP^1\rightarrow l$ be the normalisation. By assumption, we have
	\[
	f^*T_{X} \cong \sO_{\PP^1}(2)\oplus \sO_{\PP^1}^{\oplus (n-1)}.
	\]
	On the other hand, it is clear that the factor $\sO_{\PP^1}(2)$ is contained in $\sF_l$ as $\PP(\Omega_l)\subset \sC$. Thanks to Theorem \ref{Thm:Summary-VMRT}, for a general point $x\in X$, the linear span of $\PP(\Omega_{l,x})$ and $\sC_{l,x}^1$ in $\PP(\Omega_{X,x})$ has dimension at least two. It follows that at a general point $x\in l$, the vector space $\sF_l$ is spanned by the tangent directions $T_{l,x}$ and all $T_{\bar{l},x}$, where $\bar{l}$ are the  lines parametrised by $\sK_x$ such that $\bar{l}\in B_l^1$. Consider the following projection  (see Section \ref{Section:projection-local})
	\[
	\bar{P}_{l_0}:\PP(f^*\Omega_X) \dashrightarrow \PP(\sO_{\PP^1}^{\oplus (n-1)}) \rightarrow \PP^{n-2}.
	\]
	Then the image $\bar{P}_{l_0}(\sC_l^1)$ is a projective line in $\PP^{n-2}$. This implies that there exists a trivial rank two vector bundle $\bar{\sF_l}\subset f^*T_X/\sO_{\PP^1}(2)$ such that the image of $\sF_l$ in $f^*T_X/\sO_{\PP^1}(2)$ is contained in $\bar{\sF_l}$. As $\sF_l$ is of rank three, it follows that $\sF_l$ is nothing but the inverse image of $\bar{\sF}_l$ in $f^*T_X$. In particular, we get $c_1(\sF_l)=2$. As a consequence, we obtain $c_1(\sF)\cdot l=2$. 
	
	Arguing by contradiction we assume that $n\geq 4$, so the inclusion $\sF \subset T_X$ is strict. Since $X$ has Picard number one and $-K_X\cdot l=2$, it follows that we have $-K_{\sF}=-K_X$. Thus the quotient sheaf $\mathcal Q:= T_X/\sF$ has trivial determinant. In particular the tangent bundle $T_X$ is not generically ample, which is a contradiction to \cite[Thm.1.3]{Pet12}. 
	
	Hence $X$ must be a del Pezzo threefold and we then conclude by \cite[Thm. 1.5]{HLS20} that $X$ is isomorphic to the quintic del Pezzo threefold $V_5$, which is a smooth codimension $3$ linear section of $\text{Gr}(2,5)\subset \PP^9$.
\end{proof}

\begin{proof}[Proof of Corollary \ref{c.threefolds}]
	Let $X$ be a smooth Fano threefold of Picard number one. If $X$ is of index at most two, then it is known that $X$ carries a covering family of rational curves with anti-canonical degree two. Thus Theorem \ref{thm:big-finite-VMRT} implies that $X$ is isomorphic to the quintic del Pezzo threefold $V_5$. If $X$ is of index at least three, then $X$ is isomorphic to either the smooth quadric threefold $Q^3$ or the projective space $\PP^3$, whose tangent bundles are known to be nef and big. 
\end{proof}

\begin{proof}[Proof of Corollary \ref{c.indexone}]
	Let $X$ be a Fano manifold of Picard number one and with index one. Since $T_X$ is big and the VMRT $\sC_x\subset \PP(\Omega_{X,x})$ is not dual defective, by \cite[Theorem 3.4]{FuLiu21}, the general members of $\sK$ have anti-canonical degree at most two. Hence, we must have $-K_X\cdot \sK=2$ and then Theorem \ref{Thm:Summary-VMRT} implies that $X$ is of index two, which is a contradiction.
\end{proof}

\begin{remark}
	Let $X$ be a Fano manifold of Picard number one and with index two. Assume that $X$ carries a covering family $\sK$ of minimal rational curves such that the VMRT $\sC_x\subset \PP(\Omega_{X,x})$ at a general point $x\in X$ is not dual defective. Then the above proof shows that $T_X$ is big only if either $X$ is isomorphic to the quintic del Pezzo threefold $V_5$ or the general members of $\sK$ have anti-canonical degree four.
\end{remark}

%\bibliographystyle{alpha}
%\bibliography{biblio.bib}

\begin{thebibliography}{MOSCW15}

\bibitem[AC12]{AC12}
Carolina Araujo and Ana-Maria Castravet.
\newblock Polarized minimal families of rational curves and higher {F}ano
  manifolds.
\newblock {\em Amer. J. Math.}, 134(1):87--107, 2012.

\bibitem[CMSB02]{CMS02}
Koji Cho, Yoichi Miyaoka, and N.~I. Shepherd-Barron.
\newblock Characterizations of projective space and applications to complex
  symplectic manifolds.
\newblock In {\em Higher dimensional birational geometry ({K}yoto, 1997)},
  volume~35 of {\em Adv. Stud. Pure Math.}, pages 1--88. Math. Soc. Japan,
  Tokyo, 2002.

\bibitem[CP91]{CP91}
Fr{\'e}d{\'e}ric Campana and Thomas Peternell.
\newblock Projective manifolds whose tangent bundles are numerically effective.
\newblock {\em Math. Ann.}, 289(1):169--187, 1991.

\bibitem[Deb01]{Deb01}
Olivier Debarre.
\newblock {\em Higher-dimensional algebraic geometry}.
\newblock Universitext. Springer-Verlag, New York, 2001.

\bibitem[DH17]{DH17}
Thomas Dedieu and Andreas H\"{o}ring.
\newblock Numerical characterisation of quadrics.
\newblock {\em Algebr. Geom.}, 4(1):120--135, 2017.

\bibitem[dJS07]{DJS07}
Aise~Johan de~Jong and Jason Starr.
\newblock Higher {F}ano manifolds and rational surfaces.
\newblock {\em Duke Math. J.}, 139(1):173--183, 2007.

\bibitem[DPS94]{DPS94}
Jean-Pierre Demailly, Thomas Peternell, and Michael Schneider.
\newblock Compact complex manifolds with numerically effective tangent bundles.
\newblock {\em J. Algebraic Geom.}, 3(2):295--345, 1994.

\bibitem[Eis95]{Eis95}
David Eisenbud.
\newblock {\em Commutative algebra}, volume 150 of {\em Graduate Texts in
  Mathematics}.
\newblock Springer-Verlag, New York, 1995.
\newblock With a view toward algebraic geometry.

\bibitem[FL21]{FuLiu21}
Baohua Fu and Jie Liu.
\newblock Normalised tangent bundle, varieties with small codegree and
  pseudoeffective threshold.
\newblock 2021.

\bibitem[Har77]{Har77}
Robin Hartshorne.
\newblock {\em Algebraic geometry}.
\newblock Springer-Verlag, New York, 1977.
\newblock Graduate Texts in Mathematics, No. 52.

\bibitem[HIM19]{HIM19}
Genki Hosono, Masataka Iwai, and Shin-ichi Matsumura.
\newblock On projective manifolds with pseudo-effective tangent bundles.
\newblock {\em J. Inst. Math. Jussieu}, to appear, 2019.

\bibitem[HLS20]{HLS20}
Andreas Höring, Jie Liu, and Feng Shao.
\newblock Examples of {F}ano manifolds with non-pseudoeffective tangent bundle,
  2020.

\bibitem[HM01]{HM01}
Jun-Muk Hwang and Ngaiming Mok.
\newblock Cartan-{F}ubini type extension of holomorphic maps for {F}ano
  manifolds of {P}icard number 1.
\newblock {\em J. Math. Pures Appl. (9)}, 80(6):563--575, 2001.

\bibitem[HM03]{HM03}
Jun-Muk Hwang and Ngaiming Mok.
\newblock Finite morphisms onto {F}ano manifolds of {P}icard number 1 which
  have rational curves with trivial normal bundles.
\newblock {\em J. Algebraic Geom.}, 12(4):627--651, 2003.

\bibitem[HM04]{HM04}
Jun-Muk Hwang and Ngaiming Mok.
\newblock Birationality of the tangent map for minimal rational curves.
\newblock {\em Asian J. Math.}, 8(1):51--63, 2004.

\bibitem[Hwa01]{Hwa01}
Jun-Muk Hwang.
\newblock Geometry of minimal rational curves on {F}ano manifolds.
\newblock In {\em School on {V}anishing {T}heorems and {E}ffective {R}esults in
  {A}lgebraic {G}eometry ({T}rieste, 2000)}, volume~6 of {\em ICTP Lect.
  Notes}, pages 335--393. Abdus Salam Int. Cent. Theoret. Phys., Trieste, 2001.

\bibitem[Hwa14]{Hwa14}
Jun-Muk Hwang.
\newblock Mori geometry meets {C}artan geometry: varieties of minimal rational
  tangents.
\newblock In {\em Proceedings of the {I}nternational {C}ongress of
  {M}athematicians---{S}eoul 2014. {V}ol. 1}, pages 369--394. Kyung Moon Sa,
  Seoul, 2014.

\bibitem[Hwa17]{Hwa17a}
Jun-Muk Hwang.
\newblock Geometry of webs of algebraic curves.
\newblock {\em Duke Math. J.}, 166(3):495--536, 2017.

\bibitem[IN03]{IN03}
Paltin Ionescu and Daniel Naie.
\newblock Rationality properties of manifolds containing quasi-lines.
\newblock {\em Internat. J. Math.}, 14(10):1053--1080, 2003.

\bibitem[Iwa18]{Iwa18}
Masataka Iwai.
\newblock Characterization of pseudo-effective vector bundles by singular
  hermitian metrics.
\newblock {\em Michigan Math. J.}, to appear, 2018.

\bibitem[Keb02a]{Keb02}
Stefan Kebekus.
\newblock Characterizing the projective space after {C}ho, {M}iyaoka and
  {S}hepherd-{B}arron.
\newblock In {\em Complex geometry ({G}\"ottingen, 2000)}, pages 147--155.
  Springer, Berlin, 2002.

\bibitem[Keb02b]{Keb02a}
Stefan Kebekus.
\newblock Families of singular rational curves.
\newblock {\em J. Algebraic Geom.}, 11(2):245--256, 2002.

\bibitem[KM98]{KM98}
J{\'a}nos Koll{\'a}r and Shigefumi Mori.
\newblock {\em Birational geometry of algebraic varieties}, volume 134 of {\em
  Cambridge Tracts in Mathematics}.
\newblock Cambridge University Press, Cambridge, 1998.
\newblock With the collaboration of C. H. Clemens and A. Corti.

\bibitem[Kol96]{Ko96}
J{\'a}nos Koll{\'a}r.
\newblock {\em Rational curves on algebraic varieties}, volume~32 of {\em
  Ergebnisse der Mathematik und ihrer Grenzgebiete. 3. Folge. A Series of
  Modern Surveys in Mathematics}.
\newblock Springer-Verlag, Berlin, 1996.

\bibitem[Laz04a]{Laz04a}
Robert Lazarsfeld.
\newblock {\em Positivity in algebraic geometry. {I}}, volume~48 of {\em
  Ergebnisse der Mathematik und ihrer Grenzgebiete.}
\newblock Springer-Verlag, Berlin, 2004.
\newblock Classical setting: line bundles and linear series.

\bibitem[Laz04b]{Laz04b}
Robert Lazarsfeld.
\newblock {\em Positivity in algebraic geometry. {II}}, volume~49 of {\em
  Ergebnisse der Mathematik und ihrer Grenzgebiete.}
\newblock Springer-Verlag, Berlin, 2004.
\newblock Positivity for vector bundles, and multiplier ideals.

\bibitem[Mat18]{Mat18}
Shin-Ichi Matsumura.
\newblock On projective manifolds with semi-positive holomorphic sectional
  curvature.
\newblock {\em Amer. J. Math.}, to appear, 2018.

\bibitem[Mic64]{Mic64}
Artibano Micali.
\newblock Sur les alg\`ebres universelles.
\newblock {\em Ann. Inst. Fourier (Grenoble)}, 14(fasc., fasc. 2):33--87, 1964.

\bibitem[Mok88]{Mok88}
Ngaiming Mok.
\newblock The uniformization theorem for compact {K}\"{a}hler manifolds of
  nonnegative holomorphic bisectional curvature.
\newblock {\em J. Differential Geom.}, 27(2):179--214, 1988.

\bibitem[Mor79]{Mor79}
Shigefumi Mori.
\newblock Projective manifolds with ample tangent bundles.
\newblock {\em Ann. of Math. (2)}, 110(3):593--606, 1979.

\bibitem[MOSC{\etalchar{+}}15]{MunozOcchettaSolaCondeWatanabeEtAl2015}
Roberto Mu{\~n}oz, Gianluca Occhetta, Luis~Eduardo Sol{\'a}~Conde, Kiwamu
  Watanabe, and Jaros{\l}aw~A. Wi{\'s}niewski.
\newblock A survey on the {C}ampana-{P}eternell conjecture.
\newblock {\em Rend. Istit. Mat. Univ. Trieste}, 47:127--185, 2015.

\bibitem[MOSCW15]{MOSW15}
Roberto Mu{\~n}oz, Gianluca Occhetta, Luis~Eduardo Sol\'{a}~Conde, and Kiwamu
  Watanabe.
\newblock Rational curves, {D}ynkin diagrams and {F}ano manifolds with nef
  tangent bundle.
\newblock {\em Math. Ann.}, 361(3-4):583--609, 2015.

\bibitem[Pet12]{Pet12}
Thomas Peternell.
\newblock Varieties with generically nef tangent bundles.
\newblock {\em J. Eur. Math. Soc. (JEMS)}, 14(2):571--603, 2012.

\bibitem[Tev03]{Tev03}
E.~A. Tevelev.
\newblock Projectively dual varieties.
\newblock volume 117, pages 4585--4732. 2003.
\newblock Algebraic geometry.

\end{thebibliography}

\newcommand{\etalchar}[1]{$^{#1}$}

\end{document}